\documentclass[10pt,a4paper,oneside,reqno]{amsart}

\usepackage{setspace}
\usepackage[totalwidth=14cm,totalheight=22cm]{geometry}
\usepackage[T1]{fontenc}
\usepackage[dvips]{graphicx}
\usepackage{epsfig}
\usepackage{amsmath,amssymb,amscd,amsthm}
\usepackage{subfigure}
\usepackage[latin1]{inputenc}
\usepackage{graphics}
\usepackage{colortbl}
\usepackage{color}
\usepackage{ae}
\usepackage{enumitem}
\usepackage[all]{xy}
\usepackage{mathrsfs}
\usepackage{mathtools}
\usepackage{stmaryrd}
\usepackage{bbm}
\usepackage{marginfix}
\usepackage{MnSymbol}
\usepackage[normalem]{ulem}

\usepackage{enumitem}

\usepackage[dvips]{graphicx}
\usepackage{color}
\usepackage[english]{babel}
\usepackage[latin1]{inputenc}
\usepackage{afterpage}

\numberwithin{equation}{section}

\newtheorem{cor}{Corollary}[subsection]

\newtheorem{theorem}[cor]{Theorem}

\newtheorem{prop}[cor]{Proposition}

\newtheorem{lemma}[cor]{Lemma}
\theoremstyle{definition}
\newtheorem{defi}[cor]{Definition}
\theoremstyle{remark}
\newtheorem{remark}[cor]{Remark}
\newtheorem*{remark*}{Remark}
\newtheorem{example}[cor]{Example}

\newcommand{\C}{{\mathbb C}}

\newcommand{\R}{{\mathbb R}}
\newcommand{\Sp}{{\mathbb S}}
\newcommand{\Q}[1]{\partial\mathbb{A}\mathrm{d}\mathbb{S}^{#1}}
\newcommand{\Quno}{\partial\mathbb{A}\mathrm{d}\mathbb{S}^{2,1}}
\newcommand{\LL}{{\mathbb L}}
\newcommand{\Z}{{\mathbb Z}}

\newcommand{\dev}{\mbox{dev}}

\newcommand{\Hyp}{\mathbb{H}}
\newcommand{\AdS}[1]{\mathbb H^{#1}}  
\newcommand{\AdSP}[1]{\mathbb{A}\mathrm{d}\mathbb{S}^{#1}}
\newcommand{\AdSU}[1]{\widetilde{\mathbb{A}\mathrm{d}\mathbb{S}}{}^{#1}}

\newcommand{\disk}{\mathbb{D}}

\newcommand{\psl}{\mathfrak{sl}}
\newcommand{\SL}{\mathrm{SL}}
\newcommand{\PSL}{\mathrm{PSL}}

\newcommand{\osc}{\mathrm{osc}}
\newcommand{\xmm}{\mathsf{x}}

\newcommand{\hol}{\rho}

\newcommand{\im}{\mathrm{Im}}
\newcommand{\Ker}{\mathrm{Ker}}

\newcommand{\tr}{\mbox{\rm tr}\,}

\newcommand{\isom}{\mathrm{Isom}}

\newcommand{\Ip}{\mathrm{I}^+}
\newcommand{\Ipm}{\mathrm{I}^-}
\newcommand{\I}{I}
\newcommand{\II}{I\hspace{-0.1cm}I}
\newcommand{\III}{I\hspace{-0.1cm}I\hspace{-0.1cm}I}

\newcommand{\id}{\mathrm{id}}
\newcommand{\SO}{\mathrm{SO}}
\newcommand{\OO}{\mathrm{O}}

\newcommand{\D}{\mathbb{D}}
\newcommand{\JJ}{\mathcal{J}\!}
\newcommand{\RP}{\R \mathrm{P}}
\newcommand{\PR}{\mathrm{P}}

\newcommand{\Isom}{\mathrm{Isom}}

\newcommand{\V}{\mathcal M(2,\R)}

\newcommand{\fut}{\mathrm{I}^+}

\newcommand{\past}{\mathrm{I}^-}

\newcommand{\dom}{\mathcal D}

\def\Teich{\mathcal{T}}
\def\MADS{\mathcal{MGH}}

\def\En{\mathbbm{1}}

\begin{document}

\title{Anti-de Sitter geometry and Teichm\"uller theory}

\author[Francesco Bonsante]{Francesco Bonsante}
\address{Francesco Bonsante: Dipartimento di Matematica ``Felice Casorati", Universit\`{a} degli Studi di Pavia, Via Ferrata 5, 27100, Pavia, Italy.} \email{bonfra07@unipv.it} 

\author[Andrea Seppi]{Andrea Seppi}
\address{Andrea Seppi: CNRS and Universit\'e Grenoble Alpes \newline 100 Rue des Math\'ematiques, 38610 Gi\`eres, France. } \email{andrea.seppi@univ-grenoble-alpes.fr}

\thanks{The authors are members of the national research group GNSAGA}

\begin{abstract}
The aim of these notes is to provide an introduction to Anti-de Sitter geometry, with special emphasis on dimension three and on the relations with Teichm\"uller theory, whose study has been initiated by the seminal paper of Geoffrey Mess in 1990. In the first part we give a broad introduction to Anti-de Sitter geometry in any dimension. The main results of Mess, including the classification of maximal globally hyperbolic Cauchy compact manifolds and the construction of the Gauss map, are treated in the second part. Finally, the third part contains related results which have been developed after the work of Mess, with the aim of giving an overview on the state-of-the-art. 
\end{abstract}


\maketitle

\section*{Introduction}

At the end of last century the interest around  Lorentzian geometry in low dimension, and in particular Lorentzian manifolds of constant {sectional} curvature, grew significatively. Among them, the most interesting ones are  those of constant \emph{negative} sectional curvature, which are called \emph{Anti-de Sitter} manifolds and have been largely studied until nowadays.

There were at least two different motivations behind this  increased interest for Lorentzian geometry of constant sectional curvature. The first motivation was the study of proper affine actions on $\R^n$. Affine actions which preserve the Euclidean structure of $\R^n$ are well-understood since the work of Bieberbach of 1912.
On the other hand the general case seems considerably more difficult and there are still important open questions in the area.
It was natural to consider proper actions which preserve the Minkowski structure as an intermediate problem, which  already contains some deep cases, like proper actions of free groups.
In particular in dimension three, the classification of free group actions was shown to be crucial towards a complete understanding of three-dimensional affine manifolds, see \cite[Theorem 2.1]{fried-goldman}.
This problem has been studied by several authors, see for instance \cite{drumm, drumm2, glm, cdg, cdg2}, and a complete  classification has been given only recently \cite{dgk, dgk2, choi-goldman}.
Similar problems have been studied in the more general setting of proper isometric actions on Lorentzian manifolds of constant {sectional} curvature \cite{kulkarni-raymond, salein2,salein,  dgk}. See \cite{ddgs_survey} for a recent and complete survey on these topics.

In a different direction, another motivation arose from the study of gravity in dimension three. {In mathematical physics, this consists in} the study of Lorentzian metrics on manifolds which obey to the so-called Einstein equation.
In dimension three the problem is considerably simpler, since solutions of Einstein equations are {precisely Lorentzian metrics} of constant {sectional} curvature (whose sign depends on the choice of {the cosmological} constant which appears in the Einstein equation).
The study of the space of constant {sectional} curvature metrics was therefore considered as the first step towards a quantization of the three-dimensional gravity, and as a toy model which could help in the understanding
of the {four-dimensional situation}. See for instance the inspiring work of Witten \cite{witten}. 
Unlike its Riemannian counterpart, this classification is expected to include Lorentzian metrics which are not geodesically complete, in light of the relevant notion of initial and final singularity.
A standard assumption  is to consider \emph{globally hyperbolic metrics}. Roughly speaking, these are metrics which admit foliations by Riemannian  hypersurfaces, 
recovering the idea of a space evolving in time. By a result of Choquet-Bruhat, any globally hyperbolic metric solving Einstein's equation can be isometrically embedded in a 
maximal one, see \cite{choquet-bruhat, choquet-bruhat2}, {which reduces} the problem to the classification of maximal globally hyperbolic Einstein spacetimes. 
In dimension three this problem was addressed by Andersson, Moncrief and others by means of analytic methods (see for instance \cite{moncrief, amt, andersson}).

The seminal work of Geoffrey Mess \cite{mess}, which originally appeared in 1990,   represented a very striking, and successful, attempt to link these two different areas. 
On the one hand Mess proved one of the main achievements in the classification of proper isometric actions of discrete groups on Minkowski space, showing that the action is necessarily by a free group.
On the other hand he gave a noteworthy classification of the moduli space of maximal globally hyperbolic spacetimes of constant {sectional} curvature. Mess' approach, unlike that of Andersson and Moncrief, was based on geometric constructions  inspired by the work of Thurston in the 80s. Indeed a remarkable aspect of his work is the link between three-dimensional gravity and hyperbolic geometry in dimension two, with particular regard to connections with Teichm\"uller theory.
While those connections were expected, and {partially} contained in {the previous work of other authors}, it is really {the paper of Mess} that deeply clarified the picture.

The work of Mess, now considered ``classical'', provided a new perspective for the study of Lorentzian geometric structures and Teichm\"uller space. It inspired many lines of investigation which have been developed until the very recent years and seem to be still very promising.

\subsection*{Scope and organization}
The purpose of this article is threefold. The first goal is providing an introduction to Anti-de Sitter geometry, first in any dimension and then specifically in dimension three, and this is the content of Part \ref{part1}. More concretely, in Chapter \ref{sec:prel} we provide some general preliminaries on Lorentzian geometry, with focus on Lorentzian manifolds of constant sectional curvature and maximal isometry group. This serves also as a motivation for the models of Anti-de Sitter space to be introduced later, by explaining in what sense they represent  the model spaces for constant negative curvature in the Lorentzian setting. In Chapter \ref{ch:models any dim} we introduce various models of Anti-de Sitter space in arbitrary dimension, and study their geometry and their properties. Chapter \ref{ch:ads dim 3} focuses on three-dimensional Anti-de Sitter geometry, by introducing the $\PSL(2,\R)$-model which is peculiar to this dimension. 

\vspace{0.3cm}
The second goal, achieved in Part \ref{part2}, is to provide a self-contained exposition of the results of Mess, published in \cite{mess}, which concern Anti-de Sitter three-dimensional geometry. These can be divided into two main directions: the classification of maximal globally hyperbolic Anti-de Sitter three-manifolds containing a closed Cauchy surface and the construction of the Gauss map. Chapter \ref{ch:causal convex} contains a number of preliminary results necessary to develop the theory, in particular about causal properties of Anti-de Sitter geometry and isometric actions, which constitute the fundamental setup for the proofs of Mess' classification results. In Chapter \ref{sec:GH AdS mfds} we then prove the classification result of maximal globally hyperbolic manifolds containing a Cauchy surface of genus $g$. For genus $g=1$, we describe the deformation space of these structures, which is essentially identified with the deformation space of semi-translation structures on the 2-torus. The situation is extremely more interesting in genus $g\geq 2$. Here the main classification result of Mess, whose proof is concluded in Theorem \ref{thm:classification rgeq2}, is that the deformation space of maximal globally hyperbolic manifolds is homeomorphic to the product of two copies of the Teichm\"uller space of the closed surface of genus $g$.  In Chapter \ref{ch:gauss} we discuss the construction of the Gauss map associated with spacelike surfaces in Anti-de Sitter space, an idea whose main application in the work of Mess is a proof of Thurston's Earthquake Theorem, using pleated surfaces. We will sketch Mess' proof of the Earthquake Theorem, again in an essentially self-contained fashion, and at the same time we develop further tools, for instance a differential geometric approach to the Gauss map for smooth spacelike surfaces, which have been proved useful in many applications. 

\vspace{0.3cm}
Indeed, in Part \ref{part3} we survey more recent results on Anti-de Sitter three-dimensional geometry, with special interest in the relations with Teichm\"uller theory, which somehow rely on the legacy of Mess' paper. In Chapter \ref{section:closed} we still focus on maximal globally hyperbolic manifolds with closed Cauchy surfaces. We give further results on their structure, for instance on foliations by surfaces with special properties of curvatures, and on the understanding of invariants such as the volume, in relation with their deformation space. As an outcome, we obtain applications in Teichm\"uller theory, and new parameterizations of the deformation space in terms of holomorphic objects. Finally, in Chapter \ref{section:non closed} we discuss the case of spacelike surfaces with a different topology. We explain a number of results which can be seen as the ``universal'' version of the analogue problems in the closed case, and derive applications for the theory of universal Teichm\"uller theory. As a conclusion we  mention very briefly Anti-de Sitter structures with timelike cone singularities (``particles'') and with multi-black holes, and how they are related to the Teichm\"uller theory of hyperbolic surfaces with cone points and with geodesic boundary respectively.

\vspace{0.3cm}
\subsection*{Other research directions}
As mentioned already at the beginning, the aim of this paper is not to provide a comprehensive treatment of Anti-de Sitter geometry, and we decided to focus on three-dimensional geometry, in the spirit of Mess, and to the relations with Teichm\"uller theory of hyperbolic surfaces. A variety of related topics are not included here, as a result of our necessity to make certain choices in the exposition, but would certainly deserve their own place. Among others, we would like to mention:

\begin{itemize}[leftmargin=0.5cm]
\item The study of properly discontinuous actions on Anti-de Sitter three-space, a natural problem to consider in light of the results we mentioned in this introduction about proper actions on affine space, for which much work towards a complete classification has been developed in recent times. See \cite{MR3323635,MR3465975,MR3533195,MR3544632,MR3608712,MR3750161,ddgs_survey}.
\item Higher-dimensional Anti-de Sitter geometry, in particular the study of globally hyperbolic manifolds: \cite{zbMATH01907769,MR2961282,MR3356068,MR3888623,zbMATH07076387,mts}
\item The relations of Anti-de Sitter geometry with other geometric structures, both in dimension three and in higher dimensions, for instance given by the Wick rotation \cite{MR2499272,MR2497778,MR3807587} and by geometric transition \cite{MR3190306,MR3286899,MR3465975,MR3522183,MR3814342,MR3965127,rioloseppi}.
\item The study of dynamical properties of isometric actions on Anti-de Sitter space, for instance in terms of Anosov representations, and the generalizations of these properties to other types of geometric structures. See for instance \cite{MR2961282,MR3356068,MR3608719,MR3749424,MR3966802,MR3966798,10.1093/imrn/rnz098,gmt}. It is also worth remarking here the recents works which highlighted the Higgs bundles perspective, see \cite{MR3731716,MR4017517,MR3948927}.
\end{itemize}

\vspace{0.3cm}
\subsection*{Acknowledgements}
We would like to thank Athanase Papadopoulos for the opportunity of writing this article, for his patience during the preparation of the manuscript, and for suggesting several improvements on the exposition. Moreover, we are very grateful to Thierry Barbot and François Fillastre for useful comments on a preliminary version of this work.

\tableofcontents

{\large {\part{Anti-de Sitter space}\label{part1}}}

\section{Preliminaries on Lorentzian geometry} \label{sec:prel}

The aim of this preliminary section is to briefly recall some basic facts about Lorentzian manifolds. We will introduce 
Lorentzian manifolds of constant sectional curvature and we will see that, as in the Riemannian case, two Lorentzian manifolds of constant sectional curvature $K$ are locally isometric.
In particular, we focus on those with maximal isometry group, as they provide models of manifolds of constant sectional curvature: if $M$ is a Lorentzian manifold with constant sectional curvature $K$ and maximal isometry group, 
then any Lorentzian manifold with constant sectional curvature $K$ carries a natural $(\mathrm{Isom}(M), M)$-atlas made of local isometries.
Simply connected space forms have maximal isometry group, but in general there are manifolds with maximal isometry group which are not simply connected.
In particular we will focus on the case $K=-1$ and in that case it will be convenient to use  models which are not simply connected.

\subsection{Basic definitions}

A \emph{Lorentzian metric} on a manifold of dimension $n+1$  is a non-degenerate symmetric 2-tensor $g$ of signature $(n,1)$. 
A \emph{Lorentzian manifold} is a connected manifold $M$ equipped with a Lorentzian metric $g$.

In a Lorentzian manifold $M$ we say that a non-zero vector $v\in TM$ is \emph{spacelike, lightlike, timelike} if $g(v,v)$ is respectively positive, zero or negative. More generally, we say that a linear subspace $V\subset T_x M$ is \emph{spacelike, lightlike, timelike} if the restriction of $g_x$ to $V$ is positive definite, degenerate or indefinite.

The set of lightlike vectors, together with the null vector, disconnects $T_xM$ into $3$ regions: two convex open cones formed by timelike vectors, one opposite to the other, and the region of spacelike vectors. As a consequence the set of timelike vectors in the total space $TM$ is either connected or is made by two connected components. 
In the latter case $M$ is said \emph{time-orientable}, and a time orientation is the choice of one of these components. Vectors in the chosen component are said \emph{future-directed},
vectors in the other component are said \emph{past-directed}.

A differentiable curve is \emph{spacelike, lightlike, timelike} if its tangent vector is {spacelike (resp. lightlike, timelike) at every point. It is \emph{causal} if the tangent vector is either timelike or lightlike. Given a point $x$ in a time-oriented Lorentzian manifold $M$, the \emph{future} of $x$ is the set $\fut{(x)}$ of points which are connected to $x$ by a future-directed causal curve. The \emph{past} of $x$, denoted $\past{(x)}$, is defined similarly, for past-directed causal curves.

An \emph{orthonormal basis} of $T_xM$ is a basis $v_1,\ldots v_{n+1}$ such that $|g(v_i, v_j)|=\delta_{ij}$, with $v_1,\ldots v_n$ spacelike, and  $v_{n+1}$ timelike.
If $v_1,\ldots v_{n+1}$ is an orthonormal basis then for $v=\sum_{i=1}^{n+1} x_iv_i$ we have $g(v,v)=\sum_{i=1}^n x_i^2-x_{n+1}^2$.

As in the Riemannian setting, on a Lorentzian manifold $M$ there is a unique linear connection $\nabla$ which is symmetric and compatible with the Lorentzian metric $g$.
We refer to it as the \emph{Levi-Civita connection} of $M$. The Levi-Civita connection determines the Riemann curvature tensor  defined by
$$R(u,v)w=\nabla_u \nabla_v w-\nabla_v \nabla_u w-\nabla_{[u,v]}w~.$$
We then say that a Lorentzian manifold $M$ has constant sectional curvature $K$ if
\begin{equation}\label{eq:constcurv}
     g(R(u,v)v,u)=K \left(g(u,u)g(v,v)-g(u,v)^2\right)
\end{equation}
for every pair of vectors $u,v\in T_xM$ and every $x\in M$.
This definition is strictly analogous to the definition given in the Riemannian realm. However in this setting the sectional curvature can be defined only for planes in 
$T_x M$ where $g$ is non-degenerate. 

Finally, we say that $M$ is \emph{geodesically complete} if every geodesic is defined for all times, or in other words, the exponential map is defined everywhere.

\subsection{Maximal isometry groups and geodesic completeness}

Two Lorentzian manifolds $M$ and $N$ of constant curvature $K$ are locally isometric, a fact which is well-known in the Riemannian setting. 
More precisely, the following holds:

\begin{lemma} \label{lemma:extension}
Let $M$ and $N$ be Lorentzian manifolds of constant curvature $K$. Then every linear isometry $L:T_xM\to T_yN$ extends to an isometry $f:U\to V$, where $U$ and $V$ are neighbourhoods of $x$ and $y$ respectively. Any two extensions $f:U\to V$ and $f:U'\to V'$ of $L$ coincide on $U\cap U'$.
Moreover $L$ extends to a  local isometry $f:M\to N$ provided that $M$ is  simply connected and $N$ is geodesically complete.
\end{lemma}

Exactly as in the Riemannian case the proof is a simple consequence of the classical Cartan--Ambrose--Hicks Theorem (see for instance \cite{piccione} for a reference).
A direct consequence of Lemma \ref{lemma:extension} is the following:

\begin{cor} \label{cor:global extension}
Let $M$ and $N$ be simply connected, geodesically complete Lorentzian manifolds of constant curvature $K$. Then any linear isometry $L:T_xM\to T_yN$ extends to a global isometry $f:M\to N$.
\end{cor}

In particular, there is a unique simply connected geodesically complete Lorentzian manifold of constant curvature $K$ up to isometries.
For instance for $K=0$ a model is the Minkowski space $\R^{n,1}$, that is $\R^{n+1}$ provided with the standard metric $$g=dx_1^2+\ldots+dx_n^2-dx_{n+1}^2~.$$ In Section \ref{subsec:poincare} we will construct an explicit model for $K=-1$.


Another consequence of Lemma \ref{lemma:extension} is that, fixing a point $x_0\in M$, the set of isometries of $M$, which we will denote by $\mathrm{Isom}(M)$, can be realized as a subset of $\mathrm{ISO}(T_{x_0}M, TM)$, namely the fiber bundle over $M$ whose fiber over $x\in M$ is the space of linear isometries of $T_{x_0}M$ into $T_x M$.
It can be proved that $\mathrm{Isom}(M)$ has the structure of a Lie group with respect to composition so that the inclusion
$\mathrm{Isom}(M)\hookrightarrow\mathrm{ISO}(T_{x_0}M, TM)$ is a differentiable proper embedding, see \cite[Theorem 4.1]{koba}.
It follows that  the maximal dimension of $\mathrm{Isom}(M)$ is
 $\dim \OO(n,1)+n+1=(n+1)(n+2)/2$. 
 
\begin{defi}
A Lorentzian manifold $M$ has \emph{maximal isometry group} if   the action of $\mathrm{Isom}(M)$ is transitive and, 
for every point $x\in M$, every linear isometry $L:T_xM\to T_xM$ extends to an isometry of $M$.
\end{defi}
Equivalently $M$ has maximal isometry group if the above inclusion of $\mathrm{Isom}(M)$ into $\mathrm{ISO}(T_{x_0}M, TM)$ is a bijection.
Hence if $M$ has maximal  isometry group, then the dimension of the isometry group is maximal. 

From Corollary \ref{cor:global extension}, every simply connected Lorentzian manifold $M$ has maximal isometry group if it has constant sectional curvature and is geodesically complete. The converse holds even without the simply connectedness assumption. Namely:
\begin{lemma} \label{lemma:max isom group implies constant curvature}
If $M$ is a Lorentzian manifold with maximal isometry group  then 
$M$ has  constant sectional curvature and is geodesically complete.
\end{lemma}
\begin{proof}
Let us show that the sectional curvature is constant. First fix a point $x\in M$. As the identity component of $\OO(T_xM)\cong \OO(n,1)$ acts transitively on spacelike planes,  there exists a constant $K$ such that Equation \eqref{eq:constcurv} holds for 
for every pair $(u,v)$ of vectors tangent at $x$  which generate a spacelike plane.
Now, for every point $x\in M$ both sides of Equation \eqref{eq:constcurv} are polynomial in $u,v\in T_xM$. Since  the set of pairs $(u,v)$ which generate spacelike planes is open in 
$T_xM\times T_x M$, we conclude that Equation \eqref{eq:constcurv} must hold for every pair $(u,v)\in T_xM\times T_x M$. 
Since $\mathrm{Isom}(M)$ acts transitively on $M$, Equation \eqref{eq:constcurv} holds  for every $(u,v)\in T_xM\times T_xM$ independently of $x$, with the same constant $K$.

To prove  geodesic completeness, we have to show that every geodesic is defined for all times. Suppose $\gamma$ is a parameterized geodesic with $\gamma(0)=x$ and $\gamma'(0)=v\in T_xM$, which is defined for a finite maximal time $T>0$. Let $T_0=T-\epsilon>0$. Then by assumption one can find an isometry $f:M\to M$ which maps $x$ to $\gamma(T_0)$ and $v$ to $\gamma'(T_0)$. Then $t\mapsto f\circ \gamma(t-T_0)$ is a parameterized geodesic which provides a continuation of $\gamma$, thus contradicting the assumption that $T<+\infty$ is the maximal time of definition.
\end{proof}

\subsection{A classification result}

Simply connected Lorentzian manifolds  with maximal isometry group play an important role, in light of the following result of classification.

\begin{prop} \label{prop:classification}
Let $M_K$ be a simply connected Lorentzian manifold of constant sectional curvature $K$ with maximal isometry group, and let $M$ be a Lorentzian manifold of constant sectional curvature $K$.
Then:
\begin{itemize}
\item
$M$ is geodesically complete if and only if there is a local isometry $p:M_K\to M$ which is a universal covering.
\item
$M$ has maximal isometry group if and only if $\mathrm{Aut}(p:M_K\to M)$ is normal in $\mathrm{Isom}(M_K)$.
\end{itemize} 
\end{prop}
\begin{proof}
If $M$ is geodesically complete, then lifting the metric to the universal cover $\widetilde M$ one gets a simply connected geodesically complete Lorentzian manifold of constant sectional curvature $K$, which by Corollary \ref{cor:global extension} is isometric to $M_K$. The covering map $p:M_K\to M$ is then a local isometry by construction. The converse is straightforward.

Now, let $\Gamma=\mathrm{Aut}(p:M_K\to M)$, which is a discrete subgroup of $\mathrm{Isom}(M_K)$. Thus $M$ is obtained as the quotient $M=M_K/\Gamma$, 
where $\Gamma$ acts freely and properly discontinuously
on $M_K$. The isometry group of $M$ is isomorphic to $N(\Gamma)/\Gamma$,
where $N(\Gamma)$ is the normalizer of $\Gamma$ in $\mathrm{Isom}(M_K)$.  {The isomorphism is based on the observation that any isometry of $\widetilde M$ which normalises $\Gamma$ descends to an isometry of $M$, and conversely 
the lifting of any isometry of $M$ must be in $N(\Gamma)$.}

Hence the condition that $M$ has maximal isometry group is equivalent to the condition that every element $f$ of $\mathrm{Isom}(M_K)$ descends in the quotient to an isometry of $M$. This is in turn equivalent to the condition that $f\Gamma f^{-1}= \Gamma$ for every $f\in \mathrm{Isom}(M_K)$, namely, that $\Gamma$ is normal in $\mathrm{Isom}(M_K)$.
\end{proof}


\begin{remark}\label{rmk:normal}
Since $\Gamma=\mathrm{Aut}(p:M_K\to M)$ is discrete, being normal in $\Isom(M_K)$ implies that elements of $\Gamma$ commute with the elements of the identity component of $\Isom(M_K)$. This remark suggests that  there are usually not many Lorentzian manifolds of constant sectional curvature with maximal isometry group.
\end{remark}

Finally, any isometry between connected open subsets  of a Lorentzian manifold $M$ with maximal isometry group extends to a global isometry. 
In particular if $M_K$ is a Lorentzian manifold of constant sectional curvature $K$ with maximal isometry group, than any Lorentzian manifold $M$ of constant sectional curvature $K$ admits a natural 
$(\mathrm{Isom}(M_K), M_K)$-structure whose charts are  isometries between open subsets of $M$ and open subsets of $M_K$. 

We will sometimes refer to Lorentzian manifolds of constant sectional curvature $K$ with maximal isometry group as \emph{models} of constant sectional curvature $K$. After these preliminary motivations, in the following we will study several models of constant sectional curvature $-1$, or in other words, models of \emph{Anti-de Sitter geometry}.

\section{Models of Anti de Sitter $(n+1)$-space} \label{ch:models any dim}
We construct here  
 models of Lorentzian manifolds with constant sectional curvature $-1$ and maximal isometry group in any dimension, by stressing the analogies with the models of hyperbolic space.
\subsection{The quadric model} \label{subsec:quadric}
Let us start by the so-called quadric model, which is the analogue of the hyperboloid model of hyperbolic space. Denote by $\R^{n,2}$ the real vector space $\R^{n+2}$ equipped with the quadratic form $$q_{n,2}(x)=x_1^2+\ldots+x_n^2-x_{n+1}^2-x_{n+2}^2~,$$ 
and by
$\langle v, w\rangle_{n,2}$ the associated symmetric form. Finally let $\OO(n,2)$ be the group of linear transformations of $\R^{n+2}$ which preserve $q_{n,2}$.

Then we define 
$$\AdS{n,1}:=\{x\in\R^{n,2}\,|\,q_{n,2}(x)=-1\}~.$$  It is immediate to check that $\AdS{n,1}$ is a smooth connected submanifold of $\R^{n,2}$ of dimension $n+1$. The tangent space  $T_{x}\AdS{n,1}$ regarded as a subspace of $\R^{n+2}$ coincides with the orthogonal space $x^\perp=\{y\in\R^{n+2}\,|\,\langle x, y\rangle_{n,2}=0\}$.
A simple signature argument shows that the restriction of the symmetric form $\langle \cdot,\cdot\rangle_{n,2}$ to $T\AdS{n,1}$ has Lorentzian signature, hence it makes
$\AdS{n,1}$ a Lorentzian manifold. We remark that this model is somehow the analogue of the hyperboloid model of hyperbolic space, that is
 \begin{equation}\label{eq:hyperbolic space}
 \Hyp^n=\{y\in\R^{n,1}\,|\,q_{n,1}(y)=-1\,,\, y_{n+1}>0\}~,
 \end{equation}
and in fact $\Hyp^n$ is isometrically embedded in $\AdS{n,1}$ as the submanifold defined by $x_{n+2}=0$, $x_{n+1}>0$.

The natural action of $\OO(n,2)$ on $\R^{n,2}$ preserves $\AdS{n,1}$ and in fact $\OO(n,2)$ acts by isometries on
$\AdS{n,1}$. 
We remark that if $x\in\AdS{n,1}$ and $v_1,\ldots, v_{n+1}$ is an orthonormal basis of $T_{x}\AdS{n,1}$ then the linear transformation of $\R^{n+2}$ sending the standard basis to the basis $v_1,\ldots, v_{n+1},x$ is in $\OO(n,2)$.
This in particular shows that $\OO(n,2)$ acts transitively on $\AdS{n,1}$ and that the stabilizer of a point $x$
acts transitively on the space of  orthonormal bases of  $T_{x}\AdS{n,1}$. These facts imply that $\AdS{n,1}$ has maximal isometry group and that the isometry group is indeed identified to $\OO(n,2)$. Notice that $\AdS{n,1}$ can be regarded as the non-Riemannian symmetric space $\OO(n,2)/\OO(n,1)$, where $\OO(n,1)$ is identified with the stabilizer of $(0,0,\ldots,0,1)$.

\subsubsection*{The sectional curvature} By Lemma \ref{lemma:max isom group implies constant curvature}, $\AdS{n,1}$ has constant sectional curvature. Let us now check that the sectional curvature is negative (actually, $K=-1$). For this purpose, observe that the normal line in $\R^{n,2}$ of $\AdS{n,1}$ at a point $x\in\AdS{n,1}$ is identified with the line generated by $x$ itself. It follows that, if $v, w$ are tangent vector fields along $\AdS{n,1}$, we have the orthogonal decomposition (we will omit here the subscript in the bilinear form $\langle v,w\rangle_{n,2}$, and simply write $\langle v,w\rangle$):
\[
      (D_v w)(x) =(\nabla_vw)(x)+\langle v,w\rangle x~,
\]
where $D$ is the flat connection of $\R^{n+2}$ and $\nabla$ is the Levi-Civita connection of $\AdS{n,1}$.
Using the flatness of $D$, one easily computes that
\[
  R(u,v)w=\langle u, w\rangle v-\langle v, w\rangle u~,
\]
so that 
\[
\langle R(u, v)v, u\rangle=-\left(\langle u, u\rangle\langle v, v\rangle-\langle v, u\rangle^2\right)~,
\]
and this shows that $\AdS{n,1}$ has constant sectional curvature $-1$. Finally, let us remark that $\AdS{n,1}$ is not simply connected, being homeomorphic to $\R^n\times S^1$. (See Figure \ref{fig:hyperboloid}.)

\begin{figure}[htb]
\includegraphics[height=7.5cm]{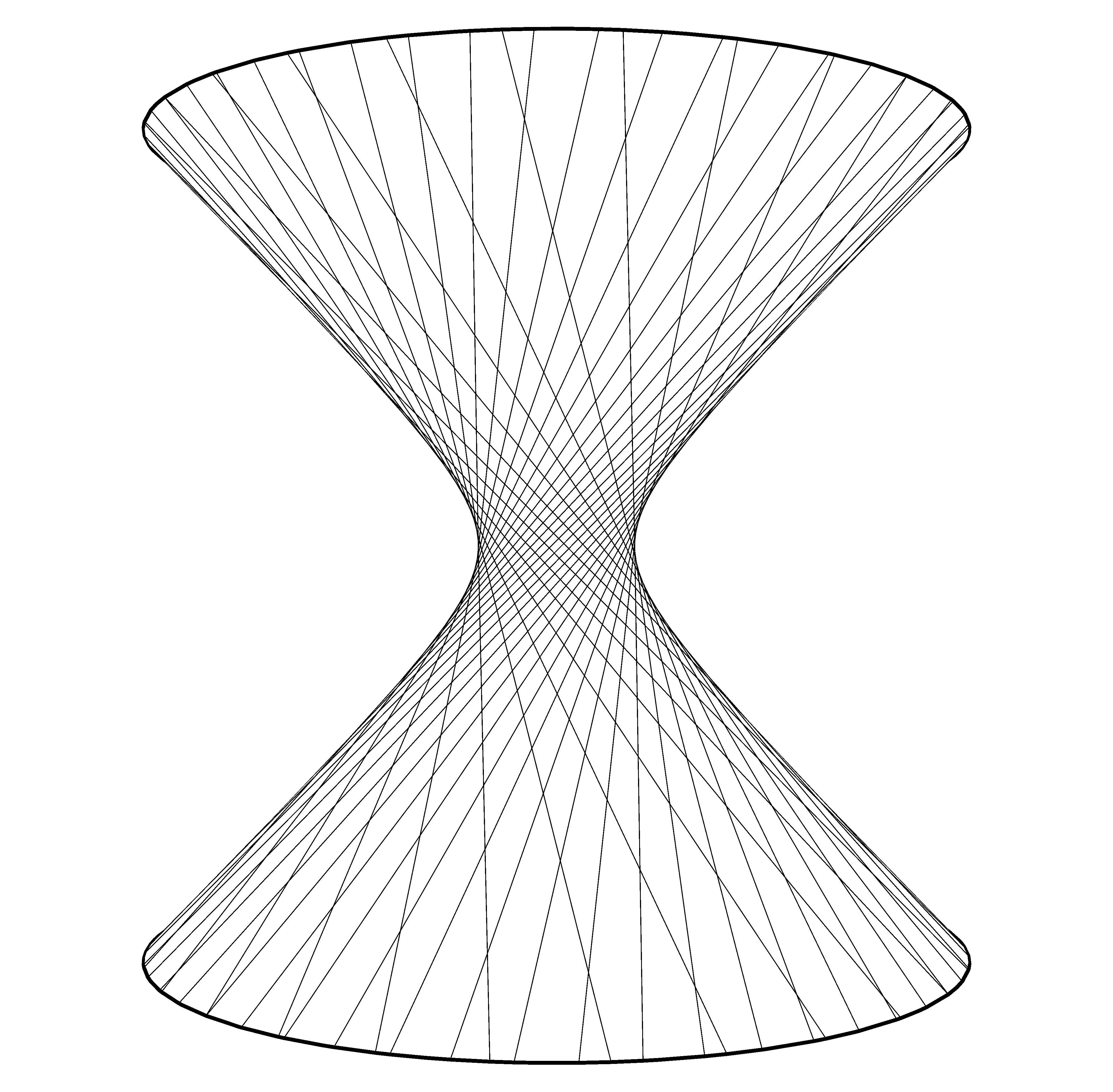}
\caption{For $n=1$, $\AdS{1,1}$ is the one-sheeted hyperboloid in $\R^{1,2}$, which is homeomorphic to the annulus $\R\times S^1$. The lines in the left and right rulings are lightlike geodesics.}\label{fig:hyperboloid}
\end{figure}

\subsection{The ``Klein model'' and its boundary}

Let us now introduce a projective model, or ``Klein model'', for Anti-de Sitter geometry. Let us define
$$\AdSP{n,1}:=\AdS{n,1}/\{\pm\En\}~.$$
Since $\{\pm\En\}$ is the center of $\OO(n,2)$ (hence normal), 
$\AdSP{n,1}$ (endowed with the Lorentzian metric induced by the quotient) has maximal isometry group by Proposition \ref{prop:classification} and is therefore a model of constant sectional curvature $-1$.
It can also be shown that the center of the isometry group of $\AdSP{n,1}$ is trivial, or equivalently that $\AdSP{n,1}$ is the quotient of its universal covering by the center of its isometry group. 
It follows (see also Remark \ref{rmk:normal}) that it is the \emph{minimal} model of AdS geometry, in the sense that any other model is a covering of $\AdSP{n,1}$.

By  definition $\AdSP{n,1}$ is naturally identified with a subset of  real projective space $\RP^{n+1}$, more precisely the subset
of timelike directions of $\R^{n,2}$:
$$\AdSP{n,1}=\{[x]\in \RP^{n+1}\,|\,q_{n,2}(x)<0\}~.$$

Like in hyperbolic geometry, the boundary of $\AdSP{n,1}$ in projective space $\RP^{n+1}$ is a quadric 
 of signature $(n,2)$, that is the projectivization of the set of   lightlike vectors in $\R^{n,2}$.  Namely
 $$\partial\AdSP{n,1}=\{[x]\in \RP^{n+1}\,|\,q_{n,2}(x)=0\}~.$$
Isometries of $\AdSP{n,1}$ induce projective transformations which 
preserve $\partial\AdSP{n,1}$.

\subsubsection*{The conformal Lorentzian structure of the boundary} In the rest of this subsection, in analogy with hyperbolic geometry, we shall equip $\partial\AdSP{n,1}$ with a conformal Lorentzian structure that extends the conformal Lorentzian structure defined 
inside. This will be obtained by means of the following construction. 

Given a point $\ell=\mathrm{Span}(x)$ of $\RP^{n+1}$, the tangent space of real projective space has the canonical identification
$$T_{\ell}\RP^{n+1}\cong\mathrm{Hom}(\ell, \R^{n+2}/\ell)~.$$
As a preliminary remark, when $\ell$ is timelike (namely $q_{n,2}(x)<0$), the quotient $\R^{n+2}/\ell$ is canonically identified with $\ell^\perp$. 
For any local section $\sigma:\AdSP{n,1}\to\R^{n,2}$ of the canonical projection, one can then define a  metric over $T\AdSP{n,1}$ by 
\begin{equation} \label{eq:metric induced by section}
\llangle f,g\rrangle_\sigma=\langle f(\sigma[x]), g(\sigma[x])\rangle_{n,2}~,
\end{equation} for $f,g\in T_{[x]}\AdSP{n,1}\cong \mathrm{Hom}(\ell, \ell^\perp)$. It is an exercise to check that if $\sigma_0$ is a section with values in $\AdS{n,1}$,
then one recovers the previously constructed metric of $\AdSP{n,1}$, which coincides with the 
pull-back of the metric over $\R^{n,2}$ since the differential of $\sigma_0$ identifies $T_{[x]}\AdSP{n,1}$ with
$T_x\AdS{n,1}=x^\perp$. This does not hold for a general section $\sigma$, but one still recovers a conformal metric as a consequence of the easy formula
\begin{equation} \label{eq: conformal change}
\llangle f,g\rrangle_{\lambda\sigma}=\lambda^2\llangle f,g\rrangle_\sigma
\end{equation}
for any function $\lambda$.


Let us now turn our attention to the case that $\ell=\mathrm{Span}(x)$ is lightlike, namely $q_{n,2}(x)=0$.
In this case there is no way to define a natural metric on $\R^{n+2}/\ell$. However, if we let 
$$\LL=\{x\in\R^{n,2}\,|\,q_{n,2}(x)=0\}$$ be the space of lightlike vectors, then $T_x\LL$ is precisely $\ell^\perp$ and contains $\ell$ itself. In fact  $T_{\ell}\partial\AdSP{n,1}$ is canonically identified to $\mathrm{Hom}(\ell, \ell^\perp/\ell)$. Moreover the bilinear form of $\R^{n,2}$, restricted to $\ell^\perp$, induces a well-defined, non degenerate bilinear form (of signature 
 $(n-1,1)$)  on $\ell^\perp/\ell$, which we denote by $\langle v,w\rangle_{\ell^\perp/\ell}$.

Hence one can define a metric on $\partial\AdSP{n,1}$ for any section $\sigma:\partial\AdSP{n,1}\to\LL$ of the canonical projection by the formula
\begin{equation} \label{eq:metric induced by section2}
\left(\!\left( f,g\right)\!\right)_\sigma=\langle f(\sigma[x]), g(\sigma[x])\rangle_{\ell^\perp/\ell}~,
\end{equation}
where now $f,g\in \mathrm{Hom}(\ell, \ell^\perp/\ell)$. Here this metric can be indeed be expressed as the pull-back 
\begin{equation}\label{confmetric:eq}
(\!( f,g)\!)_\sigma=\langle \sigma_*(f),\sigma_*(g)\rangle_{n,2}~,
\end{equation}
since the degenerate metric on $T_x\LL=\ell^\perp$ 
 is by construction the pull-back of the metric of $\ell^\perp/\ell$ by the projection along the degenerate direction $\ell$.
 
One again has the formula
\begin{equation} \label{eq: conformal change2}
\left(\!\left( f,g\right)\!\right)_{\lambda\sigma}=\lambda^2\left(\!\left( f,g\right)\!\right)_\sigma
\end{equation}
similarly to \eqref{eq: conformal change}, and therefore the induced conformal class over $T\partial\AdSP{n,1}$ is independent of the choice of $\sigma$ and equips $T\partial\AdSP{n,1}$ with a conformal Lorentzian metric. The naturality of the construction implies that the isometry group of $\AdS{n,1}$ acts by conformal transformations over $\partial\AdSP{n,1}$.
Finally, let us show that this conformal Lorentzian metric is naturally the conformal compactification of $\AdSP{n,1}$. In fact, if $\sigma$ is a section of the canonical projection $\pi:\R^{n,2}\to \RP^{n+1}$, defined in a neighborhood $U$ of a point of $\partial\AdSP{n,1}$, by construction 
the metric $(\!(\cdot,\cdot )\!)_\sigma$ over $\partial\AdSP{n,1}\cap U$ is the limit of the conformal metric associated to $\sigma$ defined in $\AdSP{n,1}\cap U$: this means that if $(p_n,v_n)$ is a sequence in $T\AdSP{n,1}$ that converges to $(p_\infty, v_\infty)\in T\partial\AdSP{n,1}$, then 
$\llangle v_n, v_n\rrangle_\sigma(p_n)\to(\!( v_\infty, v_\infty)\!)_\sigma(p_\infty)$. 


In the physics literature, the conformal Lorentzian manifold $\partial\AdSP{n,1}$ is known as \emph{Einstein universe}. 
See for instance \cite{francesphd,frances1,primer} for more details.

\begin{remark}\label{rmk lightlike cone boundary}
A conformal Lorentzian structure is equivalent to the smooth field of lightlike directions, which is also called the \emph{light cone}. 
More precisely, a diffeomorphism $f:(M,g)\to (N,g')$ between Lorentzian manifold is conformal, meaning that $f^*g'=e^{2\lambda}g$ for some smooth function $\lambda:M\to \R$, if and only if the differentials of $f$ and $f^{-1}$ map causal vectors to causal vectors. If $M$ and $N$ have dimension $\geq 3$, this is indeed equivalent to the condition that the differential of $f$ maps lightlike vectors to lightlike vectors. 
\end{remark}

\begin{remark}\label{rmk lightlike cone boundary2}
In order to better understand the light cone in the case of $\Q{n,1}$, let us notice that if $[y]\in\Q{n,1}$ formula \eqref{confmetric:eq} implies that the lightlike vectors in $T_{[y]}\Q{n,1}$ are the projection of
vectors $x\in\R^{n+2}$ such that $\langle x, y\rangle_{n,2}=0$ and $\langle x, x\rangle_{n,2}=0$. These vectors are such that $\mathrm{Span}(x,y)$ are  totally degenerate planes in
$\R^{n,2}$, or equivalently give projective lines contained in $\Q{n,1}$. Thus the light cone in $\Q{n,1}$ through $[y]$ is the union of all the projective lines contained in $\Q{n,1}$ which pass through $[y]$.
\end{remark}


\subsection{The ``Poincar\'e model'' for the universal cover} \label{subsec:poincare}

We have already observed that $\AdS{n,1}$, and its double quotient $\AdSP{n,1}$, are not simply connected. Let us now construct a simply connected model of Anti-de Sitter geometry, namely the universal cover of  $\AdS{n,1}$ and $\AdSP{n,1}$.
For this purpose,
 let  $\Hyp^n$ be the hyperboloid model of hyperbolic space (defined in \eqref{eq:hyperbolic space}). Then 
\begin{equation} \label{eq:covering map}
\pi(y, t)=(y_1,\ldots, y_n,y_{n+1} \cos t,  y_{n+1} \sin t)
\end{equation}
defines a map $\pi:\Hyp^n\times\R\to \AdS{n,1}$ and it is immediate yo check that this map is a covering with deck transformations of the form
$(y, t)\mapsto (y, t+2k\pi)$ for $k\in\mathbb Z$. See Figure \ref{fig:cylinder} for a picture in dimension 3. Clearly $\AdSU{n,1}$ is the universal cover also for the projective model $\AdSP{n,1}$, the covering map being the composition of \eqref{eq:covering map} with the double quotient $\AdS{n,1}\to \AdSP{n,1}$. 

Pulling back the Lorentzian metric over $\Hyp^n\times\R$ we get a simply connected Lorentzian manifold of constant curvature $-1$, which we  denote by $\AdSU{n,1}$. The metric of $\AdSU{n,1}$ is a warped product of the form 
\begin{equation}\label{eq:warped}
\pi^*g_{\AdS{n,1}}=g_{\Hyp^n}-y_{n+1}^2 dt^2~.
\end{equation}
Moreover $\AdSU{n,1}$ has maximal isometry group, hence we have obtained a simply connected model for AdS geometry. More precisely, we have a central extension, that is a (non split) short exact sequence
$$0\to\mathbb Z\to \Isom(\AdSU{n,1})\to \OO(n,2)\to 1~.$$

It is convenient to express the metric \eqref{eq:warped} using the Poincar\'e model of the hyperbolic space. Recall that the disk model  of the hyperbolic space is the unit disk $\disk^n=\{\xmm\in\R^n|||\xmm||<1\}$ equipped with the conformal metric 
$\frac{4}{(1-\mathsf r^2)^2}\sum d\xmm_i^2$, where
$\mathsf r^2:=||\xmm||^2=\sum\xmm_i^2$.
The isometry with the hyperboloid model of $\Hyp^n$ is given by the transformation 
$$(\xmm_1,\ldots,\xmm_n)\mapsto \left(y_1=\frac{2\xmm_1}{1-\mathsf r^2},\ldots, y_n=\frac{2\xmm_n}{1-\mathsf r^2},\,y_{n+1}=\frac{1+\mathsf r^2}{1-\mathsf r^2}\right)~.$$
In conclusion $\AdSU{n,1}$ has the model $\disk^n\times\R$ equipped with the metric
\begin{equation}\label{poinc:eq}
         \frac{4}{(1-\mathsf r^2)^2}(d\xmm_1^2+\ldots+d\xmm_n^2)-\left(\frac{1+\mathsf r^2}{1-\mathsf r^2}\right)^2d\mathsf t^2\,.
\end{equation}
The ``Poincar\'e model'' of the AdS geometry, which has been introduced in \cite{bon_schl_inv}, is then the cylinder $\disk^n\times\R$ equipped with the metric \eqref{poinc:eq}. 
From Equation  \eqref{poinc:eq}, each slice $\{\mathsf t=\mathsf c\}$ is 
a totally geodesic copy of $\Hyp^n$, a fact which will be evident also from other reasons in Section \ref{sec:geodesics}.  The expression \eqref{poinc:eq} also shows that the vector field ${\partial}/{\partial \mathsf t}$ is a timelike non-vanishing vector field on $\AdSU{n,1}$, which shows that $\AdSU{n,1}$ is time-orientable. Since any choice of time orientation is preserved by the action of deck transformations of the covering $\AdSU{n,1}\to\AdSP{n,1}$, this shows that also $\AdS{n,1}$ and $\AdSP{n,1}$ are time-orientable.
Notice however that vertical lines in the metric  are not geodesic \eqref{poinc:eq}, except for the central line, passing through $\xmm_1=\ldots=\xmm_n=\mathsf t=0$.

\subsubsection*{The conformal metric of the boundary} Using the covering map from $\Hyp^n\times\R$ to $\AdSP{n,1}$, we can endow $\AdSU{n,1}$ (and similarly any other covering of $\AdSP{n,1}$) with a natural boundary. Concretely, 
the covering map (now in the projective model of $\Hyp^n$)
$$\pi'([y_1:\ldots:y_{n},:y_{n+1}], t)=[y_1:\ldots:y_{n}:y_{n+1}\cos t:y_{n+1}\sin t]$$ extends to $\pi':(\Hyp^n\cup\partial\Hyp^n)\times\R\to \AdSP{n,1}\cup\partial\AdSP{n,1}$. To compute the conformal Lorentzian structure of the boundary, we consider the map $\tau:\Hyp^n\times\R\to\R^{n+2}$ defined by
$$\tau([y_1:\ldots:y_{n},:y_{n+1}], t)=(y_1/y_{n+1},\ldots,y_{n}/y_{n+1},\cos t,\sin t)$$
 which clearly extends to the boundary, and induces a (local) section of the projection $\R^{n+2}\to\AdSP{n,1}$. In fact, if $\eta$ is the generator of the group of deck transformations of the covering $\pi':\AdSU{n,1}\to\AdSP{n,1}$, then $\tau$ has the equivariance $\tau\circ \eta^i=(-1)^i\tau$. 
 Using also \eqref{eq: conformal change2}, the conformal Lorentzian metric on $\partial\AdSU{n,1}$ induced by $\sigma$ by means of Equation  \eqref{eq:metric induced by section2} is the pull-back of a Lorentzian metric compatible with the natural conformal structure of the boundary $\partial\AdSP{n,1}$. A direct computation (which becomes very simple by using the metric \eqref{poinc:eq}, the formula \eqref{eq: conformal change} and the observation that $\tau$ differs by the hyperboloid section by the factor $y_{n+1}=\frac{1+\mathsf r^2}{1-\mathsf r^2}$) gives the expression 
 \begin{equation}\label{eq:conf}
  \frac{4}{(1+\mathsf r^2)^2}(d\xmm_1^2+\ldots+d\xmm_n^2)-d\mathsf t^2 ~.
\end{equation}
This metric extends to $\overline{\disk}{}^n\times\R$ and thus the metric $g_{\Sp^{n-1}}-dt^2$ on $\Sp^{n-1}\times\R$, where $g_{\Sp^{n-1}}$ is the round metric over the sphere, is compatible with the conformal Lorentzian structure of $\partial\AdSU{n,1}$. This also shows that the conformal structure of $\partial\AdS{n,1}\cong \Sp^{n-1}\times \Sp^1$ admits the representative $g_{\Sp^{n-1}}-g_{\Sp^1}$, and the conformal structure of $\partial\AdSP{n,1}$ is compatible with the double quotient of the latter, by the involution $(p,q)\mapsto (-p,-q)$ on $\Sp^{n-1}\times \Sp^1$.

\subsection{Geodesics} \label{sec:geodesics}
Let us now study more precise properties of AdS geometry, concerning its geodesics.

\subsubsection*{In the quadric model}

Let us start with the exponential map in the hyperboloid model.  Given a point $x\in\AdS{n,1}$ and $v\in T\AdS{n,1}$ we shall determine the geodesic through $x$ with speed $v$. 
Let us distinguish several cases according to the sign of $q_{n,2}(v)$.
If $v$ is lightlike, then
$$\gamma(t)=x+tv$$ is a geodesic of $\R^{n,2}$ and is contained in $\AdS{n,1}$, hence $\gamma$ is a  geodesic for the intrinsic metric. See Figure \ref{fig:hyperboloid}.

If $v$ is either timelike or spacelike, we claim that the geodesic $\gamma(t)=\exp_x(tv)$ is contained in the linear plane $W=\mathrm{Span}(x,v)$. In fact, the linear transformation $T$ that fixes pointwise $W$ and whose restriction to
$W^\perp$ is $-\En_{W^\perp}$ is in $\OO(n,2)$. By the uniqueness of the solutions of the geodesic equation, $T\circ\gamma=\gamma$ hence $\gamma$ is necessarily contained in $\AdS{n,1}\cap W$.
One can then easily derive the expressions 
\begin{equation} \label{eq:space geo quadric}
\gamma(t)=\mathrm{cosh}(t) x+\mathrm{sinh}(t) v
\end{equation}
if $q_{n,2}(v)=1$ and
\begin{equation} \label{eq:time geo quadric}
\gamma(t)=\mathrm{cos}(t) x+\mathrm{sin}(t) v
\end{equation} if $q_{n,2}(v)=-1$.

\subsubsection*{In the Klein model}
In analogy with the hyperbolic case, in the Klein model $\AdSP{n,1}$ geodesics are intersection of projective lines with the domain 
$\AdSP{n,1}\subset\RP^{n+1}$. From the above discussion, 
\begin{itemize}
\item Timelike geodesics correspond to projective lines that are entirely contained in $\AdSP{n,1}$, are closed non-trivial loops and have length $\pi$.
\item Spacelike geodesics correspond to lines that meet $\Q{n,1}$ transversally in two points. They have infinite length. 
\item Lightlike geodesics correspond to 
lines tangent to $\Q{n,1}$.
\end{itemize}
In particular the light cone through a point $[x]\in\AdSP{n,1}$ coincides with the cone of lines through $[x]$ tangent to $\Q{n,1}$. See Figure \ref{fig:hyperboloids} for a picture (in dimension 3) in an affine chart, where geodesics look like straight lines. For instance in the affine chart $\mathbb A_{n+2}=\{x_{n+2}\neq 0\}$, where 
in coordinates $(y_1,\ldots,y_{n+1})=(x_1/x_{n+2},\ldots,x_{n+1}/x_{n+2})$, the intersection $\AdSP{n,1}\cap \mathbb A_{n+2}$ is the interior of a one-sheeted hyperboloid, that is,
$$\AdSP{n,1}\cap \mathbb A_{n+2}=\{y_1^2+\ldots+y_n^2-y_{n+1}^2<1\}~,$$
while its boundary is the one-sheeted hyperboloid itself:
$$\partial\AdSP{n,1}\cap \mathbb A_{n+2}=\{y_1^2+\ldots+y_n^2-y_{n+1}^2=1\}~.$$
In an affine chart, timelike geodesics corresponds to affine lines which are entirely contained in the Anti de Sitter space, and which are not asymptotic to its boundary; 
lightlike geodesics are tangent to the one-sheeted hyperboloid, or are asymptotic to it (tangent at infinity). 

\begin{figure}[htb]
\centering
\begin{minipage}[c]{.5\textwidth}
\centering
\includegraphics[height=8cm]{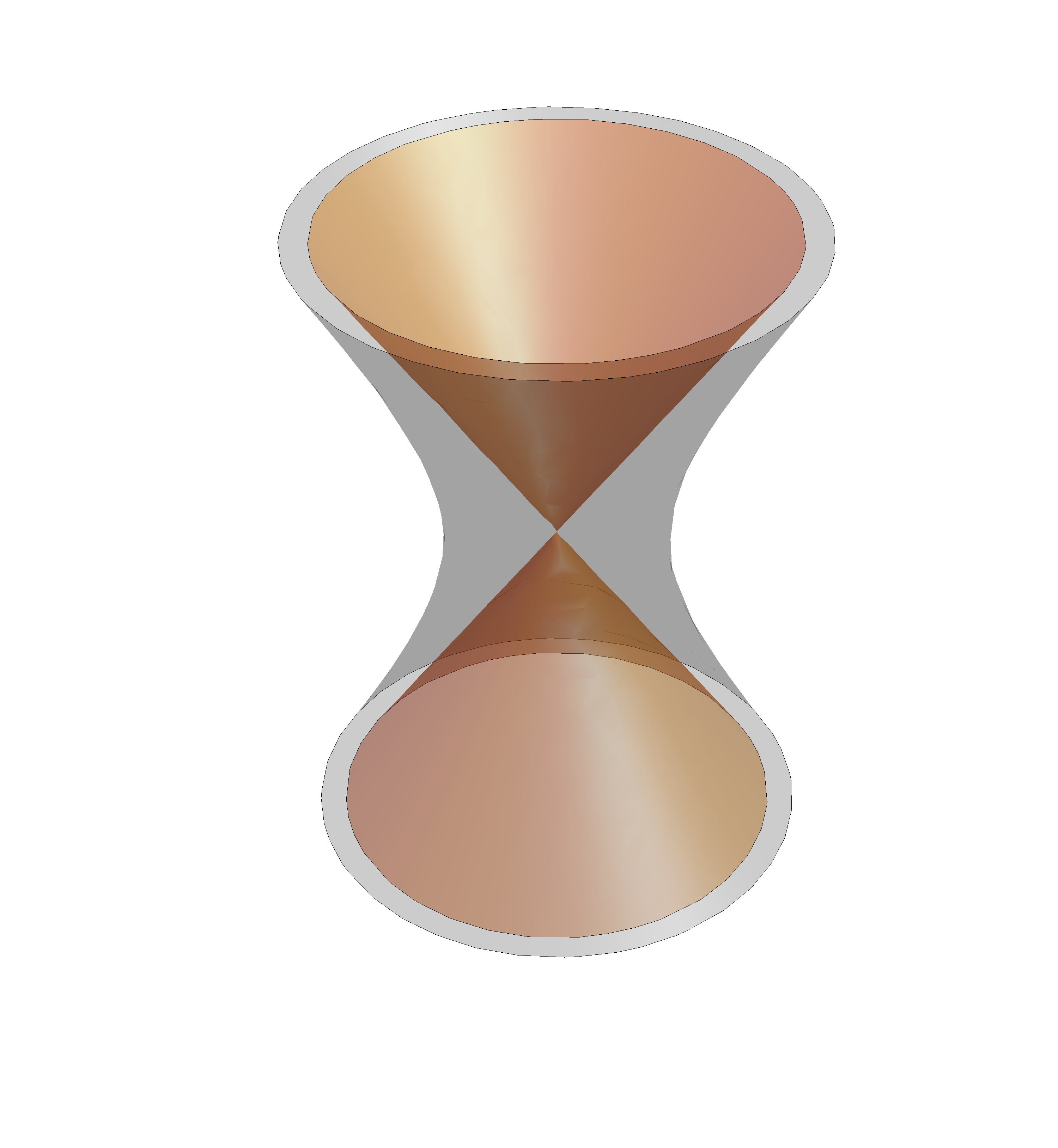}
\end{minipage}%
\begin{minipage}[c]{.5\textwidth}
\centering
\includegraphics[height=8cm]{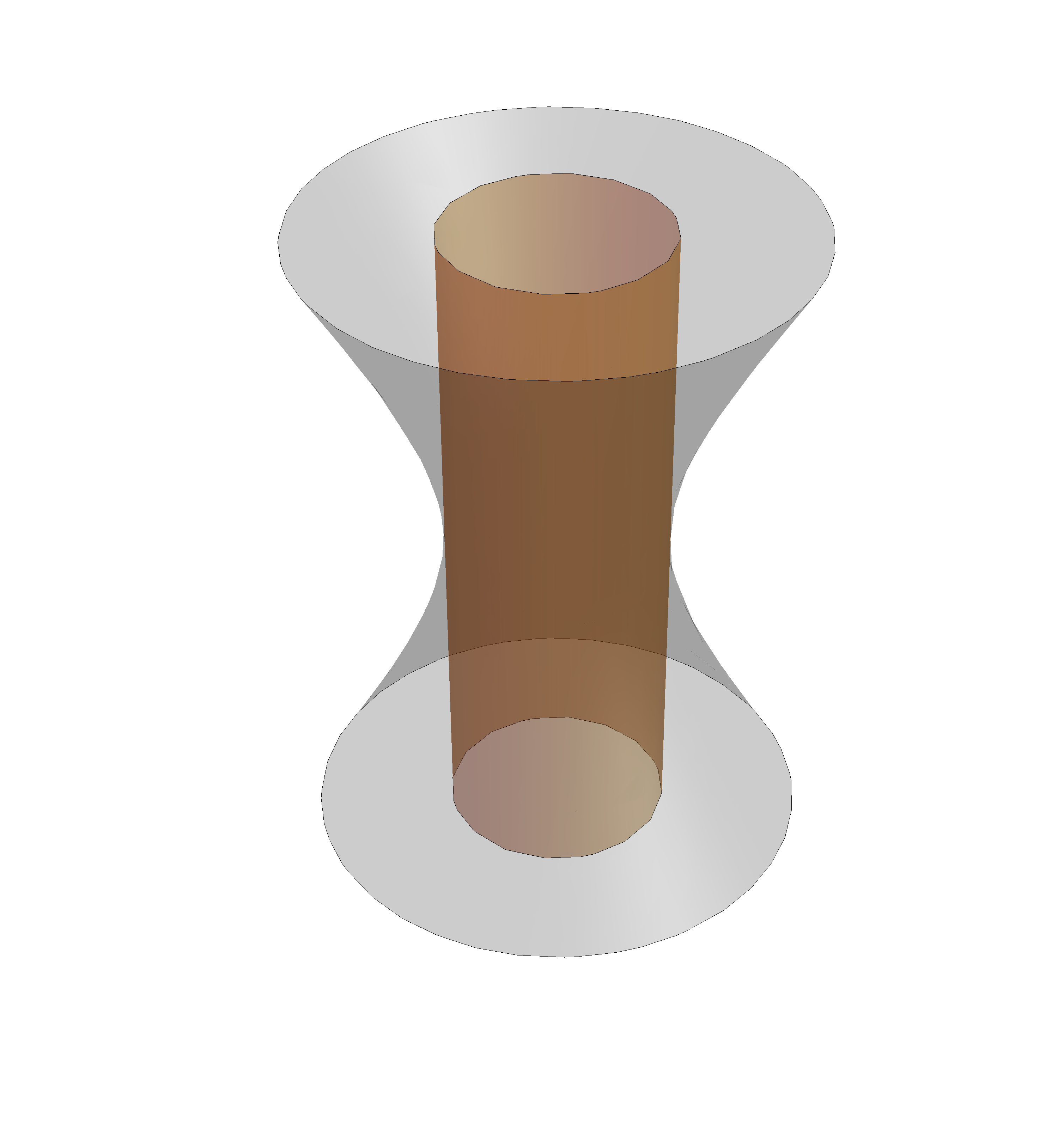}
\end{minipage}
\vspace{-1cm}
\caption{The projective model of three-dimensional AdS space in an affine chart. The interior quadric is the lightcone from the point $[0:0:0:1]$, which is tangent to the boundary as explained in Section \ref{sec:geodesics}, pictured in the affine charts $x_4\neq 0$ (left) and  $x_3\neq 0$ (right).    \label{fig:hyperboloids}}
\end{figure}


\begin{remark}
An important observation concerns the space of timelike geodesics. Any timelike line is the projectivisation of a negative definite plane. As $\Isom(\AdSP{n,1})\cong\mathrm{P}\OO(n,2)$ acts transitively on the space of timelike lines,
and since the stabiliser of a timelike line is the group $\mathrm{P}(\OO(n)\times \OO(2))$ which is the maximal compact subgroup of $\mathrm{P}\OO(n,2)$, the space of timelike geodesics of 
$\AdSP{n,1}$ is 
naturally identified with the Riemannian symmetric space of $\mathrm{P}\OO(n,2)$.
\end{remark}

\subsubsection*{Totally geodesic subspaces} Before discussing the geodesics in the Poincar\'e model, let us briefly discuss more in general totally geodesics subspaces. By an argument analogous to the case of geodesics, totally geodesic subspaces of $\AdSP{n,1}$ of dimension $k$ are obtained as the intersection of $\AdSP{n,1}$ with the projectivisation $\PR(W)$ of
  a linear subspace $W$ of $\R^{n,2}$ of dimension $k+1$. The negative index of $W$ can be either $2$ or $1$, for otherwise the intersection 
$\AdSP{n,1}\cap \PR(W)$ would be empty. We have several cases -- see Figure \ref{fig:planes}:
\begin{itemize}
\item 
If $W$ has signature $(k-1,2)$, then $\PR(W)\cap\AdSP{n,1}$  is isometric to $\AdSP{k-1,1}$.
\item
If $W$ has signature $(k-2,1)$, then it is a copy of Minkowski space $\R^{k-2,1}$, hence $\PR(W)\cap\AdSP{n,1}$ is a copy of the Klein model of hyperbolic space.
\item
If $W$ is degenerate, then $\PR(W)\cap\AdSP{n,1}$ is a lightlike subspace foliated by lightlike geodesics tangent to the same point of $\Q{n,1}$. 
 \end{itemize}

A particular case of the last point is when $W$ is degenerate and $\dim W=n+1$. Then 
 $\PR(W)\cap\AdSP{n,1}$ is a projective hyperplane tangent to $\Q{n,1}$ at a point $[x]$ and $\PR(W)\cap\Q{n,1}$ is the lightlike cone of $\Q{n,1}$ through $[x]$ (Remark \ref{rmk lightlike cone boundary2}). 
 
 \begin{figure}[htb]
\includegraphics[height=8.5cm]{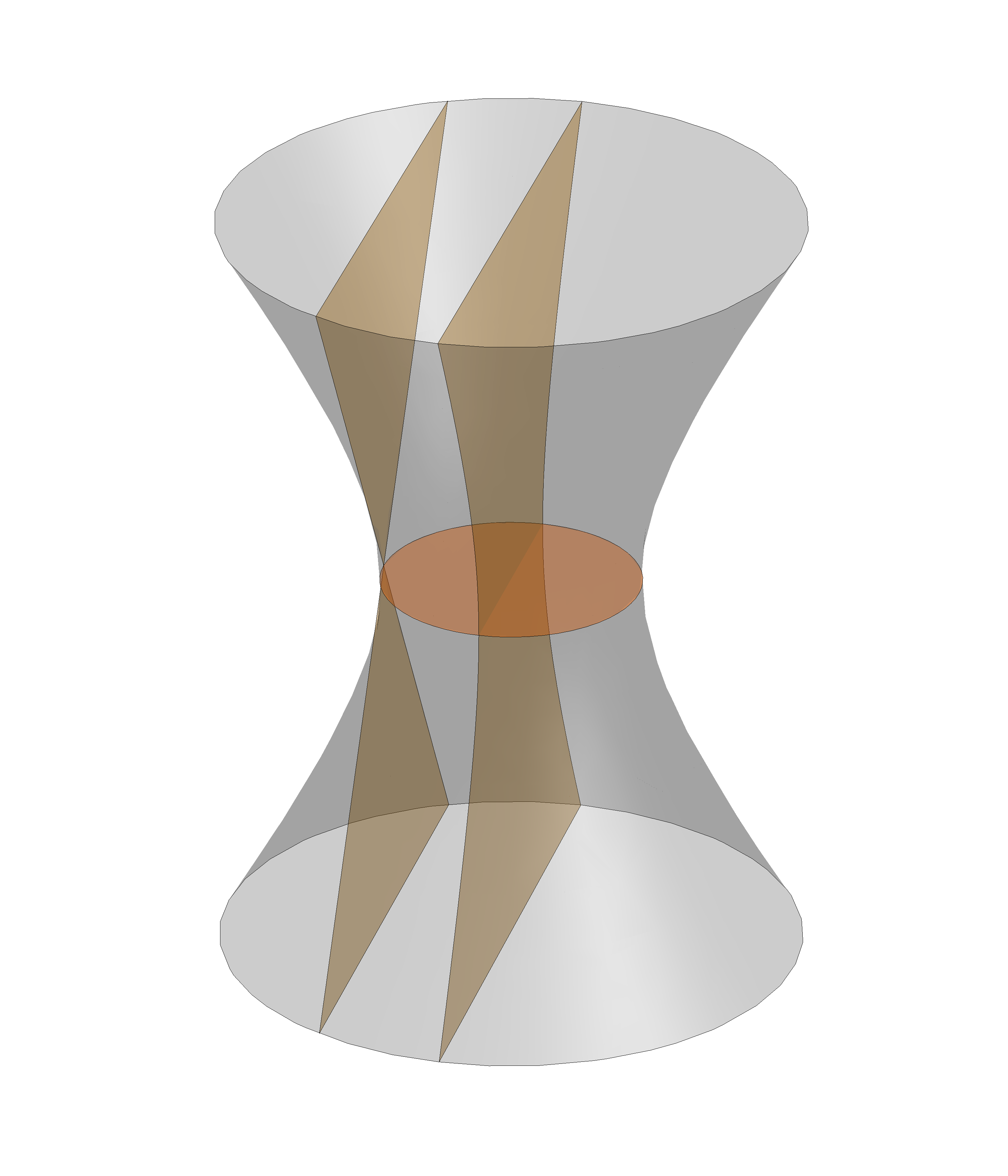}
\vspace{-1cm}
\caption{In an affine chart for $\AdSP{2,1}$, a spacelike plane (horizontal), which intersects a timelike plane (vertical) in a spacelike geodesic. A lightlike plane (on the left) is tangent to $\partial\AdSP{2,1}$ at a point.}\label{fig:planes}
\end{figure}

\subsubsection*{In the universal cover} In the universal cover $\AdSU{n,1}$, geodesics are the lifts of the geodesics of the models $\AdSP{n,1}$ or $\AdS{n,1}$ which we have just described. Hence every lightlike or spacelike geodesic in $\AdSP{n,1}$ and $\AdS{n,1}$, which is topologically a  line, has a countable number of lifts to $\AdSU{n,1}$. On the other hand timelike geodesics in $\AdSP{n,1}$ and $\AdS{n,1}$ are topologically circles and are in bijections with timelike geodesics of $\AdSU{n,1}$, as the covering map from $\AdSU{n,1}$, restricted to a timelike geodesic, induces a covering map onto the circle.

Using the Poincar\'e model for the universal cover, introduced in Section \ref{subsec:poincare}, it is easy to give an explicit description of (unparameterized) lightlike geodesics. In fact, in Lorentzian geometry not only the nature of a vector (i.e. timelike, lightlike or spacelike) is conformally invariant, but also unparameterized lightlike are a conformal invariant. More concretely, the following holds, see for instance \cite[Proposition 2.131]{galhullaf}.

\begin{theorem}\label{thm:geo conf inv}
If two Lorentzian metrics $g$ and $g'$ on a manifold $M$ are conformal, then they have the same unparameterized lightlike geodesics.
\end{theorem}

As a consequence of Theorem \ref{thm:geo conf inv}, we can replace the Poincar\'e metric \eqref{poinc:eq} by the conformal metric given by \eqref{eq:conf}:
\begin{equation} \label{poinc:eq conformal}
\frac{4}{(1+\mathsf r^2)^2}(d\xmm_1^2+\ldots+d\xmm_n)^2-d\mathsf t^2\,.
\end{equation}
Now observe that the first term in the expression \eqref{poinc:eq conformal} is exactly the form of the spherical metric on a hemisphere, pulled-back to the unit disc by means of the stereographic projection. 
We will call such a metric the \emph{hemispherical} metric and we will denote it, with a small abuse of notation, by $g_{\Sp^n}$.
In other words, the conformal metric \eqref{poinc:eq conformal} is isometric to $g_{\Sp^n}-dt^2$ on the product of a hemisphere and the line. 
The boundary of $\partial\disk$ is an equator for the hemispherical metric, and in fact it is the only equator completely contained in $(\disk\cup\partial\disk, g_{\Sp^n})$, which justifies the fact that it will be called \emph{the} equator for simplicity.

As a consequence, unparameterized lightlike geodesics of $\AdSU{n,1}$ going through a point $(\mathsf p_0,\mathsf t_0)$ are characterized by the conditions that they are mapped to spherical geodesics under  the vertical projection $(\mathsf p,\mathsf t)\mapsto \mathsf p$ and moreover 
\begin{equation}\label{eq lightgeo}
\mathsf t-\mathsf t_0=d_{\mathbb S^n}(\mathsf p,\mathsf p_0)
\end{equation}
 on the geodesic. In particular, these lightlike geodesics meet the boundary of $\AdSU{n,1}$ at the points which satisfy \eqref{eq lightgeo} such that $\mathsf p$ is on the equator of the hemisphere: as an example, if $\mathsf p_0$ is the center of the hemisphere, then the points at infinity of the lightcone over $(\mathsf p_0,\mathsf t_0)$ are the horizontal slice $\mathsf t=\mathsf t_0+\pi/2$. This sphere is also the boundary of a hyperplane dual to $(\mathsf p_0,\mathsf t_0)$, see next section.
 
  The same argument also permits to describe explicitely a lightlike hyperplane in the Poincar\'e model for the universal cover: the lightlike hyperplane having $(\mathsf p_0,\mathsf t_0)$ as a past endpoint, (where now $\mathsf p_0$ is on the equator) is precisely $\{(\mathsf p,\mathsf t)\,|\,\mathsf t-\mathsf t_0=d_{\mathbb S^n}(\mathsf p,\mathsf p_0)\}$, and its future endpoint is $(-\mathsf p_0,\mathsf t+\pi).$
  See Figure \ref{fig:causality} for pictures in dimension $2+1$.

\begin{figure}[htb]
\centering
\begin{minipage}[c]{.33\textwidth}
\centering
\includegraphics[height=8cm]{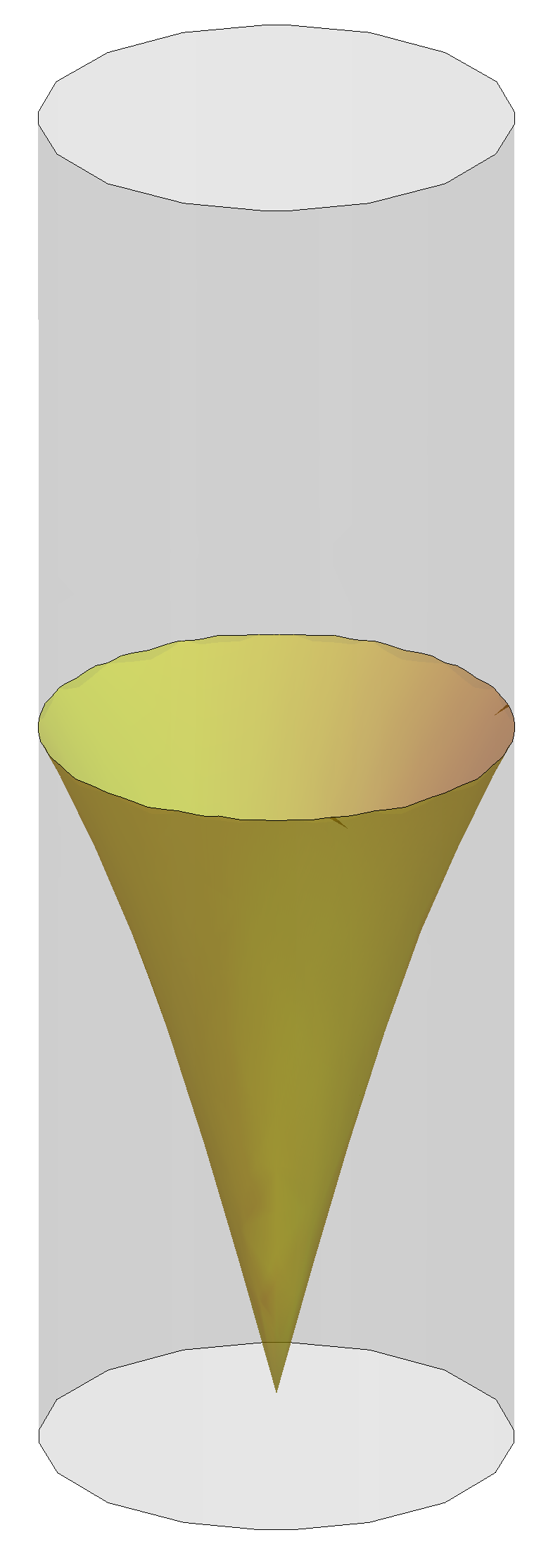}
\end{minipage}%
\begin{minipage}[c]{.33\textwidth}
\centering
\includegraphics[height=8cm]{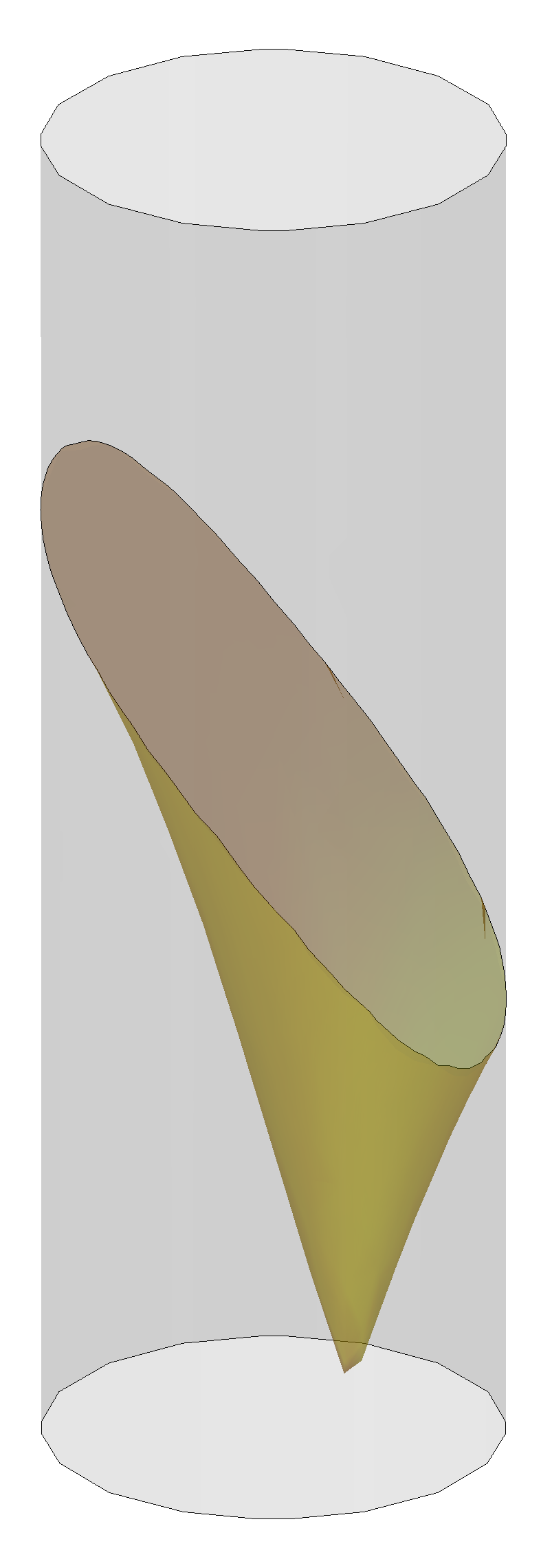}
\end{minipage}
\begin{minipage}[c]{.33\textwidth}
\centering
\includegraphics[height=8cm]{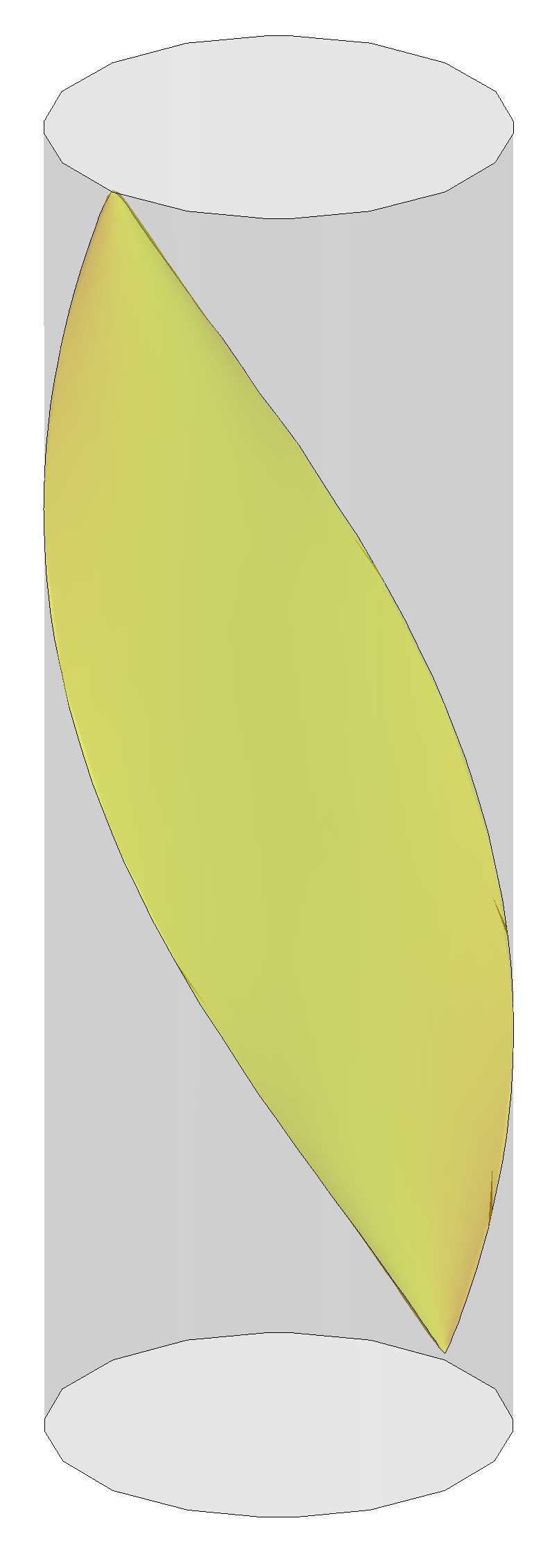}
\end{minipage}%
\caption{In the left and middle pictures, future lightcones over a point in $\AdSU{2,1}$. In the left picture the basepoint of the lightcone projects to the center of the disc, and therefore the closure of the lightcone in the cylinder $\partial\AdSU{2,1}$ is a horizontal slice. In the right picture, a lightlike plane, which is actually the degenerate limit of future lightcones as the basepoints tend to the boundary. \label{fig:causality}}
\end{figure}


\subsection{Polarity in Anti-de Sitter space}\label{sec:duality}

The quadratic form $q_{n,2}$ induces a polarity on the projective space $\RP^{n+1}$, namely the correspondence which associates to the projective subspace $\PR(W)$ the subspace $\PR(W^\perp)$.
In particular this correspondence induces a duality between spacelike totally geodesic subspaces of $\AdSP{n,1}$: the dual of a spacelike
$k$-dimensional subspace is a $n-k+1$ subspace. For instance the dual of a point $[x]$ is a $n$-dimensional spacelike hyperplane $P_{[x]}=\PR(x^\perp)$.
Projectively $P_{[x]}$ is characterised as the hyperplane spanned by the intersection of $\Q{n-1,1}$ with the lightcone from $[x]$.
More geometrically, it can be checked that $P_{[x]}$ is the set of antipodal points to $[x]$ along timelike geodesics through $[x]$. 
Also, every timelike geodesic
through $[x]$ meets $P_{[x]}$ orthogonally at time $\pi/2$. Conversely, given a totally geodesic spacelike hyperplane $H$, all the timelike geodesics that meet $H$ orthogonally  intersect in a single point, which is the dual point of $H$.

\subsubsection*{In the quadric model}
To some extent, the duality between points and planes lifts to the coverings of $\AdSP{n,1}$. In $\AdS{n,1}$ there are two dual planes associated to any point $x$: the sets 
$$P_{x}^\pm=\{\exp_{x}(\pm(\pi/2)v)\,|\,q_{n,2}(v)=-1,\,v\text{ future-directed}\}~.$$
Clearly $P_{x}^+$ and $P_{x}^{-} $ are antipodal and  $P_{-x}^\pm=P_{x}^\mp$. The planes $P_x^{\pm}$ disconnect $\AdS{n,1}$ in two regions $U_x$ and
$U_{-x}$, where $U_{x}$ is the region containing $x$. See Figure \ref{fig:torus1}.
They can be characterised by 
$$U_{x}=\{y\in\AdS{n,1}\,|\,\langle x, y\rangle_{n,1}<0\}~.$$

Spacelike and lightlike geodesics through $x$ do not exit $U_{x}$, 
while all the timelike geodesics through $x$ meet orthogonally $P^{\pm}_x$ and all pass through the point $-x$.
More precisely, a point $y\neq x$ is connected to $x$:
\begin{itemize}
\item by a spacelike geodesic if and only if $\langle x,y\rangle_{n,1}<-1$,
\item by a lightike geodesic if and only if $\langle x,y\rangle_{(n,1)}=-1$,
\item by a timelike geodesic if and only if 
$|\langle x,y\rangle_{(n,1)}|<1$.
\end{itemize}
(To check this, see also the expressions of geodesics in Section \ref{sec:geodesics}.) An immediate consequence is that if $y$ is connected to $x$ by a spacelike geodesic, there is no geodesic joining $y$ to $-x$.
Hence the exponential map of $\AdS{n,1}$ is not surjective. But as any point $y\in\AdS{n,1}$ can be connected through a geodesic either to $x$ or to $-x$,
the exponential over $\AdSP{n,1}$ is surjective.

\begin{figure}[htb]
\includegraphics[height=5.5cm]{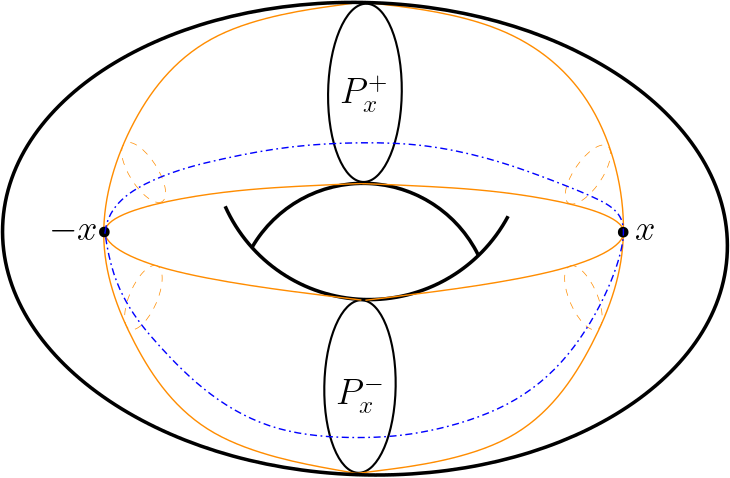}
\caption{The duality in $\AdS{2,1}$, which is the interior of a solid torus. The lightcone from a point $x$ is tangent to $\partial\AdS{2,1}$ in two meridians, which span the dual planes $P_x^{\pm}$. Timelike geodesics through $x$ intersect $P_x^{\pm}$ orthogonally and all meet again at the antipodal point $-x$. The region $U_x$ is the solid cylinder bounded by $P_x^{\pm}$ and containing $x$.}\label{fig:torus1}
\end{figure}

\subsubsection*{In the universal cover}
Finally, let us consider the situation in $\AdSU{n,1}\cong \Hyp^n\times\R$. Recall that the group of deck transformations for the covering $\AdSU{n,1}\to\AdS{n,1}$ is $\Z$, where a generator acts by translations of $2\pi$ in the $\R$ factor.
Hence the  preimage of a spacelike plane $P\subset\AdSP{n,1}$ is the disjoint union of spacelike planes $(P^k)_{k\in\Z}$, enumerated so that the generator $\eta$ of $\Z$ acts by sending $P^k$ to $P^{k+1}$.
Moreover each connected component of $\AdSU{n,1}\setminus\bigcup_{k\in\Z}P^k$ is a fundamental  domain for the action of deck transformations of the covering $\AdSU{n,1}\to\AdSP{n,1}$. 

Now given a point $x$, let us apply the previous construction to the plane $P_x=P_{\pi'(x)}$ which is the dual of the image $\pi'(x)$ in $\AdSP{n,1}$, and let $V_x$ be the connected component which contains $x$. 
 We will refer to $V_x$ as the \emph{Dirichlet domain} in $\AdSU{2,1}$
 centered at $x$, since the construction of $V_x$ is the analogue of a Dirichlet domain in this context.
 Then the restricted covering map $\pi'|_{V_x}:V_x \to\AdSP{n,1}\setminus P_{x}$ is an isometry. Therefore lightlike and spacelike geodesics through $x$ 
are entirely contained in $V_x$. See Figure \ref{fig:cylinder}.

\begin{figure}[htb]
\includegraphics[height=2.6cm]{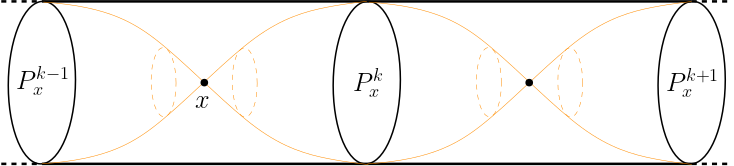}
\caption{A topological picture of the universal cover $\AdSU{2,1}$. The planes $P_x^k$ are spacelike and differ by deck transformations. The Dirichlet domain $V_x$ is a solid cylinder containing $x$, bounded by $P_x^{k-1}$ and $P_x^k$.}\label{fig:cylinder}
\end{figure}

\section{Anti de Sitter space in dimension $(2+1)$}\label{ch:ads dim 3}
The purpose of this section is to focus on some peculiarites of Anti-de Sitter geometry in dimension three.

\subsection{The $\PSL(2,\R)$-model}
The fundamental observation is the existence of a special model in dimension three which naturally endows Anti-de Sitter space with a Lie group structure. To construct this, 
consider  the vector space $\V$ of $2\times 2$ matrices with real entries. Then $q=-\det$ is a quadratic form with signature $(2,2)$, hence there is an isomorphic identification
between $(\V, -\det)$ and $(\R^{2,2}, q_{2,2})$, unique up to composition by elements in $\OO(2,2)$. Under this isomorphism  $\AdS{2,1}$ is identified with the Lie group
$\SL(2,\R)$. 

Let us notice that $\SL(2, \R)\times \SL(2, \R)$ acts linearly on $\V$ by left and right multiplication: \begin{equation}\label{eq: action left right mult}
(A,B)\cdot X:=AXB^{-1}~.
\end{equation}
As a
simple consequence of the Binet Formula, this action preserves the quadratic form $q=-\det$ and thus induces a representation $$\SL(2,\R)\times \SL(2, \R)\to \OO(\V, q)~.$$
Since the center of $\SL(2,\R)$ is $\pm\En$, the kernel of such a representation is given by  $K=\{(\En, \En), (-\En, -\En)\}$, and by a dimensional argument it turns out that the image of the  representation is the connected component of the identity: 
$$\isom_0(\AdS{2,1})\cong\SO_0(\V, q)\cong (\SL(2,\R)\times \SL(2, \R))/K~.$$
Using this model, one then has a natural identification of $\AdSP{2,1}$  with the Lie group $\PSL(2,\R)$, in such a way that
\begin{equation}\label{eq:isom AdS PSL2R}
\isom_0(\AdSP{2,1})\cong\PSL(2,\R)\times \PSL(2,\R)
\end{equation}
acting by left and right multiplication on $\PSL(2,\R)$.

The stabilizer of the identity in $\isom_0(\AdSP{2,1})$ is the diagonal subgroup $\Delta<\PSL(2,\R)\times \PSL(2,\R)$. Under the obvious identification of $\PSL(2,\R)$ and $\Delta$, 
the action of the identity stabilizer on the Lie algebra $\psl(2,\R)=T_{\En}\PSL(2, \R)$ is the adjoint action of $\PSL(2,\R)$. 
A direct consequence of this construction is the bi-invariance of the quadratic form $q$. Indeed, denoting by $q_{\En}$ the restriction of $q$ to $T_{\En}\SL(2,\R)$, a direct computation shows that $q_{\En}$ equals $(1/8)\kappa$, where $\kappa(X,Y)=4\tr(XY)$ is the Killing form of $\psl(2,\R)$.

\begin{remark}\label{rmk sl2R mink}
The Lie algebra $\mathfrak{sl}(2,\R)$ equipped with the quadratic form $q_{\En}$ is then a copy of the $3$-dimensional Minkowski space,
hence the adjoint action yields a representation
 $$\PSL(2,\R)\to \OO(\psl(2,\R),q_{\En})$$
 which in turn induces the well-known isomorphism
 $$\SO_0(2,1)\cong\SO_0(\mathfrak{sl}(2,\R),q_{\En})\cong \PSL(2,\R)~,$$
 which is nothing but the restriction of the isomorphism \eqref{eq:isom AdS PSL2R} to the stabilizer of the identity in the left-hand side $\isom_0(\AdSP{2,1})$, and to the diagonal subgroup $\Delta$ in the right-hand side $\PSL(2,\R)\times \PSL(2,\R)$.
\end{remark}

\begin{remark}The identification between $\AdS{2,1}$ and $\SL(2,\R)$ parallels the more classical identification between the three sphere $S^ 3$ and the Lie group $\mathrm{SU}(2)$.
The analogy can be deepened by considering the isomorphism of $\mathfrak{gl}(2,\R)$ with the algebra of pseudo-quaternions, namely the four-dimensional real algebra generated by $1,i,j,k$ with the relations $-i^ 2=j^ 2=k^2=1$ and $k=ij=-ji$. Under this isomorphism the quadratic form $\det$ corresponds to 
$$q(a+bi+cj+dk)=a^2+b^2-c^2-d^2~,$$ hence $\AdS{2,1}$ is identified to the set of unitary pseudo-quaternions. 
\end{remark}

\subsection{The boundary of $\PSL(2,\R)$} \label{sec:bdy in PSL2R}
From the identification between $\AdSP{2,1}$ and $\PSL(2,\R)$, we obtain an identification of $\partial\AdSP{2,1}$ with the boundary of  $\PSL(2,\R)$ 
into $\mathrm P(\V)$, which is the projectivization of the cone of rank $1$ matrices.
Therefore from now on we shall always consider
\[
   \Q{2,1}=\{[X]\in \mathrm P(\V)\,|\,\mathrm{rank}(X)=1\}~.
\]
We have a homeomorphism
\[
   \Q{2,1}\to\RP^ 1\times\RP^ 1
\]
which is defined by 
$$[X]\mapsto (\im X, \Ker X)~,$$
and is equivariant under the actions of $\PSL(2,\R)\times\PSL(2,\R)$: the obvious action on $\RP^ 1\times\RP^1$, and  the action on $\Q{2,1}$ induced by \eqref{eq: action left right mult}.

\begin{lemma}\label{lemma time rev isometry}
The inversion map $\iota[X]=[X]^{-1}$ is a time-reversing isometry of $\AdSP{2,1}$ which induces the homeomorphism $(x,y)\mapsto(y,x)$ on $\Q{2,1}\cong \RP^1\times\RP^1$.
\end{lemma}
\begin{proof}
Clearly $\iota$ is equivariant with respect to the isomorphism of $\PSL(2,\R)\times \PSL(2,\R)$ which switches the two factors. To show that it is an isometry it thus suffices to check that its differential at the identity is a linear isometry, which is obvious since $d_{\En}\iota$ is minus the identity, which also shows time-reversal.  The second claim is easily checked by observing that for an invertible $2\times 2$ matrix we have $(\det X) X^{-1}=(\tr X)\En-X$ by the Cayley-Hamilton theorem, so that
projectively $[X^{-1}]=[\tr X\En-X]$. This shows that the inversion map of $\AdSP{2,1}$ extends to the transformation $[X]\to[\tr X\En-X]$ along the boundary.
If $X$ is a rank $1$ matrix, then it is traceless if and only if $X^2=0$, that is, if and only if $\mathrm{Ker} X=\im X$. So in this case the statement is easily proved.
If $\tr X\neq 0$, then $X$ is diagonalizable with eigenvalues $0$, and $\tr X$. Moreover $\mathrm{Ker} X$ and $\im X$ are the corresponding eigenspaces.
It is  easily seen that $\mathrm{Ker} (\tr X\En-X)=\im X$ and $\im (\tr X\En-X)=\mathrm{Ker} X$.
\end{proof}

Using the upper half-plane model for the hyperbolic space $\Hyp^2$, $\RP^1$ corresponds to the boundary at infinity $\partial\Hyp^ 2$ and $\PSL(2,\R)$ is identified to $\isom_0(\Hyp^2)$, which acts on $\RP^1$ in the canonical way. One can therefore consider $ \partial\AdSP{2,1}$
as $\partial\Hyp^2\times\partial\Hyp^ 2$. We can interpret the convergence to $ \partial\AdSP{2,1}$ in this setting.

\begin{lemma} \label{lemma convergence at infinity}
A sequence $[X_n]\in \AdSP{2,1}$ converges to $(x,y)\in \Q{2,1}\cong \RP^1\times\RP^1$ if and only if for every $p\in\Hyp^2$, $X_n(p)\to x$ and $X_n^{-1}(p)\to y$.
\end{lemma}
\begin{proof}
Since the action of $\PSL(2,\R)$ on $\Hyp^2$ is isometric, if the condition holds for some $p$, then it holds for all $p\in\Hyp^2$. Hence one can take for instance $p=i$ in the upper half-plane. Assuming $X_n$ converges projectively to a rank 1 matrix $X$, one checks immediately that $X(p)$ is in the projective class of $x=\im(X)$. The convergence $X_n^{-1}(p)\to y$ then follows by Lemma \ref{lemma time rev isometry}.
\end{proof}


In this dimension, $\Quno$ is a double ruled quadric, which in an affine chart looks like in Figure \ref{fig:hyperboloid}. We shall now describe geometrically these rulings. Given any $(x_0,y_0)\in\Quno$, 
$$\lambda_{y_0}:=\{(x,y_0)\,|\,x\in\RP^1\}$$
describes a projective line in $\RP^3$ which is contained in $\Quno$, hence lightlike for the conformal Lorentzian structure of $\Quno$ by Remark \ref{rmk lightlike cone boundary2}. In fact, $\lambda_{y_0}$ is the orbit of $(x_0,y_0)$ by the action of $\PSL(2,\R)\times\{\En\}$, or by the (now free) action of $\mathrm{PSO}(2)\times\{\En\}$, where $\mathrm{PSO}(2)$ corresponds to a 1-parameter elliptic subgroup in $\PSL(2,\R)$. In short, 
$$\lambda_{y_0}=\PSL(2,\R)\cdot (x_0,y_0)=\mathrm{PSO}(2)\cdot (x_0,y_0)~.$$

We refer to $\lambda_{y_0}$ as the \emph{left ruling} through $(x_0,y_0)$, and similarly the \emph{right ruling} is 
$$\mu_{x_0}:=\{(x_0,y)\,|\,y\in\RP^1\}~,$$
for which analogous considerations hold. 



We conclude this section by remarking that the conformal Lorentzian structure on $\Quno$ is easily expressed in terms of the left and right rulings.
Let us  start by carefully choosing a time-orientation on $\AdSP{2,1}$.
Orienting $\RP^1$ in the usual way, consider the induced orientation on $\mathrm{PSO}(2)$.
We remark that $\mathrm{PSO}(2)$ is a timelike geodesic of $\AdSP{2,1}$ and we choose the time orientation on $\AdSP{2,1}$ in such a way that $\mathrm{PSO}(2)$ oriented as above is future directed.
Observe that the action of $\mathrm{PSO}(2)\times\{\En\}$ on $\AdSP{2,1}$ yields a flow on $\AdSP{2,1}$ generated by a right-invariant vector field, which at $\En$ is the \emph{positive} tangent vector of $\mathrm{PSO}(2)$.
So orbits are all timelike and \emph{future directed}. Similarly  $\{\En\}\times\mathrm{PSO}(2)$ yields a flow generated by a left-invariant vector field, which at $\En$ is the \emph{negative} tangent vector of $\mathrm{PSO}(2)$, and its orbits are all timelike and \emph{past directed}. 

\begin{prop} \label{prop conformal class ein}
Let $\pi_l,\pi_r:\RP^1\times\RP^1\to\RP^1$ be the canonical projections and  $d\theta$ the angular form on $\RP^1\cong \partial\Hyp^2$. Then the symmetric product
$\pi_l^*(d\theta)\pi_r^*(d\theta)$
is in the conformal class of $\Quno$.
\end{prop}
\begin{proof}
Since we already know that the left and right rulings are lightlike for the conformal class of $\Quno$, it only remains to check the sign by Remark \ref{rmk lightlike cone boundary}. 
Notice that $\lambda_{y_0}$ is the orbit of the action of  $\mathrm{PSO}(2)\times\{\En\}$, while $\mu_{x_0}$ is the orbit of the action of $\{\En\}\times\mathrm{PSO}(2)$.
Then $\lambda_{y_0}$ with the obvious parameterization is future directed while $\mu_{x_0}$ is past directed. 
The result follows.
\end{proof}

Therefore a $C^1$ curve in $\Quno$ is spacelike when it is  locally the graph of an orientation-preserving function, and timelike when it is 
locally the graph of an orientation-reversing function. Given two intervals $I_1$ and $I_2$ in $\partial\Hyp^2$ and assuming $\theta_1$ and $\theta_2$ are angle
determinantion over $I_1$ and $I_2$, the future $\Ip_{I_1\times I_2}(p_0, q_0)$ of a point $(p_0, q_0)$ in $I_1\times I_2$ is region where $\theta_1(p)-\theta_1(p_0)>0$ and $\theta_2(q)-\theta_2(q_0)<0$, while the past is determined by reversing both inequalities.
In conclusion
\begin{equation}\label{eq:futbord}
\Ip_{I_1\times I_2}(p_0, q_0)\cup\Ipm_{I_1\times I_2}(p_0, q_0)=\{(p,q)\in I_1\times I_2\,|\, (\theta_1(p)-\theta_1(p_0))(\theta_2(q)-\theta_2(q_0))<0\}\,.
\end{equation}


\subsection{Levi-Civita connection}\label{subsec levi civita}

In this section we shall describe the properties of natural metric connections on $\AdSP{2,1}$, for which the theory of Lie groups permits to give a transparent description. Let us start by some general facts of Lie groups.

Recall that the Lie bracket on the Lie algebra $\mathfrak g=T_{\En}G$ of a Lie group $G$ is defined  as 
\begin{equation}\label{eq:lie}
[V,W]_{\mathfrak{g}}=[\widetilde V, \widetilde W](\En)=-[\widetilde V', \widetilde W'](\En)~,
\end{equation}
where $[\cdot,\cdot]$ now denotes the bracket of vector fields and $\widetilde V,\widetilde W$ (resp. $\widetilde V', \widetilde W'$) are the left-invariant (resp. right-invariant) vector fields extending $V$ and $W$ respectively.

Now, any Lie group $G$  is equipped with two natural connections, the \emph{left-invariant connection} $D^l$ and the \emph{right-invariant connection} $D^r$. The former is uniquely determined by the condition that left-invariant vector fields are parallel, and is left-invariant in the sense that, if $L_g:G\to G$ denotes left multiplication by $g$, then 
$$
(L_g)_*(D^l_V W)=D^l_{(L_g)_*(V)}(L_g)_*(W)~.
$$
The left-invariant connection $D^l$ at a point $g\in G$ can be easily expressed as ordinary differentiation in $T_g G$, after pulling-back a vector field $W$ to $g$ by left multiplication. More precisely,
\begin{equation}\label{eq:left-invariant by derivative}
D^l_V W=\left.\frac{d}{dt}\right|_{t=0}(L_{g\gamma(t)^{-1}})_*(W_{\gamma(t)})~,
\end{equation}
where $\gamma(t)$ is a path with $\gamma(0)=g$ and $\gamma'(0)=V$.

The analogous definition and property holds for $D^r$, replacing left-invariant by right-invariant vector fields. Both connections $D^l$ and $D^r$ are flat and are compatible with any metric on $G$ which is left-invariant or right-invariant respectively. Indeed parallel transport of a vector $W\in T_g G$ to $T_{g'}G$ consists just in left (resp. right) multiplication, namely in applying $(L_{g'g^{-1}})_*$ (resp. $(R_{g'g^{-1}})_*$) to $W$, and is therefore path-independent.

But $D^l$ and $D^r$ are not torsion-free, as can be easily checked by the definition of torsion, which we recall is a tensor of type $(2,1)$. For instance, computing at the identity and using left-invariant extensions $\widetilde V$ and $\widetilde W$ of $V$ and $W$, one obtains 
$$\tau^l(V,W)=D^l_V \widetilde W-D^l_W \widetilde V-[\widetilde V,\widetilde W](\En)=-[\widetilde V,\widetilde W](\En)=-[V,W]_{\mathfrak g}~.$$
Similarly one obtains
$$\tau^r(V,W)=[V,W]_{\mathfrak g}~.$$
By construction, $\tau^l$ is left-invariant and $\tau^r$ is right-invariant. But by $\mathrm{Ad}$-invariance of the Lie bracket of $\mathfrak g$, the torsions $\tau^l$ and $\tau^r$ are actually bi-invariant.

Moreover, a direct computation shows that the tensorial quantity $D^r-D^l$ admits the following expression at the identity:
\begin{equation}\label{eq:difference left right connections}
D^r_V W-D^l_V W=[V,W]_{\mathfrak g}~.
\end{equation}
To check Equation \eqref{eq:difference left right connections}, it suffices to consider the right-invariant extension $\widetilde W$ of  $W$, so that $D^r_V \widetilde W=0$. Using the expression \eqref{eq:left-invariant by derivative} for $D^l$ at the identity, we see that
\begin{align*}
D^l_V \widetilde W&=\left.\frac{d}{dt}\right|_{t=0}(L_{\exp(-tV)})_*(\widetilde W_{\exp(tV)})\\
&=\left.\frac{d}{dt}\right|_{t=0}(L_{\exp(-tV)})_*(R_{\exp(tV)})_*(W)=-\mathrm{ad}_V(W)=-[V,W]_\mathfrak g~,
\end{align*}
which thus shows Equation \eqref{eq:difference left right connections}.

Now, given a bi-invariant pseudo-Riemannian metric on $G$, its Levi-Civita connection $
\nabla$ can be expressed as the mid-point between $D^l$ and $D^r$. Namely, using now $V$ and $W$ to denote vector fields,
\begin{equation}\label{eq:levi civita midpoint}
\nabla_V W=\frac{1}{2}\left(D^l_V W+D^r_V W\right)~,
\end{equation}
which is still a connection on $G$ since the space of connections forms an affine space with underlying vector space the space of $(2,1)$-tensors. Indeed $\nabla$ is still compatible with the metric and is moreover torsion-free, since its torsion, which equals
$(\tau^l+\tau^r)/2$, vanishes.

A direct consequence of Equations \eqref{eq:difference left right connections} and \eqref{eq:levi civita midpoint} is the following well-known expression for the Levi-Civita connection in terms of  left-invariant vector fields:
\begin{lemma}\label{lemma exp map lie group}
Given left-invariant vector fields $V$ and $W$ on $G$, the Levi-Civita connection of a bi-invariant metric has the expression:
$$\nabla_V W=\frac{1}{2}[V,W]~.$$
In particular, the Lie group exponential map coincides with the pseudo-Riemannian exponential map.
\end{lemma}
\begin{proof}
The first part of the statement follows from Equations \eqref{eq:difference left right connections} and \eqref{eq:levi civita midpoint}, since for left-invariant vector fields $D_V^lW=0$. The second part is a direct consequence, since 1-parameter groups $\gamma:I\to G$ integrate left-invariant vector fields, and therefore $\nabla_{\dot\gamma}\dot\gamma=0$.
\end{proof}

\subsection{Lorentzian cross-product}
Before a discussion on geodesics in the $\PSL(2,\R)$-model, which will rely on the Lie group generalities of the previous section, we discuss here some particular features of the Lie group $G=\PSL(2,\R)$. Namely, we have a natural Lorentzian cross product, that is a $T\AdSP{2,1}$-valued 2-form $(V,W)\mapsto V\boxtimes W$, which is defined by the equality
\begin{equation}\label{eq:defi cross product}
\langle V\boxtimes W,U\rangle=\Omega(V,W,U)~,
\end{equation}
where $\langle\cdot,\cdot\rangle$ is the Anti-de Sitter metric and $\Omega$ is the associated volume form, namely the unique 3-form taking the value 1 on any positive oriented orthonormal basis. Here we orient $\PSL(2,\R)$ by declaring that the orthonormal basis $$V=\begin{pmatrix} 0 & 1 \\ 1 & 0 \end{pmatrix}
\qquad
W=\begin{pmatrix} 1 & 0 \\ 0 & -1 \end{pmatrix}
\qquad
U=\begin{pmatrix} 0 & -1 \\ 1 & 0 \end{pmatrix}
$$
at the identity is positive.
In other words, $V\boxtimes W$ equals $*(X\wedge Y)$, where $*:\Lambda^2T\AdSP{2,1}\to T\AdSP{2,1}$ is the Hodge star operator defined similarly to the Riemannian case.

At the identity, a very simple equality holds for the Lorentzian cross product and the Lie bracket of $\mathfrak g$:

\begin{lemma} \label{lemma cross product lie bracket}
Given $V,W\in T_{\En}\PSL(2,\R)$, $[V,W]_{\mathfrak{g}}=-2V\boxtimes W$.
\end{lemma}
\begin{proof}
We claim that the volume form of the Anti-de Sitter metric equals:
\begin{equation}\label{eq:volume form lie bracket}
\Omega(V,W,U)=-\frac{1}{2}\langle [V,W]_{\mathfrak{g}},U\rangle~.
\end{equation}
The stated equality then follows from Equation \eqref{eq:defi cross product}. To see the claim, first let us observe that the expression in \eqref{eq:volume form lie bracket} is an alternating three-form, as a consequence of the skew-symmetry of the Lie bracket and of the (infinitesimal version of) $\mathrm{Ad}$-invariance of the Anti-de Sitter metric, namely:
\begin{equation}\label{eq:ad invariance}
\langle [V,W]_{\mathfrak{g}},U\rangle=-\langle W,[V,U]_{\mathfrak{g}}\rangle~.
\end{equation}
Hence $\Omega$ is a multiple of the volume form. To check the multiplicative factor, by left-invariance, it suffices to perform the computation at $T_{\En}\AdSP{2,1}=\psl(2,\R)$ on the positive orthonormal basis
$$V=\begin{pmatrix} 0 & 1 \\ 1 & 0 \end{pmatrix}
\qquad
W=\begin{pmatrix} 1 & 0 \\ 0 & -1 \end{pmatrix}
\qquad
U=\begin{pmatrix} 0 & -1 \\ 1 & 0 \end{pmatrix}
$$
for which $V,W$ are spacelike and $U$ is timelike.
The equality follows since $[V,W]_{\mathfrak{g}}=2U$.
\end{proof}
Lemma \ref{lemma cross product lie bracket} permits to rewrite the expression for the Levi-Civita connection of left-invariant vector fields, from Lemma \ref{lemma exp map lie group}, simply as $\nabla_V W=-V\boxtimes W$ and, together with Equations \eqref{eq:levi civita midpoint} and \eqref{eq:difference left right connections}, to obtain the following general expression for the Levi-Civita connection. 
\begin{equation}\label{eq:levi-civita general}
\nabla_V W=D^l_V W-V\boxtimes W=D^r_V W+V\boxtimes W~.
\end{equation}

\begin{remark}
Using the set-up of this section, one easily gets another computation of the curvature of $\AdSP{2,1}$, different from that given in Section \ref{subsec:quadric}. 
Fix $V,W,U\in \mathfrak g=T_{\En}\PSL(2,\R)$, and denote by $\widetilde V, \widetilde W, \widetilde U$ the left invariant extensions of $V,W, U$.
From Lemma \ref{lemma exp map lie group} and the Jacobi identity, one gets the following expression for the Riemann tensor:
\begin{align*}
R(V,W)U&=\left(\nabla_{\widetilde V}\nabla_{\widetilde W}\widetilde U-\nabla_{\widetilde W}\nabla_{\widetilde V} \widetilde U-\nabla_{[\widetilde V,\widetilde W]}\widetilde U\right)(\En) \\
&=\left(\frac{1}{4}[\widetilde V,[\widetilde W,\widetilde U]]-\frac{1}{4}[\widetilde W,[\widetilde V,\widetilde U]]-\frac{1}{2}[[\widetilde V,\widetilde W],\widetilde U]\right)(\En)\\
&=\frac{1}{4}[V,[W, U]_{\mathfrak{g}}]_{\mathfrak{g}}-\frac{1}{4}[ W,[ V,U]_{\mathfrak{g}}]_{\mathfrak{g}}-\frac{1}{2}[[V,W]_{\mathfrak{g}},U]_{\mathfrak{g}}=\frac{1}{4}[U,[V,W]_{\mathfrak{g}}]_{\mathfrak{g}}~.
\end{align*}
Hence from Lemma \ref{lemma cross product lie bracket} and Equation \eqref{eq:ad invariance}:
$$\langle R(V,W)W,V\rangle=\frac{1}{4}\langle [W,[V,W]_{\mathfrak{g}}]_{\mathfrak{g}},V\rangle=\frac{1}{4}\langle [V,W]_{\mathfrak{g}},[V,W]_{\mathfrak{g}}\rangle=\langle V\boxtimes W,V\boxtimes W\rangle=-1~,$$
for $V,W$ orthonormal spacelike vectors, hence spanning a spacelike plane. An analogous computation holds for timelike planes, thus showing that the sectional curvature is identically $-1$.\end{remark}

\subsection{Geodesics in $\PSL(2,\R)$} \label{sec:geodesics PSL2R}

In this section we will describe the geodesics of the $\PSL(2,\R)$-model, applying its Lie group structure. 

\subsubsection*{Exponential map} Let us start by understanding the geodesics through the identity. 
Recalling Remark \ref{rmk sl2R mink}, the Lie algebra of $\PSL(2,\R)$ is isometrically identified to a copy of Minkowski space, where under such an isometry the stabilizer of a point (namely $\PSL(2,\R)$ acting by means of the adjoint action) corresponds to the group of linear isometries of Minkowski space. In short, this means that we shall distinguish geodesics by their type (timelike, spacelike, lightlike) and those will be equivalent under this action. Moreover, by Lemma \ref{lemma exp map lie group} it suffices to understand the one-parameter groups for the Lie group structure of $\PSL(2,\R)$. We immediately get the following:

\begin{itemize}
\item Timelike geodesics are, up to conjugacy, of the form
$$\begin{pmatrix}
\cos(t) & -\sin(t) \\
\sin(t) & \cos(t)
\end{pmatrix}$$
namely, under the identification of $\PSL(2,\R)$ with $\isom(\Hyp^2)$, they are elliptic one-parameter groups fixing a point in $\Hyp^2$. In this example, the tangent vector is the matrix
$$\begin{pmatrix}
0 & -1 \\
1 & 0
\end{pmatrix}~.$$
These are in fact closed geodesics, parameterized by arclength, of total  length $\pi$.

\item Spacelike geodesics are, again up to conjugacy:
$$\begin{pmatrix}
\cosh(t) & \sinh(t) \\
\sinh(t) & \cosh(t)
\end{pmatrix}$$
with initial velocity 
$$\begin{pmatrix}
0 & 1 \\
1 & 0
\end{pmatrix}~.$$
In terms of hyperbolic geometry, these are hyperbolic one-parameter groups, fixing two points in the boundary of $\Hyp^2$ (in this case, $\pm 1$). 

\item Finally, lightlike geodesics are the parabolic one-parameter groups conjugate to 
$$\begin{pmatrix}
1 & t \\
0 & 1
\end{pmatrix}~,$$
whose initial vector has indeed zero length.
\end{itemize}

\subsubsection*{A totally geodesic spacelike plane} 
Using the above description of timelike geodesics through $\En$, we can also interpret the duality of Section \ref{sec:duality} in terms of the structure of $\PSL(2,\R)$. Recalling that the dual plane of a point $A$ is the set of antipodal points along timelike geodesics through $A$, one sees that the dual plane of $\En$ consists of elliptic isometries of $\Hyp^2$ which rotate by an angle $\pi$. Equivalently, this is the set of (projective classes) of traceless matrices, that is (by the Cayley-Hamilton theorem)  
$$P_\En=\{[J]\in \PSL(2, \R)\,|\, J^2=-\En\}~.$$
In other words, $P_\En$ is identified with the space of \emph{linear almost-complex structures} on $\R^2$, up to sign reversing.  The boundary at infinity of $P_\En$ is made of traceless matrices of rank $1$, that is, the projectivization of the set of nilpotent $2\times 2$ matrices. 

Recall that the stabilizer of $\En$ is the diagonal subgroup of $\PSL(2,\R)\times\PSL(2,\R)$, and it also acts on the dual plane $P_\En$ by conjugation.
The following statement is then straightforward:

\begin{lemma} \label{rmk dual plane traceless2}
The map from $\Hyp^2$ to $P_\En$, sending $p\in\Hyp^2$ to the elliptic order-two element in $\PSL(2,\R)$ fixing $p$, is a $\PSL(2,\R)$-equivariant isometry.
\end{lemma}
\begin{proof}
Equivariance with respect to the actions of $\PSL(2,\R)$ is easy since, for an element $X\in\PSL(2,\R)$, the order-two elliptic element fixing $X\cdot p$ is precisely the $X$-conjugate of the order-two elliptic element fixing $p$. Using the equivariance, it follows that the pull-back of the metric of $P_\En$ is a constant multiple of the hyperbolic metric of $\Hyp^2$. Since both have curvature $-1$, they must coincide.
\end{proof}

On the double cover $\AdS{2,1}$, which is the $\SL(2,\R)$-model, $P_\En$ lifts to the two planes $P_\En^\pm$ dual to the identity. One of them consists of the matrices $J$ such that $J^2=-\En$, namely the linear almost-complex structures on $\R^2$, which are compatible with the standard orientation of $\R^2$; the other contains the linear almost-complex structures on $\R^2$ compatible with the opposite orientation of $\R^2$.

\subsubsection*{Timelike geodesics} To get a complete description of timelike geodesics (not only those through the identity) it suffices to let (the identity component of) the isometry group of $\AdSP{2,1}$, namely $\PSL(2,\R)\times\PSL(2,\R)$ act on $\PSL(2,\R)$ by left and right multiplication. In particular an important description of the space of timelike geodesics of $\AdSP{2,1}$ (which is also the space of timelike geodesics of each finite-index cover of $\AdSP{2,1}$) can be obtained, see \cite{MR3888623}. 

\begin{prop} \label{prop space of time geo}
There is a homeomorphism between the space of (unparameterized) timelike geodesics of $\AdSP{2,1}$ and $\Hyp^2\times\Hyp^2$. The homeomorphism is equivariant for the action of $\isom_0(\AdSP{2,1})\cong \PSL(2,\R)\times\PSL(2,\R)$.
\end{prop}
\begin{proof}
The homeomorphism is defined as follows. Given $(p,q)\in\Hyp^2\times\Hyp^2$, we associate to it the subset 
$$L_{p,q}=\{X\in\PSL(2,\R)\,|\,X\cdot q=p\}~.$$
By the previous discussion, geodesics through the identity are precisely of the form $L_{p,p}$ for some $p\in\Hyp^2$. It is easy to check that the map $(p,q)\mapsto L_{p,q}$ is equivariant for the natural actions of $\PSL(2,\R)\times\PSL(2,\R)$, namely $(A,B)\cdot L_{p,q}=L_{A\cdot p,B\cdot q}$, which also implies that $L_{p,q}$ is an unparameterized geodesic and that all unparameterized geodesics are of this form, namely the map we defined is surjective. It remains to see the injectivity: if $L_{p,q}=L_{p',q'}$ for $(p,q)\neq (p',q')$ then in particular there exists an isometry of $\Hyp^2$ sending $p$ to $q$ and $p'$ to $q'$. But such an isometry is necessarily unique since the identity is the only isometry of $\Hyp^2$ fixing two different points. This gives a contradiction and thus concludes the proof.
\end{proof}

\subsubsection*{Spacelike geodesics}
Let us conclude this section by an analysis of spacelike geodesics. Let us consider an oriented geodesic $\ell$ of $\Hyp^2$. From the discussion at the beginning of this section, the one-parameter group of hyperbolic transformations fixing $\ell$ \emph{as an oriented geodesic} constitutes a spacelike geodesic through the origin. By an argument very similar to Proposition \ref{prop space of time geo}, relying on the equivariance of the construction by the actions of $\PSL(2,\R)\times\PSL(2,\R)$, one then proves that every spacelike geodesic is of the form:
$$L_{\ell,\jmath}=\{X\in\PSL(2,\R)\,|\,X\cdot \jmath=\ell\text{ as oriented geodesics}\}~,$$
where $\ell$ and $\jmath$ denote oriented geodesics of $\Hyp^2$. 
We remark that every (unparameterized, unoriented) spacelike geodesic can be expressed in the above form in two ways, as one can change the orientation of both $\ell$ and $\jmath$. Every such choice corresponds to a choice of orientation for the spacelike geodesic. In other words, one can show:

\begin{prop} \label{prop space of space geo}
There is a homeomorphism between the space of (unparameterized) oriented spacelike geodesics of $\AdSP{2,1}$ and the product of two copies of $\partial\Hyp^2\times\partial\Hyp^2\setminus \Delta$, the space of oriented geodesics of $\Hyp^2$. The homeomorphism is equivariant for the action of $\isom_0(\AdSP{2,1})\cong \PSL(2,\R)\times\PSL(2,\R)$.
\end{prop}

However, for our purpose, we will mostly deal with \emph{unoriented} geodesics, hence we will have  $L_{\ell,\jmath}=L_{\ell',\jmath'}$ where $\ell'$ equals $\ell$ but endowed with  the opposite orientation. Given a spacelike geodesic, there is a natural notion of \emph{dual} spacelike geodesic, which is defined using the projective duality between points and planes from Section \ref{sec:duality}:

\begin{defi} \label{def:dual geodesics}
Given a spacelike geodesic $L_{\ell,\jmath}$ in $\AdSP{2,1}$, the \emph{dual spacelike geodesic} is the intersection of all spacelike planes dual to points of $L_{\ell,\jmath}$.
\end{defi}

The construction of the dual geodesic is involutive. Let us now see an explicit example. For the case of the geodesic $L_{\ell,\ell}$ through the origin, which consists of the one-parameter hyperbolic group of $\PSL(2,\R)$ translating along $\ell$, it can be checked that the dual geodesic consists of all elliptic order-two elements (hence contained in $P_\En$, as it is expected from the definition) whose fixed point lies in $\ell$. In other words, the dual spacelike geodesic of $L_{\ell,\ell}$ is $L_{\ell,\ell'}$. 

We can easily describe the points at infinity in $\Quno$ of these geodesics. Using Lemma \ref{lemma convergence at infinity}, if $x$ and $y$ are the endpoints at infinity of $\ell$ in $\partial\Hyp^2$, then clearly any sequence diverging  towards an end of $L_{\ell,\ell}\subset\PSL(2,\R)$ maps an interior point towards $x$, and the sequence of inverses towards $y$ (up to switching $x$ and $y$). In other words, under the identification of $\Quno$ with $\RP^1\times\RP^1$ (Section \ref{sec:bdy in PSL2R}), the endpoints of $L_{\ell,\ell}$ are $(x,y)$ and $(y,x)$. A similar argument applied to $L_{\ell,\ell'}$, which consists of order-two elliptic elements with fixed point in $\ell$, shows that its endpoints are $(x,x)$ and $(y,y)$. 

Recalling the descriptions of the left and right rulings of $\Quno$, we conclude that the endpoints of a spacelike geodesic and its dual are mutually connected by lightlike segments in $\Quno$. See also Figure \ref{fig:duallines} in Section \ref{sec:GH AdS mfds}, where this configuration is studied and applied more deeply.

By naturality of the construction with respect to the action of $\PSL(2,\R)\times \PSL(2,\R)$, one has:

\begin{prop} \label{prop:dual geodesics}
Given a spacelike geodesic $L_{\ell,\jmath}$ of $\AdSP{2,1}$, its endpoints in $\Quno$ are $(x_1,y_2)$ and $(y_1,x_2)$, where $x_1$ and $y_1$ are the final and initial endpoints of $\ell$  in $\partial\Hyp^2$, and $x_2$ and $y_2$ are the final and initial endpoints of $\jmath$ (where final and initial refers to the orientation of $\ell$ and $\jmath$). The dual geodesic is $L_{\ell,\jmath'}$ and has endpoints $(x_1,x_2)$ and $(y_1,y_2)$. 
\end{prop}

{\large {\part{The seminal work of Mess}\label{part2}}} 
The aim of this part is to describe Mess' work, including the classification of maximal globally hyperbolic spacetimes with compact Cauchy surface and the Gauss map of spacelike surfaces.
The material is organized in the following way. Chapter \ref{ch:causal convex} analyses various properties of causality and convexity in Anti-de Sitter space, which are preliminary to the proof of Mess' classification result. The latter is given in Chapter \ref{sec:GH AdS mfds}. In Chapter \ref{ch:gauss} we then treat the Gauss map and its first properties, and discuss Mess' proof of the Earthquake Theorem. 

\vspace{0.3cm}
\section{Causality and convexity properties} \label{ch:causal convex}

Here we will first study achronal sets in the conformal compactification of Anti-de Sitter space, a notion that makes sense in the universal cover $\AdSU{2,1}$, and then adapt the notion for subsets of $\AdSP{2,1}$. Then we introduce the fundamental notions of invisible domain and of domain of dependence, and describe their properties. 

\subsection{Achronal and acausal sets}\label{sec:part2-achronal}
Let us begin with the first definitions.

\begin{defi}\label{defi achronal acausal}
A subset $X$ of $\AdSU{2,1}\cup\partial\AdSU{2,1}$ is \emph{achronal} (resp. \emph{acausal}) if no pair of points in $X$ is connected by timelike (resp. causal)  lines  in $\AdSU{2,1}$.
\end{defi}

Since acausality and achronality are conformally invariant notions, it will be often convenient to consider the metric $g_{\mathbb S^{2}}-dt^2$ on $\D\times\R$ we introduced in \eqref{poinc:eq conformal} (for $g_{\mathbb S^{2}}$ the hemispherical metric on the disc), which is conformal to the Poincar\'e model of $\AdSU{2,1}$. 

\begin{lemma} \label{lemma achronal is graph}
A subset $X$ of $\AdSU{2,1}\cup\partial\AdSU{2,1}$ is achronal (resp. acausal) if and only if it is the graph of a function $\mathsf f:\mathsf D \to\mathbb R$ that is $1$-Lipschitz (resp. strictly $1$-Lipschitz) 
with respect to the distance induced by the hemispherical metric $g_{\Sp^2}$. 
\end{lemma}
Clearly here $\mathsf D=\pi_{\disk}(X)$ denotes the projection of $X$ to the $\D$ factor.
\begin{proof}
Assume that $X$ is an achronal subset.
Since vertical lines in the Poincar\'e model are timelike,  the restriction of the projection $\pi_{\disk}:\D\times\R\to\D$ to $X$ is injective. So $X$ can be regarded as the graph of a function $\mathsf f:\mathsf D\to\mathbb R$.
Imposing that $(\mathsf x,\mathsf f(\mathsf x))$ and  $(\mathsf y,\mathsf f(\mathsf y))$ are not related by a timelike curve we deduce that 
\begin{equation}\label{eq:lipschitz achronal}
|\mathsf f(\mathsf x)-\mathsf f(\mathsf y)|\leq d_{\Sp^2}(\mathsf x,\mathsf y)~,
\end{equation}
 where $d_{S^2}$ is the hemispherical distance (see also Section \ref{sec:geodesics}).
The same argument shows that conversely the graph of a $1$-Lipschitz function defined on some subset of $\disk$ is achronal.

Moreover, two points $(\mathsf x,\mathsf t)$ and $(\mathsf y,\mathsf s)$ are on the same lightlike geodesic if and only if $|\mathsf t-\mathsf s|=d_{\Sp^2}(\mathsf x,\mathsf y)$. Hence $X$ is acausal if and only if the inequality in \eqref{eq:lipschitz achronal} is strict.
\end{proof}

Observe that a 1-Lipschitz function on a region $\mathsf D\subset\D$ extends uniquely to the boundary of $\mathsf D$. As a simple consequence of the previous lemma, we thus have:

\begin{lemma} \label{lemma global graph}
An achronal subset $X$ in $\AdSP{2,1}$ is properly embedded if and only if it is a global graph over $\D$, and in this case it extends uniquely to the global graph of a 1-Lipschitz function over $\D\cup\partial\D$. 
\end{lemma}

In light of Lemma \ref{lemma global graph}, in the following we will refer to an \emph{achronal surface} as an achronal subset $X$ in $\AdSP{2,1}$ which is the graph of a 1-Lipschitz function defined on a domain in $\D$.

Before studying more detailed properties, we shall remark that {achronality and acausality are} global conditions. Let us first recall the definition of spacelike surface:

\begin{defi}
Given a surface $S$ and a Lorentzian manifold $(M,g)$, a $C^1$ immersion $\sigma:S\to M$ is \emph{spacelike} if the pull-back metric $\sigma^*g$ is a Riemannian metric. If $\sigma$ is an embedding, we refer to its image as a \emph{spacelike surface}.
\end{defi}

A spacelike surface $S$  is locally acausal (in the sense that any point admits a neighborhood in $S$ that is acausal), but there are  examples
of spacelike surfaces which are not achronal (hence a fortiori not acausal), a fact which highlights the global character of Definition \ref{defi achronal acausal}. On the other hand, we have this global result.

\begin{lemma}\label{lem:proper embedding}
Any  properly embedded spacelike surface in $\AdSU{2,1}$ is acausal.
\end{lemma}
\begin{proof}
By Lemma \ref{lemma global graph}, any properly embedded spacelike surface $S$ in $\AdSU{2,1}$ disconnects the space in two regions $U$ and $V$, whose common boundary is $S$, and we can assume that the outward pointing normal from $U$ (resp. $V$)  is past-directed (resp. future directed).
It then turns out that any future oriented causal path that meets $S$ passes from $V$ towards $U$.
This implies that any causal path  meets $S$ at most once.
\end{proof}

Recall from Theorem \ref{thm:geo conf inv} that unparameterized lightlike geodesics only depend on the conformal class of the Lorentzian metric, hence in the following we will simply refer to lightlike geodesics in $\AdSU{2,1}$, although we very often use the conformal metric \eqref{poinc:eq conformal}. 

\begin{lemma}\label{lem:llike}
Let $S$ be a properly embedded achronal surface of $\AdSU{2,1}\cup\partial\AdSU{2,1}$ and assume that a lightlike geodesic segment $\gamma$ joins two points of $S$. Then $\gamma$ is entirely contained in $S$. 
\end{lemma}
\begin{proof}
Recalling Lemma \ref{lemma global graph}, let $\mathsf f^S:\overline{\disk}\to\mathbb R$ be the  function defining $S$, which is $1$-Lipschitz with respect to the hemispherical metric.
Now if $\gamma$ joins $(\mathsf x, \mathsf f^S(\mathsf x))$ to $(\mathsf y, \mathsf f^S(\mathsf y))$, then (up to switching the role of $\mathsf x$ and $\mathsf y$) $\mathsf f^S(\mathsf y)=\mathsf f^S(\mathsf x)+d_{\Sp^2}(\mathsf x,\mathsf y)$. Moreover $\gamma$ consists of points of the form $(\mathsf z, \mathsf f^S(\mathsf x)+d_{\Sp^2}(\mathsf x,\mathsf z))$, for $\mathsf z$ lying on the $g_{\Sp^2}$-geodesic segment joining $\mathsf x$ to $\mathsf y$.
For such a point $\mathsf z$ on the geodesic segment joining $\mathsf x$ to $\mathsf y$, by achronality of $S$ we have:
\[
      \mathsf f^S(\mathsf z)-\mathsf f^S(\mathsf x)\leq d_{\Sp^2}(\mathsf x,\mathsf z)\qquad\text{and}\qquad \mathsf f^S(\mathsf y)-\mathsf f^S(\mathsf z)\leq d_{\Sp^2}(\mathsf z,\mathsf y)=d_{\Sp^2}(\mathsf x,\mathsf y)-d_{\Sp^2}(\mathsf x,\mathsf z)~.
\] 
But the second inequality implies that $\mathsf f^S(\mathsf z)\geq \mathsf f^S(\mathsf x)+d_{\Sp^2}(\mathsf x,\mathsf z)$ so we conclude that $\mathsf f^S(\mathsf z)=\mathsf f^S(\mathsf x)+d_{\Sp^2}(\mathsf z,\mathsf x)$, proving that $\gamma$ is contained in $S$.
\end{proof}

{Given a function $\mathsf f:\overline\D\to\R$, we define its oscillation as 
$$\osc(\mathsf f):=\max_{\mathsf y\in\overline{\disk}}\mathsf f(\mathsf y)-\min_{\mathsf y\in\overline{\disk}}\mathsf f(\mathsf y)~.$$
It is important to stress that this quantity is not invariant under the isometry group of $\AdSU{2,1}$.}

\begin{lemma}\label{lem:osc}
Let $S$ be a properly embedded achronal surface, defined as the graph of $\mathsf f^S:\overline{\disk}\to\mathbb R$.
Then $\osc(\mathsf f^S)\leq \pi$. Moreover $\osc(\mathsf f^S)=\pi$ if and only if $S$ is a lightlike plane.
\end{lemma}
\begin{proof}
As $\mathsf f^S$ is $1$-Lipschitz for the hemispherical metric, and the diameter of $\disk$ for $g_{\Sp^2}$ is $\pi$ we easily see that $\osc(\mathsf f^S)$ is bounded by $\pi$.
Moreover if the value $\pi$ is attained it follows that there are two antipodal points $\mathsf y,\mathsf y'\in\partial\disk$ such that $\mathsf f^S(\mathsf y')=\mathsf f^S(\mathsf y)+\pi$. 
Recall from Section \ref{sec:geodesics} (see also Figure \ref{fig:causality}) that the lightlike plane with past and future points $(\mathsf y,\mathsf f^S(\mathsf y))$ and $(\mathsf y',\mathsf f^S(\mathsf y)+\pi)$ is
$$P=\{(\mathsf x,\mathsf t)\,|\,\mathsf t=\mathsf f^S(\mathsf y)+d_{\Sp^2}(\mathsf x,\mathsf y)\}$$
and is foliated by lightlike geodesics joining $(\mathsf y, \mathsf f^S(\mathsf y))$ to $(\mathsf y', \mathsf f^S(\mathsf y)+\pi)$. By Lemma \ref{lem:llike}, $P$ is included in $S$. Since both $P$ and $S$ are global graphs over $\overline\D$, $S=P$.
\end{proof}

\subsection{Invisible domains} \label{subsec:invisible}
The first part of this subsection will be devoted to the definition and first properties of invisible domains, which was first given in \cite{MR2369412}, and in the last part we will focus on the case that $X$ is a subset of $\partial\AdSU{2,1}$.

\begin{defi} 
Given an achronal domain $X$ in $\AdSU{2,1}\cup\partial\AdSU{2,1}$, the \emph{invisible domain} of $X$ is the subset $\Omega(X)\subset\AdSU{2,1}$ of points which are connected to $X$ by no causal path.
\end{defi}

Recall that by McShane's Theorem (\cite{mcshane}) any $1$-Lipschitz function on a  subset of a metric space admits a $1$-Lipschitz extension everywhere. Hence any achronal set $X$, which by Lemma \ref{lemma achronal is graph} is the graph of a 1-Lipschitz function $\mathsf f^X:\mathsf D\to\R$ for $\mathsf D=\pi_\D(X)$, is a subset of a properly embedded achronal surface.  

Here we introduce two particular extensions $\mathsf f^X_\pm:\D\cup\partial\D$, to which we sometimes refer as the \emph{extremal} extensions:
\[
 \mathsf f_-^X(\mathsf y)=\sup\{\mathsf f^X(\mathsf x)-d_{\Sp^2}(\mathsf x,\mathsf y)\,|\, \mathsf x\in \pi_\disk(X)\}\qquad \mathsf f_+^X(\mathsf y)=\inf\{\mathsf f^X(\mathsf x)+d_{\Sp^2}(\mathsf x,\mathsf y)\,|\, \mathsf x\in \pi_\disk(X)\}~.
\]
Clearly $\mathsf f^X_\pm$ coincide with $\mathsf f^X$ on $X$ and are 1-Lipschitz.

\begin{lemma}\label{lem:ext}
Let $X$ be any 
 closed achronal subset $X$ of $\AdSU{2,1}\cup\partial\AdSU{2,1}$ and let $S_\pm(X)$ be the graphs of the extremal extensions $\mathsf f^X_\pm$. 

\begin{enumerate}
\item The properly embedded surfaces $S_-(X)$ and $S_+(X)$ are achronal with $S_-(X)\subset\overline{\past(S_+(X))}$, and $\Omega(X)=\fut(S_-(X))\cap\past(S_+(X))$.

\item Every achronal subset containing $X$ is contained in $S_-(X)\cup\Omega(X)\cup S_+(X)$.

\item Every point of $S_\pm(X)$ is connected to $X$ by at least one lightlike geodesic segment, which is entirely contained in $S_\pm(X)$. 
Finally, $S_+(X)\cap S_-(X)$ is the union of $X$ and all lightlike geodesic segments joining points of $X$.
\end{enumerate}
\end{lemma}

\begin{proof}
Let us first show that $S_-(X)\subset\overline{\past(S_+(X))}$. Given a point $(\mathsf y,\mathsf t)$, $\mathsf t\leq \mathsf f^X_+(\mathsf y)$ if and only if $\mathsf t\leq \mathsf f^X(\mathsf x)+d_{S^2}(\mathsf x,\mathsf y)$ for every $\mathsf x\in \pi_\disk(X)$, that is, if and only if 
$(\mathsf y,\mathsf t)$ lies outside $\fut(X)$. Similarly $(\mathsf y,\mathsf t)$ lies outside $\past(X)$ if and only if $\mathsf t\geq \mathsf f^X_-(\mathsf y)$. 
 By achronality, $S_+(X)$ does not meet the past of $X$, so
we deduce that $\mathsf f^X_+(\mathsf y)\geq \mathsf f^X_-(\mathsf y)$ for all $\mathsf y\in\overline{\disk}$, that is, $S_-(X)$ is contained in  $\overline{\past(S_+(X))}$.

As a similar observation, given a point $(\mathsf y,\mathsf t)$, $\{(\mathsf y,\mathsf t)\}\cup X$ is achronal if and only if $\mathsf f^X_-(\mathsf y)\leq \mathsf t\leq \mathsf f^X_+(\mathsf y)$.
Moreover $(\mathsf y,\mathsf t)$ is connected to $X$ by no causal curve  if and only if $\mathsf f^X_-(\mathsf y)<\mathsf t<\mathsf f^X_+(\mathsf y)$. This shows that
$$\Omega(X)=\{(\mathsf y,\mathsf t)\,|\, \mathsf f^X_-(\mathsf y)<\mathsf t<\mathsf f^X_+(\mathsf y)\}~,$$
and also the second item, by applying the previous observation to any point of an achronal set containing $X$ which is not in $X$ itself.



To prove the third item, fix a point in $(\mathsf y,\mathsf t)\in S_+(X)$. As we are assuming that $X$ is closed in $\AdSU{2,1}\cup\partial\AdSU{2,1}$, the fact that $\mathsf f^X$ is $1$-Lipschitz implies that $\pi_\disk(X)$ is closed in $\disk\cup\partial\disk$, so it is compact.
In particular there exists $\mathsf x\in\partial\disk$ such that $\mathsf t=\mathsf f^X_+(\mathsf y)=\mathsf f^X(\mathsf x)+d_{S^2}(\mathsf x,\mathsf y)$.  Thus $(\mathsf y,\mathsf t)$ is connected to $(\mathsf x,\mathsf f^X(\mathsf x))$ by a lightlike geodesic segment. By Lemma \ref{lem:llike}, this geodesic segment is entirely contained in $S_+(X)$. Clearly the proof for $S_-(X)$ is analogous.

It remains to compute $S_-(X)\cap S_+(X)$. For this purpose, notice that if two points of $X$ are connected by a {lightlike} geodesic segment $\gamma$, applying Lemma \ref{lem:llike}  we deduce that $\gamma\subset S_-(X)\cap S_+(X)$.
Conversely let $(\mathsf y,\mathsf t) \in S_-(X)\cap S_+(X)$  so that $\mathsf f^X_-(\mathsf y)=\mathsf f^X_+(\mathsf y)$.
There exist $\mathsf x$ and $\mathsf x'$ in $\pi_\disk(X)$ such that
\[
  \mathsf f^X_+(\mathsf y)=\mathsf f^X(\mathsf x)+d_{\Sp^2}(\mathsf x,\mathsf y)\qquad\text{and}\qquad \mathsf f^X_-(\mathsf y)=\mathsf f^X(\mathsf x')-d_{\Sp^2}(\mathsf x',\mathsf y)
\]
Using that $\mathsf f^X_-(\mathsf y)=\mathsf f^X_+(\mathsf y)$, the triangle inequality and the fact that $\mathsf f^X$ is $1$-Lipschitz we deduce that
\begin{equation} \label{eq:triang}
  \mathsf f^X(\mathsf x)-\mathsf f^X(\mathsf x')=d_{\Sp^2}(\mathsf x,\mathsf x')=d_{\Sp^2}(\mathsf x,\mathsf y)+d_{\Sp^2}(\mathsf y,\mathsf x') \,.
\end{equation}
Hence the points $(\mathsf x, \mathsf f^X(\mathsf x))$ and $(\mathsf x', \mathsf f^X(\mathsf x'))$ are joined by a lightlike segment.
If $\mathsf x,\mathsf x'$ are not antipodal points on $\partial\disk$ there, there is a unique hemispherical geodesic $\eta$ in $\overline{\disk}$ joining  $\mathsf x$ to $\mathsf x'$, which must pass through $\mathsf y$ by \eqref{eq:triang}, and which we may assume parameterized by arclength. 
In this case the geodesic segment joining $(\mathsf x, \mathsf f^X(\mathsf x))$ to $(\mathsf x', \mathsf f^X(\mathsf x'))$  takes the form $t\mapsto (\eta(t), \mathsf f^X(\mathsf x')+t)$, so it passes through $(\mathsf y,\mathsf f^X_+(\mathsf y))=(\mathsf y, \mathsf f^X_-(\mathsf y))$.

If $\mathsf x$ and $\mathsf x'$ are antipodal, then there are infinitely many geodesics joining $\mathsf x$ to $\mathsf x'$, and we can pick one going through $\mathsf y$. Then the same argument as above applies.
\end{proof}


\begin{remark}\label{rem:omegax}
Given a point $(\mathsf y,\mathsf t)$, the set of points $(\mathsf x,\mathsf s)$ satisfying $|\mathsf s-\mathsf t|<d_{S^2}(\mathsf x, \mathsf y)$ coincides with the region of $\AdSU{2,1}$ which is connected to $(\mathsf y,\mathsf t)$ by a spacelike geodesic for the Anti-de Sitter metric. It coincides also with the region of points connected to $(\mathsf y,\mathsf t)$ by a spacelike geodesic for the conformal metric $g_{\Sp^2}-dt^2$, although in general spacelike geodesics for the two metrics do not coincide.  

Now, since $\mathsf f^X_-(\mathsf y)\leq \mathsf t\leq \mathsf f^X_+(\mathsf y)$ is equivalent to the condition that
$|\mathsf s-\mathsf t|\leq d_{S^2}(\mathsf x, \mathsf y)$ for all $(\mathsf x,\mathsf s)\in X$, the region 
$$S_+(X)\cup\Omega(X)\cup S_-(X)=\{(\mathsf y,\mathsf t)\,|\, \mathsf f^X_-(\mathsf y)\leq \mathsf t\leq \mathsf f^X_+(\mathsf y)\}$$ consists of points that are connected to any point of $X$ by spacelike or lightlike geodesics. 
Moreover $\Omega(X)$ consists of points connected to any point of $X$ by a spacelike geodesic.

We remark that in general $\Omega(X)$ could be empty. For instance if $X$ is a global graph then $S_+(X)=S_-(X)=X$ and $\Omega(X)$ is empty.
\end{remark}


\begin{remark}
Since any point of $S_\pm(X)$ is connected to $X$ by a lightlike geodesic, it follows by Lemma \ref{lem:llike} that the intersection of any properly embedded achronal surface containing $X$ with $S_\pm(X)$ is a union of lightlike geodesic segments with an endpoint in $X$.
In particular any properly embedded acausal surface containing $X$ is contained in the region $\Omega(X)$.
\end{remark}

\subsection{Achronal meridians in $\partial\AdSU{2,1}$.}
We will be mainly interested in the invisible domains of achronal meridians $\Lambda$ in the boundary of $\AdSU{2,1}$, that are graphs of $1$-Lipschitz functions
$\mathsf f:\partial\disk\to\R$. Let us study more closely this case.

\begin{lemma} \label{lemma:ciaociao}
Let $\Lambda$ be an achronal meridian in $\partial\AdSU{2,1}$.
Then either $\Lambda$ is the boundary of a lightlike plane, or $S_+(\Lambda)\cap S_-(\Lambda)=\Lambda$.
In the latter case there is an achronal properly embedded surface in $\Omega(\Lambda)$ whose boundary in $\partial\AdSU{2,1}$ is $\Lambda$.
\end{lemma}
\begin{proof}
Let $\mathsf f:\partial\disk\to\R$ be the function whose graph is   $\Lambda$. Recall from Lemma \ref{lem:osc} that $\osc(\mathsf f)\leq\pi$.
If there are points $\mathsf x_0,\mathsf x_0'$ such that  $\mathsf f(\mathsf x_0')=\mathsf f(\mathsf x_0)+\pi$, then combining Lemma \ref{lem:osc} and Lemma \ref{lem:ext} we deduce that $\Lambda$ is the boundary of a lightlike plane, and this lightlike plane coincides with $S_+(\Lambda)\cap S_-(\Lambda)$.


Assume now that the maximal oscillation of $\mathsf f$ is smaller than $\pi$, and let us show that $S_+(\Lambda)\cap S_-(\Lambda)=\Lambda$. By the assumption, if a lightlike geodesic connects $(\mathsf x_0, \mathsf f(\mathsf x_0))$ to $(\mathsf x_0', \mathsf f(\mathsf x_0'))$, then $\mathsf x_0$ and $\mathsf x_0'$ are not antipodal.
But then $\mathsf x_0, \mathsf x_0'$ are connected by  a unique length-minimizing geodesic in $\overline{\disk}$ for the hemispherical metric, which lies in $\partial\disk$.  
So the lightlike line connecting $(\mathsf x_0, \mathsf f(\mathsf x_0))$ to $(\mathsf x_0', \mathsf f(\mathsf x_0'))$ is contained in
$\partial\AdSU{2,1}$. By Lemma \ref{lem:ext} we conclude that $S_-(\Lambda)$ and $S_+(\Lambda)$ do not meet in $\AdSU{2,1}$ and therefore $S_+(\Lambda)\cap S_-(\Lambda)=\Lambda$.

Finally, in this latter case the function $F=(\mathsf f^\Lambda_-+\mathsf f^\Lambda_+)/{2}$ is $1$-Lipschitz and defines an achronal properly embedded surface contained in $\Omega(\Lambda)$, whose boundary is $\Lambda$.
 \end{proof}

We remark that
in fact for any achronal meridian there is a spacelike surface whose boundary at infinity is $\Lambda$, see Remark \ref{remark:cannone} below.

Recall from Section \ref{sec:duality} that, given a point $x$ in $\AdSU{2,1}$, the Dirichlet domain of $x$ is the region $R_x$ containing $x$ and bounded by two spacelike planes ``dual'' to $x$. Namely the planes, which by a small abuse we denote by $P_x^+$ and $P_x^-$, consisting of points at timelike distance $\pi/2$ in the future (resp. past) along timelike geodesics with initial point $x$.  

\begin{prop}\label{pr:geom-inv}
Let $\Lambda$ be an achronal meridian  in $\partial\AdSU{2,1}$ different from the boundary of a lightlike plane.
Then 
\begin{enumerate}
\item A point $x\in\AdSU{2,1}$ lies in $\Omega(\Lambda)$ if and only if $\Lambda$ is contained in the interior of the Dirichlet region $R_{x}$.
\item For any $z\in\Lambda$, let $L_-(z)$ and $L_+(z)$ be the two lightlike planes such that $z$ is the past vertex of $L_+(z)$ and the future vertex of $L_-(z)$.
Then 
$$\Omega(\Lambda)=\bigcap_{z\in\Lambda}\fut(L_-(z))\cap\past(L_+(z))~.$$
\item  The length of the intersection of $\Omega(\Lambda)$ with any timelike geodesic of $\AdSU{2,1}$ is at most $\pi$.
Moreover, there exists a timelike geodesic whose intersection with $\Omega(\Lambda)$ has length $\pi$ if and only if
 $\Lambda$ is the boundary at infinity of a spacelike   plane.
\end{enumerate}
\end{prop}
\begin{proof}
By Remark \ref{rem:omegax}  a point $x$ lies in $\Omega(\Lambda)$ if and only if it is connected to any point of $\Lambda$ by a spacelike geodesic.
The region of points connected to $x$ by a spacelike geodesic has boundary the lightcone from $x$, whose intersection with $\partial\AdSU{2,1}$ coincides with $P_x^\pm\cap\partial\AdSU{2,1}$. This proves the first statement.

Similarly the region bounded by $L_+(z)$ and $L_-(z)$ contains exactly points connected to $z$ by a spacelike geodesic.
Using the characterization of $\Omega(\Lambda)$ as above, we conclude the proof of the second statement.

For the third statement, if a timelike geodesic $\gamma$ meets $\Omega(\Lambda)$  at a point $x$, then, $\Omega(\Lambda)\subset R_x$, so that the length of $\gamma\cap\Omega(\Lambda)$ is smaller than the length
of  $\gamma\cap R_x$. But the latter is $\pi$. Assume there exists a geodesic $\gamma$ such that the length of $\gamma\cap\Omega(\Lambda)$ is $\pi$. Up to applying an isometry of $\AdSU{2,1}$ we may assume that $\gamma$ is vertical in the Poincar\'e model of $\AdSU{2,1}$ and the
mid-point of $\gamma\cap\Omega(\Lambda)$ is $(0,0)$. Thus $(0,-\pi/2)$ and $(0,\pi/2)$  lie on $S_-(\Lambda)$ and $S_+(\Lambda)$ respectively.
By Remark \ref{rem:omegax} points of $\Lambda$ are connected to $(0,-\pi/2)$ by a spacelike or lightlike geodesic, hence  $\mathsf s\leq 0$ for all $(\xi,\mathsf s)\in\Lambda$.
Analogously using the point $(0,\pi/2)$ we deduce that $\mathsf s\geq 0$ for all $(\xi,\mathsf s)\in\Lambda$, so that $\Lambda=\partial\disk\times\{0\}$.
 \end{proof}

With similar arguments, we obtain that the invisible domain of an achronal meridian which is not the boundary of a lightlike plane is always contained in a Dirichlet region. 

\begin{prop}\label{rk:omega-dirichlet}
Given  an achronal meridian  $\Lambda$ in $\partial\AdSU{2,1}$ different from the boundary of a lightlike plane, the invisible domain $\Omega(\Lambda)$ is contained in a Dirichlet region. Moreover the closure of $\Omega(\Lambda)$ is contained in a Dirichlet region unless $\Lambda$ is the boundary of a spacelike plane.
\end{prop}
\begin{proof}
In fact let us set $\mathsf a_+=\sup \mathsf f_+^\Lambda$ and $\mathsf a_-=\inf \mathsf f_-^\Lambda$, and consider the planes $$Q_{\mathsf a_+}=\{(\mathsf x,\mathsf t)\,|\,\mathsf t=\mathsf a_+\}\qquad\text{and}\qquad Q_{\mathsf a_-}=\{(\mathsf x,\mathsf t)\,|\,\mathsf t=\mathsf a_-\}$$
in the Poincar\'e model.  Since clearly $\Omega(\Lambda)$ lies in the open region bounded by those planes,
it is sufficient to show that $\mathsf a_+-\mathsf a_-\leq\pi$. Assume by contradiction that $\mathsf a_+- \mathsf a_->\pi$. Notice that $P_{\mathsf a_+}$ meets $S_+(\Lambda)$ at some point $p_+=(\mathsf x_+, \mathsf a_+)$, and $P_{\mathsf a_-}$ meets $S_-(\Lambda)$ at some point $p_-=(\mathsf x_-, \mathsf a_-)$,
where $\mathsf x_+$ and $\mathsf x_-$ are points on $\overline{\disk}$. For $\epsilon=(\mathsf a_+-\mathsf a_--\pi)/2$ we can find $\mathsf x'_+$ and $\mathsf x'_-$ in $\disk$ such that $p'_+=(\mathsf x'_+, \mathsf a_+-\epsilon)$ and  $p'_-=(\mathsf x'_-, \mathsf a_-+\epsilon)$  lie in $\Omega(\Lambda)$ 
(clearly if $\mathsf x_\pm$ lies in $\disk$ we can take $\mathsf x'_\pm=\mathsf x_\pm$). As $(\mathsf a_+-\epsilon)-(\mathsf a_--\epsilon)=\pi$, the geodesic segment $\gamma$ joining $p'_+$ and $p'_-$ is timelike of length $\pi$. 
Its end-points are in $\fut(S_-(\Lambda))\cap\past(S_+(\Lambda))$, so $\gamma$ is entirely contained in $\Omega(\Lambda)$. As  end-points of $\gamma$ are contained in $\Omega(\Lambda)$, $\gamma$ can be extended within $\Omega(\Lambda)$ but this contradicts the third point of Proposition \ref{pr:geom-inv}.

The third point of Proposition \ref{pr:geom-inv} then shows that if $\mathsf a_+-\mathsf a_-=\pi$ then $\Lambda$ is  the boundary of a spacelike plane. Hence apart from this case, one has $\mathsf a_+-\mathsf a_-<\pi$, so the closure of $\Omega(\Lambda)$ is contained in a Dirichlet region. 
\end{proof}

\begin{figure}[htb]
\includegraphics[height=9cm]{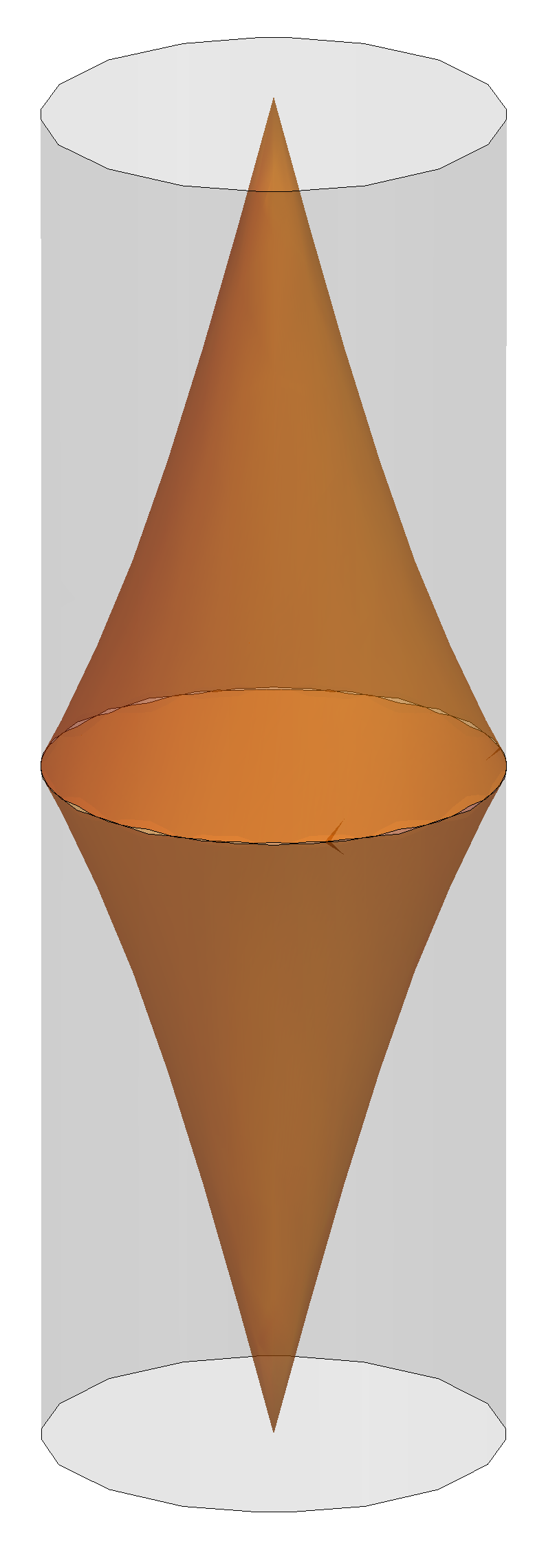}
\caption{The invisible domain of the boundary of a spacelike plane in the Poincar\'e model for $\AdSU{2,1}$.}\label{fig:invisibledomain}
\end{figure}

\begin{remark}
When $\Lambda$ is the boundary of a spacelike plane $P$, then there are two points $x_-$ and $x_+$, such that $P=P^+_{x_-}=P^{-}_{x_+}$. 
The previous arguments show that in this case
$\Omega(\Lambda)$ is the union of all timelike lines joining $x_-$ to $x_+$. In this case $S_-(\Lambda)$ is the union of future directed lightlike geodesic rays emanating from $x_-$, 
whereas $S_+(\Lambda)$ is the union of future directed lightlike geodesic rays ending at $x_+$. See Figure \ref{fig:invisibledomain}.
\end{remark}

\subsection{Domains of dependence}
We shall now introduce the notion of Cauchy surface and domains of dependence, which is general in Lorentzian geometry, and develop some properties in $\AdSU{2,1}$.
\begin{defi}
Given an achronal subset $X$ in a Lorentzian manifold $(M,g)$, the \emph{domain of dependence} of $X$ is the set 
\[
  \dom(X)=\{p\in M\,|\,\textrm{every inextensible causal curve through }p\textrm{ meets } X\}\,.
\] 
We say that $X$ is a \emph{Cauchy surface} of $M$ if $\dom(X)=M$. A spacetime $M$ is said \emph{globally hyperbolic} if it admits a Cauchy surface.
\end{defi}

Globally hyperbolic spacetimes have some strong geometric properties, which we summarize in the following theorem. We refer to \cite{beem, geroch, bernal, bernal2} for an extensive treatment.

\begin{theorem}\label{prop:GH}
Let $M$ be a globally hyperbolic spacetime. Then
\begin{enumerate}
\item Any two Cauchy surfaces in $M$ are diffeomorphic.
\item There exists a submersion $\tau:M\to\R$ whose fibers are Cauchy surfaces.
\item $M$ is diffeomorphic to $\Sigma\times\R$, where $\Sigma$ is any Cauchy surface in $M$.
\end{enumerate}
\end{theorem}

\begin{remark}
The spacetime $\AdSU{2,1}$ is not globally hyperbolic. In fact if $X$ is achronal, it is contained in the graph of a $1$-Lipschitz function $\mathsf f:({\disk}\cup\partial\D, g_{\Sp^2})\to\mathbb R$.
If $\mathsf t_0>\sup \mathsf f$ and $\xi\in\partial\disk$, then any lightlike ray with past end-point $(\xi, \mathsf t_0)$ does not intersect $X$.
\end{remark}

\begin{remark}
By the usual invariance of causality notions under conformal change of metrics, causal paths in $\AdSU{2,1}$ are the graphs of 1-Lipschitz functions from (intervals in) $\R$ to $\D$ with respect to the hemispherical metric in the image. Hence an inextensible causal curve in $\AdSU{2,1}$ is either the graph of a global 1-Lipschitz function from $\R$, or it is defined on a proper interval and has endpoint(s) in $\partial\AdSU{2,1}$.
\end{remark}

\begin{lemma}\label{prop:cauchy in invisible}
Given an achronal meridian $\Lambda$ in $\partial\AdSU{2,1}$, any Cauchy surface in $\Omega(\Lambda)$ is properly embedded with boundary at infinity $\Lambda$. 
\end{lemma}
\begin{proof}
Let $S$ be a Cauchy surface in $\Omega(\Lambda)$. For every $x\in\disk$, the vertical line through $x$ in the Poincar\'e model meets $\Omega(\Lambda)$, and its intersection with $\Omega(\Lambda)$ must meet $S$ by definition of Cauchy surface.
This shows that $S$ is a graph over $\disk$, proving that $S$ is properly embedded, and clearly $\partial S=\Lambda$. 
\end{proof}

\begin{prop}\label{prop:doms}
Let $\Lambda$ be an achronal meridian in $\partial\AdSU{2,1}$ different from the boundary of a lightlike plane.
Let $S$ be a properly embedded {achronal} surface in $\Omega(\Lambda)$.
Then $\dom(S)=\Omega(\Lambda)$. In particular $\Omega(\Lambda)$ is a globally hyperbolic spacetime.
\end{prop}
\begin{proof}
Let $x$ be  any point in $\Omega(\Lambda)$ and take any inextensible causal path through $x$. A priori its future endpoint might be either in $S_+(\Lambda)$ or in $\Lambda$, but by definition of  $\Omega(\Lambda)$, $x$ cannot be connected by a causal path to $\Lambda$, hence the latter case is excluded. The same argument applies to show that the past endpoint is in $S_-(\Lambda)$. Since the inextendible causal path meets both $S_+(\Lambda)$ and $S_-(\Lambda)$, it must meet $S$ by Lemma \ref{prop:cauchy in invisible}, hence $x\in\dom(S)$.

Conversely, if $x$ is not in $\Omega(\Lambda)$, then one can find a causal path joining $x$ to $\Lambda$, which is necessarily inextensible. Hence $x$ is not in $\dom(S)$. This concludes the proof. 
\end{proof}

\begin{remark} \label{remark:cannone}
It follows from Theorem \ref{prop:GH} and Proposition \ref{prop:doms} that $\Lambda$ is the boundary of a spacelike surface in $\Omega(\Lambda)$, namely a Cauchy surface for  $\Omega(\Lambda)$. By lemma \ref{prop:cauchy in invisible}, this surface is properly embedded, hence the graph of a global 1-Lipschitz function. This shows that any proper achronal meridian $\Lambda$ is the boundary at infinity of a properly embedded spacelike surface, which improves the statement of Lemma \ref{lemma:ciaociao}.
\end{remark}

The most remarkable consequence of Proposition \ref{prop:doms} is that the domain of dependence of a properly embedded surface in $\AdSU{2,1}$ only depends on the boundary at infinity.
More precisely we have:
\begin{cor}\label{cor boundary at infinity cauchy}
If $S$ and $S'$ are properly embedded spacelike surfaces in $\AdSU{2,1}$, then $\dom(S)=\dom(S')$ if and only if $\partial S=\partial S'$.
\end{cor}

\subsection{Properly achronal sets in $\AdSP{2,1}$}

It will be important for the applications of this theory  to consider the model $\AdSP{2,1}$. As $\AdSP{2,1}$ contains closed timelike lines, it does not contain any achronal subset.
However if $P$ is a spacelike plane in $\AdSP{2,1}$, then $\AdSP{2,1}\setminus P$ does not contain closed causal curves. 
Indeed it is simply connected, so it admits an isometric embedding into $\AdSU{2,1}$, given by a section
of the covering map $\AdSU{2,1}\to\AdSP{2,1}$, and whose image is a Dirichlet region.

\begin{defi}
A subset $X$ of $\AdSP{2,1}\cup\partial\AdSP{2,1}$ is a \emph{proper achronal subset} if there exists a spacelike plane $P$ such that $X$ is contained in $\AdSP{2,1}\cup\partial\AdSP{2,1}\setminus \overline{P}$ and is achronal as a subset of $\AdSP{2,1}\cup\partial\AdSP{2,1}\setminus \overline{P}$.
\end{defi}

Notice that if $X$ is a proper achronal  subset of $\AdSP{2,1}\cup\partial\AdSP{2,1}$, then it admits a section to $\AdSU{2,1}\cup\partial\AdSU{2,1}$ and the image is achronal in $\AdSU{2,1}\cup\partial\AdSU{2,1}$. 
Conversely if $\widetilde X$ is an achronal subset of $\AdSU{2,1}$ different from a lightlike plane, then it is contained in a Dirichlet region, as a consequence of Lemma \ref{lem:osc} and the fact that any achronal subset of $\AdSU{2,1}$ is contained in a properly embedded one.
As Dirichlet regions are projected in $\AdSP{2,1}$ to the complement of a spacelike plane, the image of $\widetilde X$ to $\AdSP{2,1}$ is a proper achronal subset.

Let us provide an important  example which will be extensively used later.

\begin{lemma}
Let $\varphi:\RP^1\to\RP^1$ be an orientation preserving homeomorphism.
Then the graph of $\varphi$, say $\Lambda_{\varphi}\subset\RP^1\times\RP^1\cong\partial\AdSP{2,1}$ is a proper achronal subset and any lift $\widetilde\Lambda_\varphi$ is an achronal meridian in $\partial\AdSU{2,1}$.
\end{lemma}
\begin{proof}
First let us prove that $\Lambda_{\varphi}$ is locally achronal. In fact if $U$ and $V$ are intervals around $x$ and $\varphi(x)$ and $\theta_1$ and $\theta_2$ are positive coordinates on $U$ and $V$ respectively, then
timelike curves  $\gamma(t)=(\gamma_1(t), \gamma_2(t))$ in $U\times V$  are characterized by the property that $\theta_1'(t)\theta_2'(t)<0$, where we have put $\theta_i(t):=\theta_i(\gamma_i(t))$. (See Proposition \ref{prop conformal class ein} and the following paragraph.)
In particular points on $\Lambda_{\varphi}\cap U\times V$ are not related by a timelike curve contained in $U\times V$, by the assumption that $\varphi$ is orientation-preserving.

Let us prove that  there exists a spacelike plane $P$ such that $\overline{P}\cap\Lambda_{\varphi}=\emptyset$. 
Let us consider the identification $\RP^1 =\mathbb R \cup \{\infty\}$, and take $\varphi_0\in\PSL(2,\R)$ so that $\varphi_0^{-1}\varphi(0)=1$, $\varphi_0^{-1}\varphi(1)=\infty$, and $\varphi_0^{-1}\varphi(\infty)=0$.
Then notice that  $\varphi_0^{-1}\varphi$ sends  the intervals $(\infty,0)$, $(0,1)$ and $(1,\infty)$ respectively to $(0,1)$, $(1,\infty)$, $(\infty,0)$. Thus $\varphi_0^{-1}\varphi$ has no fixed points, that is, the graph of $\varphi$ does not meet the graph 
of $\varphi_0$, which is the asymptotic boundary of a spacelike plane $P_{\varphi_0}$.

Let us consider now a lift of $\Lambda_{\varphi}$ to the boundary of $\AdSU{2,1}$, say $\widetilde\Lambda_{\varphi}$. As $\Lambda_{\varphi}$ is contained in a simply connected region of $\AdSP{2,1}\cup\partial\AdSP{2,1}$, $\widetilde\Lambda_{\varphi}$
 is a closed locally achronal curve contained in $\partial\AdSU{2,1}$. In particular the projection $\widetilde\Lambda_{\varphi}\to\partial\disk$ is locally injective. As $\widetilde\Lambda_{\varphi}$ is compact, the map is a covering.
On the other hand, since $\Lambda_{\varphi}$ is homotopic to the boundary of a plane in $\partial\AdSP{2,1}$, it turns out that $\widetilde\Lambda_{\varphi}$ is homotopic to $\partial\disk$ in $ \partial\AdSU{2,1}$ so that the projection 
$\widetilde\Lambda_{\varphi}\to\partial\disk$ is bijective.
It follows that $\widetilde\Lambda_{\varphi}$ is achronal, and the conclusion follows. 
\end{proof}


All the results we have proven for achronal sets in $\AdSU{2,1}$ can be rephrased for proper achronal sets of $\AdSP{2,1}$. For instance any proper achronal set $X$ 
 can be extended to a properly embedded proper achronal surface and there are two extremal extensions, as in Lemma \ref{lem:ext}.

We will now focus on proper achronal meridians of $\partial\AdSP{2,1}$, which are proper achronal  embedded circles of the boundary of $\AdSP{2,1}$. They lift  to achronal meridians of $\partial\AdSU{2,1}$
different from the boundary of a lightlike plane. Indeed the boundary of a lightlike plane is not contained in a Dirichlet region. Conversely any achronal meridian of $\partial\AdSU{2,1}$ different from the boundary of a lightlike plane projects
to an achronal meridian of $\AdSP{2,1}$.

\begin{prop}\label{lem:omegads}
Let $\Lambda$ be a proper achronal meridian in  $\partial\AdSP{2,1}$ and denote by $\widetilde \Lambda$ any lift to the universal covering. 
Then the universal covering map of $\AdSP{2,1}$ maps $\Omega(\widetilde \Lambda)$ injectively to the domain
$$\Omega(\Lambda):= \{x\in\AdSP{2,1}\,|\, P_x\cap \Lambda=\emptyset\}~.$$
\end{prop}
\begin{proof}
If $p:\AdSU{2,1}\to\AdSP{2,1}$ denotes the covering map, by  Proposition  \ref{rk:omega-dirichlet} the invisible domain $\Omega(\widetilde\Lambda)$ is contained in a Dirichlet region $R_{\tilde x}$, hence the restriction of $p$ to $\Omega(\widetilde\Lambda)$ is injective and its image is contained in $p(R_{\tilde x})$, namely the complement  in $\AdSP{2,1}\cup\partial\AdSP{2,1}$ of the spacelike plane $P_x$ dual to $x=p({\tilde x})$.
Moreover by the first point of Proposition \ref{pr:geom-inv}, one can actually pick for $\tilde x$ any point in $\Omega(\widetilde\Lambda)$, which shows that the image $p(\Omega(\widetilde\Lambda))$ is contained in $\Omega(\Lambda):=\{x\in\AdSP{2,1}\,|\, P_x\cap \Lambda=\emptyset\}$.

For the converse inclusion, let $x\in\AdSP{2,1}$ be a point whose dual plane $P_x$ does not meet $\Lambda$.
The preimage $p^{-1}(P_x)$ is a countable disjoint union of planes which disconnect $\AdSU{2,1}\cup \partial\AdSU{2,1}$ in a disjoint union of  Dirichlet regions centered at  preimages of $x$. The lift $\widetilde\Lambda$ is contained
in exactly one such region, say $R_{\tilde x}$. By the first point of Proposition \ref{pr:geom-inv} $\tilde x\in\Omega(\widetilde\Lambda)$ which implies that $x=p(\tilde x)$ lies in $p(\Omega(\widetilde\Lambda))$.
\end{proof}


When $\Lambda$ is the graph of an orientation-preserving homeomorphism $\varphi:\mathbb RP^1\to\mathbb RP^1$, there is a fairly simple characterization of $\Omega(\Lambda)$ using the identification $\AdSP{2,1}=\PSL(2,\mathbb R)$.

\begin{cor}\label{cor:inv-homeo}
Let $\varphi$ be an orientation-preserving homeomorphism. Then $x\in\AdSP{2,1}$ lies in $\Omega(\Lambda_{\varphi})$ if and only if $x\circ\varphi$ has no fixed point as a homeomorphism of  $\mathbb RP^1$.
\end{cor}
\begin{proof}
It is easy to check that the dual plane of $x$, as an element of $\PSL(2,\R)$, meets $\partial\AdSP{2,1}$ along the graph of $x^{-1}$, say $\Lambda_{x^{-1}}$. Indeed this is easily checked if $x=\id$ is the identity by the description of $P_\En$ we gave in Section \ref{sec:geodesics PSL2R} together with Lemma \ref{lemma convergence at infinity}. The general case then follows by applying left multiplication by $x$ itself, which maps the graph of the identity to the graph of $x^{-1}$.

With this remark in hand, we have that $x\in\Omega(\Lambda_{\varphi})$
 if and only if $\Lambda_{x^{-1}}\cap\Lambda_{\varphi}=\emptyset$. This condition is equivalent to requiring that $x\circ\varphi$ has no fixed point on $\mathbb RP^1$.
\end{proof}

\begin{prop}\label{pr:spacelikeimm}
Let $\sigma: S\to\AdSP{2,1}$ be a proper spacelike immersion. Then
\begin{itemize}
\item $\sigma$ is a proper embedding.
\item $\sigma$ lifts to a proper embedding $\widetilde\sigma:S\to\AdSU{2,1}$.
\item The boundary at infinity of $\sigma(S)$ is a proper achronal meridian $\Lambda$ in $\partial\AdSP{2,1}$.
\item $\dom(\sigma(S))=\Omega(\Lambda)$.
\end{itemize} 
\end{prop}
\begin{proof}
 Denote by $\widehat S$ the covering of $S$ admitting a lift $\widehat\sigma:\hat S\to\AdS{2,1}$.
 In general either $\widehat S=S$ or it is a $2:1$ covering.
Since the covering is finite, $\widehat \sigma$ is a proper immersion.

Let us consider the identification $\pi:\Hyp^2\times \mathbb S^1\to\AdS{2,1}$ defined in \eqref{eq:covering map}.
The induced projection $\mathrm{pr}:\AdS{2,1}\to\Hyp^{2}$ is a proper fibration with timelike fibers. 
In particular  $\widehat \sigma$ is trasverse to the fibers of $\mathrm{pr}$.
It follows that $\mathrm{pr}\circ\sigma:\widehat S\to\Hyp^2$ is a proper local diffeomorphism, hence a covering map. 
Since $\Hyp^2$ is simply connected, we deduce that the projection $\mathrm{pr}\circ\sigma:\widehat S\to\Hyp^2$ is a homeomorphism, $\widehat\sigma$ is an embedding, and $\widehat S$ is homeomorphic to the plane.

In particular we can lift $\widehat \sigma$ to the universal covering, say $\widetilde\sigma:\widehat S\to\AdSU{2,1}$, which is still a proper spacelike embedding $\widehat S\to\AdSU{2,1}$.
By Lemma \ref{lemma global graph} and Lemma \ref{lem:proper embedding} we know that the image is an achronal surface whose boundary is an achronal meridian, and is contained in a Dirichlet domain by Lemma \ref{lem:osc}.
It follows that $\widetilde\sigma(\widehat S)$ is contained in a Dirichlet domain of the covering map $\AdS{2,1}\to\AdSP{2,1}$, on which we know that the covering map is injective. In particular $\sigma$ is also injective, hence $\widehat S=S$ and this concludes the proof.
\end{proof}

\begin{remark}
In the proof of Proposition \ref{pr:spacelikeimm}, once we proved that $\widehat S$ is homeomorphic to $\R^2$, then we could have inferred immediately that $\widehat S=S$ since it is known, although non-trivial, that $\Z/2\Z$ cannot act freely on $\R^2$ by diffeomorphisms.
\end{remark}

We therefore have the following analogue version of Corollary \ref{cor boundary at infinity cauchy} in $\AdSP{2,1}$.

\begin{cor}
If $S$ and $S'$ are properly embedded spacelike surfaces in $\AdSP{2,1}$, then $\dom(S)=\dom(S')$ if and only if $\partial S=\partial S'$.
\end{cor}

\subsection{Convexity notions}
Let $\Lambda$ be a proper achronal meridian in $\partial\AdSP{2,1}$.
 In this section we will investigate the convexity properties of $\Omega(\Lambda)$.

Let us recall that $X\subset \RP^3$ is  convex if it is contained in an affine chart and it is convex in the affine chart. This notion does not depend on the affine chart containing $X$.
It is a proper convex set if it is moreover compactly contained in an affine chart.


\begin{prop}\label{prop invisible convex}
Given a proper achronal meridian $\Lambda$ in $\partial\AdSP{2,1}$, $\Omega(\Lambda)$ is convex. 
If $\Lambda$ is different from the boundary of a spacelike plane then $\Omega(\Lambda) $ is a proper convex set.
\end{prop}
\begin{proof}
By Proposition \ref{rk:omega-dirichlet} there exists a spacelike plane $P$ such that $\Omega(\Lambda)$ is contained in the affine chart $V$ of $\RP^3$ obtained by removing  the projective plane containing $P$.
The domain $\AdSP{2,1}\cap V=\AdSP{2,1}\setminus P$ is isometric to a Dirichlet region $R$ of $\AdSU{2,1}$, by an isometry sending $\Lambda$ to a lifting $\widetilde\Lambda$ and $\Omega(\Lambda)$ to $\Omega(\widetilde\Lambda)$.
By the second point of Proposition \ref{pr:geom-inv} we have 
$$\Omega(\widetilde\Lambda)=\bigcap_{\tilde z\in\widetilde \Lambda}\fut(L_-(\tilde z))\cap\past(L_{+}(\tilde z))~.$$
Now if $\tilde z$ projects to $z$, then the images of $L_-(\tilde z)$ and $L_+(\tilde z)$ in $V$ are the two components of
$L(z)\cap\AdSP{2,1}$, where $L(z)$ is the affine tangent plane of $\partial\AdSP{2,1}\cap V$ at $z$. It turns out that the image of the region $\fut(L_-(\tilde z))\cap\past(L_{+}(\tilde z))$ is  the intersection of $\AdSP{2,1}$ with the open half-space $U(z)$ bounded by
$L(z)$ and whose closure contains $\Lambda$.
This shows:
\[
   \Omega(\Lambda)=\AdSP{2,1}\cap\bigcap_{z\in\Lambda} U(z)~.
\]
We now claim that actually 
$$\Omega(\Lambda)=\bigcap_{z\in\Lambda} U(z)\subset \AdSP{2,1}~,$$ which will conclude the proof. As $\bigcap_{z\in\Lambda} U(z)$ is connected and meets $\AdSP{2,1}$, to show that it is contained in $\AdSP{2,1}$ it is sufficient to show that
it does not meet the boundary of $\AdSP{2,1}$.  For any  $w\in\partial\AdSP{2,1}$ let us consider the leaf of the left ruling through $w$, which intersects $\Lambda$ at a point $z$. It turns out that $L(z)$ contains the leaf of the left ruling through $z$, hence $w\notin U(z)$.

Now, assume that $\Lambda$ is not the boundary of a spacelike plane. Then by Proposition \ref{rk:omega-dirichlet} on the universal covering the compact set $\Omega(\widetilde\Lambda)\cup S_+(\widetilde\Lambda)\cup S_-(\widetilde\Lambda)$ is contained in a Dirichlet domain, so its image is a compact set 
contained in an affine chart. 
\end{proof}

A a consequence of the previous argument is that $\Lambda$ is contained in an affine chart whose complement in $\RP^3$ is a projective plane containing a spacelike plane of $\AdSP{2,1}$. (Indeed $\Lambda$ is contained in the closure of $\Omega(\Lambda)$, which is contained in an affine chart, unless $\Lambda$ is the boundary of a spacelike plane, in which case the statement is trivial.)
Hence it makes sense to give the following definition:
\begin{defi}
Given a proper achronal meridian $\Lambda$ in $\partial\AdSP{2,1}$, we define  $C(\Lambda)$ to be the convex hull of  $\Lambda$, which can be taken in an affine chart containing $\Lambda$. 
\end{defi}

Observe that we have proved implicitly that if $\Lambda$ is an achronal meridian in $\partial\AdSP{2,1}$, then $C(\Lambda)$ is contained in $\AdSP{2,1}$, which is not obvious as 
$\AdSP{2,1}$ is not convex in $\RP^3$.

\begin{remark}
Since $\Omega(\Lambda)$ is convex, $C(\Lambda)$ is contained in $\Omega(\Lambda)$. Moreover if $K$ is any convex set contained in $\AdSP{2,1}\cup\partial\AdSP{2,1}$ and containing $\Lambda$, then $C(\Lambda)\subset K\subset\overline{\Omega(\Lambda)}$.

To see this, let $V$ be an affine chart such that $\Lambda\subset V$ is obtained by removing a spacelike projective plane.
Now, if $z\in\Lambda$ then for any $x\in\AdSP{2,1}\cap V$ the segment connecting $z$ and $x$ in $V$ is contained in $\AdSP{2,1}$ if and only if $x\in U(z)$, the half-space containing $\Lambda$ and bounded by the tangent space of $\Lambda$ at $z$, as defined in the proof of Proposition \ref{prop invisible convex}.  

Hence by the characterization of $\Omega(\Lambda)$ as the intersection of the $U(z)$ given in Proposition \ref{prop invisible convex}, if $x$ is not in $\overline{\Omega(\Lambda)}$, then it cannot be in $K$. This shows that $\overline{\Omega(\Lambda)}$ is the biggest convex subset of $\AdSP{2,1}$ containing $\Lambda$.
 \end{remark}

Assume now that $\Lambda$ is not the boundary of a spacelike plane.
Then the topological frontiers  in $\RP^3$ of $\Omega(\Lambda)$ and of $C(\Lambda)$ are Lipschitz surfaces homeomorphic to a sphere.
This sphere is disconnected by $\Lambda$ into two regions, homeomorphic to disks, which form the boundary of $\Omega(\Lambda)$ and of $C(\Lambda)$ in $\AdSP{2,1}$. 
For $\Omega(\Lambda)$ those components are the image of $S_\pm(\widetilde\Lambda)$ and will be denoted by $S_\pm(\Lambda)$.

Let us now focus on $C(\Lambda)$. 
Let $C(\widetilde\Lambda)$ be a lifting of $C(\Lambda)$, which is contained in a Dirichlet region, say $R$.
Let $P$ be a support plane for $C(\Lambda)$, which is necessarily either spacelike or lightlike, and let  $\widetilde P$ be its lift which touches $C(\widetilde\Lambda)$.
 Hence either $\widetilde\Lambda$ is in  $\fut(\widetilde P)\cup \widetilde P$   or in $\past(\widetilde P)\cup \widetilde P$. 
This permits to distinguish the components of $\partial C(\Lambda)\setminus\Lambda$: the \emph{past boundary component} $\partial_-C(\Lambda)$ has the property that $\widetilde \Lambda$ is contained in $\fut(\widetilde P)\cup\widetilde P$ for all support planes which touch $\partial_-C(\Lambda)$, and analogously we define the \emph{future boundary component} $\partial_+C(\Lambda)$ by replacing $\fut$ with $\past$.
The following proposition explains that the boundary components $\partial_\pm C(\Lambda)$ and $S_\pm(\Lambda)$ have a kind of duality.

\begin{prop}\label{pr:supp}
Let  $\Lambda$ be a proper achronal meridian in $\AdSP{2,1}$,  $x\in\AdSP{2,1}$ and $P_x$ the dual plane. Then
\begin{itemize}
\item  $x\in\Omega(\Lambda)$ if and only if $P_x\cap C(\Lambda)=\emptyset$.
\item $x\in C(\Lambda)$ if and only if $P_x\cap \Omega(\Lambda)=\emptyset$.
\end{itemize}

In particular if $\Lambda$ is not the boundary of a spacelike plane, then 
\begin{itemize}
\item $x\in \partial_\pm\Omega(\Lambda)$ if and only if
$P_x$ is a support plane for $\partial_{\mp} C(\Lambda)$.
\item $x\in  \partial_\pm C(\Lambda)$ if and only if 
$P_x$ is a support plane for  $S_\mp(\Lambda)$.
\end{itemize}
\end{prop}
\begin{proof}
From Proposition \ref{lem:omegads}, points in $\Omega(\Lambda)$ are dual to planes disjoint from $\Lambda$, which are precisely those which do not intersect $C(\Lambda)$, by the definition of convex hull.
For the second statement,  fix $x$ and observe that $z\in P_x$ if and only if $x\in P_z$. Hence there exists a point $z$ in the intersection $P_x\cap \Omega(\Lambda)$ if and only if $x$ is in a plane $P_z$ which is disjoint from $\Lambda$, namely when $x$ is not in $C(\Lambda)$.

As a consequence $\partial C(\Lambda)$ consists of points dual to support planes of $\Omega(\Lambda)$.
Take a support plane $P_x$ of $S_+(\Lambda)$ (hence dual to a point $x$) which meets $S_+(\Lambda)$ at $z$. If $\tilde z$ denotes the corresponding point on $S_+(\widetilde\Lambda)$, then
 $\widetilde \Lambda\subset \fut(P^{-}_{\tilde z})$, and $P^{-}_{\tilde z} \cap\tilde\Lambda\neq\emptyset$.
Thus $P_z$, which is the projection of $P^{-}_{\tilde z}$,  is a support plane of $C(\Lambda)$ touching the past boundary. As $x\in P_z$, we conclude that $x$ lies in the past boundary.
Similarly points of the future boundary of $C(\Lambda)$ correspond to support planes for $S_-(\Lambda)$.  
\end{proof}

\begin{remark} \label{rmk:light triangles convex invisible}
It may happen that a boundary component of $C(\Lambda)$ meets the boundary of $\Omega(\Lambda)$. This exactly happens when the curve $\Lambda$ contains a \emph{sawtooth}, namely two consecutive lightlike segments in $\partial\AdSP{2,1}$ one past directed and the other future-directed.  
In this case the lightlike plane $L(z)$ tangent to $\partial\AdSP{2,1}$ at the vertex $z$ of the sawtooth contains the two consecutive lightlike segments of $\Lambda$, while the convex hull of $\Lambda$ contains a lightlike triangle contained in $L(z)$. This is however 
 not contained in $\Omega(\Lambda)$. If the curve $\Lambda$ does not contain any sawtooth, then $C(\Lambda)\setminus\Lambda$ is entirely contained in $\Omega(\Lambda)$.
 
 The fundamental example is given in Figure \ref{fig:duallines}, where the yellow region represents at the same time the convex hull of the proper achronal meridian $\Lambda$ in $\partial\AdSP{2,1}$ composed of four lightlike segments, two past-directed and two future-directed, and the closure of $\Omega(\Lambda)$. See also Remark \ref{rmk:Ksurfaces sawtooth} and Figure   \ref{fig:sawtooth} below. 
\end{remark}

\begin{prop} \label{prop bdy convex hull achronal}
The past and future boundary components of $C(\Lambda)$ are achronal surfaces.
\end{prop}
\begin{proof}
Let us give the proof for $\partial_+C(\Lambda)$. Take $x,y\in\partial_+C(\Lambda)$ and consider the segment joining $x$ to $y$ in an affine chart containing $\Lambda$.
If this segment was timelike then the dual planes $P_x$ and $P_y$ would be disjoint.
Then up to switching $x$ to $y$ we may assume that, in the universal cover, $P^{1}_{\tilde x}\subset\fut (P^{1}_{\tilde y})$, where $\tilde x$ and $\tilde y$ are the lifting of $x$ and $y$ in the same Dirichlet region mapping to the fixed affine chart.
But then $S_+(\widetilde\Lambda)$ would be contained in $\past(P^{1}_{\tilde y})$ and could not meet $P^{1}_{\tilde x}$, thus contradicting Proposition \ref{pr:supp}.
\end{proof}

\begin{remark} \label{rmk:lightlike triangles in boundary components}
The past and future boundary components of $C(\Lambda)$ are not smooth, but only Lipschitz surfaces. Indeed the complement of $\Lambda$ and of the lightlike triangles (as  described in Remark \ref{rmk:light triangles convex invisible})  is locally connected by acausal Lipschitz arcs,
and one can define a pseudo-distance, that in fact turns out to be a distance and makes $C(\Lambda)$ locally isometric to the hyperbolic plane.

The situation is very similar to the counterpart in hyperbolic three-space. The intersection of a spacelike support plane with $C(\Lambda)$ is either a geodesic or a straight convex subset of $\Hyp^2$, i.e. a subset bounded by geodesics. Thus $\partial C(\Lambda)\setminus\Lambda$ is intrinsically a hyperbolic
surface pleated along a measured geodesic lamination. A remarkable difference with respect to the hyperbolic case is that in general  
those surfaces may be non complete, but they are always isometric to straight convex subsets of $\Hyp^2$. See \cite{MR2764871} for more details.
\end{remark}

\section{Globally hyperbolic three-manifolds} \label{sec:GH AdS mfds}

The aim of this section is to study maximal globally hyperbolic (MGH) Anti-de Sitter spacetimes containing a compact Cauchy surface of genus $r$ (we briefly say that the globally hyperbolic spacetimes have genus $r$).
We first prove that there are no examples for $r=0$. We will then briefly consider the torus case, and finally we will deepen the study for $r\geq 2$, first by introducing examples, and then by giving a complete classification.

\subsection{General facts}

We begin by some general results which will be used both in the genus one and in the higher genus case. Recall that an immersion $\sigma: S\to\AdSP{2,1}$ is spacelike if the pull-back metric (also called first fundamental form) is a Riemannian metric. We will provide more details on the theory of spacelike immersions in Section \ref{sec:spacelike immersions} below, which can be read independently.

\begin{lemma} \label{lemma complete implies proper}
Let $\sigma: S\to\AdSP{2,1}$ be a spacelike immersion. If $\sigma^*(g_{\AdSP{2,1}})$ is a complete Riemannian metric, then $\sigma$ is a proper embedding and $S$ is diffeomorphic to $\mathbb R^2$.
\end{lemma}
\begin{proof}
By Proposition  \ref{pr:spacelikeimm} it is sufficient to prove that $\sigma$ is a proper immersion.
In the notation of Proposition  \ref{pr:spacelikeimm}, consider a lift $\widehat\sigma:\widehat S\to\AdS{2,1}$. It is clearly sufficient to prove that $\widehat\sigma$ is proper.
We will prove that if $\gamma:[0,1)\to\widehat S$ is a path such that the limit $\lim_{t\to 1}\widehat\sigma(\gamma(t))$ exists, then also $\lim_{t\to 1}\gamma(t)$ exists.

Using the expression \eqref{poinc:eq} for the metric on $\AdS{2,1}$ under the identification with $\Hyp^2\times\Sp^1$ given by \eqref{eq:covering map}, we see that the length of $\gamma$ for the pull-back metric is smaller that the length of the projection of $\gamma$ to the $\Hyp^2$ factor, with respect to the hyperbolic metric on $\Hyp^2$. The latter hyperbolic length is finite by the assumption, hence 
$\gamma$ has finite length
for the pull-back metric. 
The assumption on the completeness of the pull-back metric implies the existence of the limit point for $\gamma(t)$.
\end{proof}

As an immediate consequence, there can be no globally hyperbolic AdS spacetime whose Cauchy surfaces are diffeomorphic to the sphere. In fact, supposing such a spacetime exists and $\Sigma$ is a Cauchy surface, the developing map restricted to $\Sigma$ would be a spacelike immersion, and the pull-back metric would be complete by compactness. But this contradicts Lemma \ref{lemma complete implies proper}. Hence we proved:

\begin{cor}
There exists no globally hyperbolic Anti-de Sitter  spacetime of genus zero.
\end{cor}

The following is a fundamental result on the structure of globally hyperbolic AdS spacetimes.

\begin{prop}\label{pr:MGHADS}
Let $M$ be a globally hyperbolic Anti-de Sitter spacetime of genus $r\geq 1$.
Then
\begin{enumerate}
\item The developing map $\dev:\widetilde M\to\AdSP{2,1}$ is injective.
\item  If $\Sigma$ is a Cauchy surface of $M$, then the image of $\dev$ is contained in $\Omega(\Lambda)$, where $\Lambda$ is the boundary at infinity of $\dev(\widetilde \Sigma)$. 
\item If $\rho:\pi_1(M)\to\isom(\AdSP{2,1})$ is the holonomy representation, $\rho(\pi_1(M))$ acts freely and properly discontinuously on $\Omega(\Lambda)$, and $\Omega(\Lambda)/\rho(\pi_1(M))$ is a globally hyperbolic spacetime containing $M$.
\end{enumerate}
\end{prop}
\begin{proof}
Let $\widetilde\dev:\widetilde M\to\AdSU{2,1}$ be a lift of $\dev$ to the universal cover. By Theorem \ref{prop:GH}, the spacetime $M$ admits a foliation by smooth spacelike surfaces $(\Sigma_t)_{t\in\R}$  of genus $r\geq 1$, such that $\Sigma_t\subset\fut(\Sigma_{t'})$ for $t>t'$.
Let $\widetilde \Sigma_t$ the lift of the foliation on $\widetilde M$. Since $\Sigma_t$ is closed,  the induced metric on $\Sigma_t$ is complete, and  so is the induced metric on $\widetilde \Sigma_t$.
As $\widetilde\dev$ is a local isometry, we deduce by Lemma \ref{lemma complete implies proper} that the restriction of $\widetilde\dev$ to $\widetilde \Sigma_t$ is a proper embedding.

Assume now by contradiction that $\widetilde\dev(\widetilde \Sigma_t)\cap\widetilde\dev(\widetilde \Sigma_{t'})\neq\emptyset$ for some $t\geq t'$.  Then there is a point $x\in \widetilde \Sigma_t$ such that $\widetilde\dev(x)\in \dev(\widetilde \Sigma_{t'})$. By the assumption $x$ is connected to
$\widetilde \Sigma_{t'}$ by a timelike arc $\eta$ in $\widetilde M$. 
Then $\widetilde\dev(\eta)$ is a timelike arc in $\AdSU{2,1}$ with end-points in $\widetilde\dev(\widetilde \Sigma_{t'})$ and this contradicts the achronality of  $\widetilde\dev(\widetilde \Sigma_{t'})$.
This shows that $\widetilde\dev$ is injective, and moreover we conclude that $\widetilde\dev(\widetilde \Sigma_t)$ is a Cauchy surface of $\widetilde\dev(\widetilde M)$.
It follows using Proposition \ref{pr:spacelikeimm} that $\widetilde\dev(\widetilde M)\subset\dom(\widetilde\dev(\widetilde \Sigma_t))=\Omega(\widetilde\Lambda)$, where $\widetilde\Lambda$ is the boundary at infinity of $\widetilde\dev(\widetilde \Sigma_t)$, which proves  the second point.

Now, the map $\widetilde\dev$ is $\widetilde\rho$-equivariant, for a representation $\widetilde\rho:\pi_1(M)\to\Isom(\AdSU{2,1})$ which is a lift of the holonomy of $M$.
As $\widetilde\dev(\widetilde \Sigma_t)$ is $\widetilde\rho$-invariant, then so are $\widetilde\Lambda$ and $\Omega(\widetilde\Lambda)$. We shall prove that the action of $\pi_1(M)$ on $\Omega(\widetilde\Lambda)$ given by $\widetilde\rho$
is proper. This will also show that the action is free, since $\pi_1(M)$ is isomorphic to $\pi_1(\Sigma_r)$ and therefore has no torsion.

For this purpose, let us notice that if $K$ is relatively compact in $\Omega(\widetilde\Lambda)$
then 
$$X_K:=(\fut(K)\cup\past(K))\cap\widetilde\dev(\widetilde\Sigma_t)$$ 
is relatively compact as well.
 As the action of $\pi_1(M)$ on $\widetilde \Sigma_t$, and thus on $\widetilde\dev(\widetilde \Sigma_t)$, is proper and $X_{\gamma K}=\gamma(X_K)$, we deduce that the set of $\gamma$ such that
$X_{\gamma K}\cap X_K\neq\varnothing$ is finite. On the other hand if $K\cap\gamma K\neq\emptyset$ then $X_K\cap X_{\gamma K}\neq\emptyset$. We thus conclude that the action is proper.
By applying the path lifting property, one sees that the quotient $\widetilde\dev(\widetilde \Sigma_t)/\pi_1(M)$  is a Cauchy surface of $\Omega(\widetilde\Lambda)/\pi_1(M)$, which is therefore globally hyperbolic.

The proof of the statement is then accomplished since by Proposition \ref{lem:omegads} the restriction of the covering map $\AdSU{2,1}\to\AdSP{2,1}$ to $\Omega(\widetilde\Lambda)\cup\Lambda$ is injective.
\end{proof}

A remarkable difference between Lorentzian and Riemannian geometry is that in Lorentzian geometry geodesic completeness is a very strong assumption, and in fact interesting classification results are obtained without such an assumption. However, it is necessary to impose  
some maximality condition to compensate for non-completeness.
Among several approaches, one of the most common is the classification of a maximal globally hyperbolic spacetimes. We give a definition here in our special setting, although one can give more general definitions in the larger class of Einstein spacetimes.

\begin{defi}\label{defi:GHAdS}
A globally hyperbolic Anti-de Sitter manifold $(M,g)$ is \emph{maximal} if any isometric embedding of $(M,g)$ into a globally hyperbolic Anti-de Sitter manifold $(M',g')$, which sends a Cauchy surface of $(M,g)$ to a Cauchy surface of $(M',g')$, is surjective.
\end{defi}

The following corollary is a direct consequence of Proposition \ref{pr:MGHADS} and Definition \ref{defi:GHAdS}.

\begin{cor}\label{cor:GHAdS}
An Anti de-Sitter globally hyperbolic spacetime $M$ is maximal if and only if $\widetilde M$ is isometric to the invisible domain of a proper achronal meridian in $\AdSP{2,1}$.
\end{cor}


\subsection{Genus $r=1$: examples}\label{sec:torus examples}
Our first objective is the classification of MGH AdS spacetimes of genus $1$. This case has not been considered in the paper of Mess. However it has been studied in the physics literature, for instance in \cite{zbMATH00745362} and \cite{zbMATH02172608}. We start by constructing a family of examples, which will later be shown to be all examples of genus $1$ up to isometry, thus providing a full classification.

Recall from Definition \ref{def:dual geodesics} the construction of dual spacelike lines, as in Figure \ref{fig:duallines}. In the $\PSL(2,\R)$ model, up to isometry the two dual spacelike lines $\mathcal L$ and $\mathcal L'$ can be taken of the form $\mathcal L=L_{\ell,\ell}$ where $\ell$ is an oriented spacelike geodesic in $\Hyp^2$, and $\mathcal L'=L_{\ell,\ell'}$ 
where $\ell'$ is $\ell$ endowed with the opposite orientation. This means that $\mathcal L$ consists of the hyperbolic isometries of $\Hyp^2$ which translate along the geodesic $\ell$, while $\mathcal L'$ consists of order-two elliptic elements with fixed point in $\ell$. By Proposition \ref{prop space of space geo}, the endpoints of $\mathcal L$ are of the form $(x,y)$ and $(y,x)$ in $\partial\AdSP{2,1}\cong\RP^1\times\RP^1$, for $x$ and $y$ the endpoints of $\ell$ in $\RP^1$, while the endpoints of $\mathcal L'$ are of the form $(x,x)$ and $(y,y)$. 

The following lemma exhibits proper achronal meridians in $\partial\AdSP{2,1}$ containing these four points, each of which, together with the two dual lines $\mathcal L$ and $\mathcal L'$, constitute the 1-skeleton of the affine tetrahedron as in Figure \ref{fig:duallines}.

 \begin{figure}[htb]
\includegraphics[height=8.5cm]{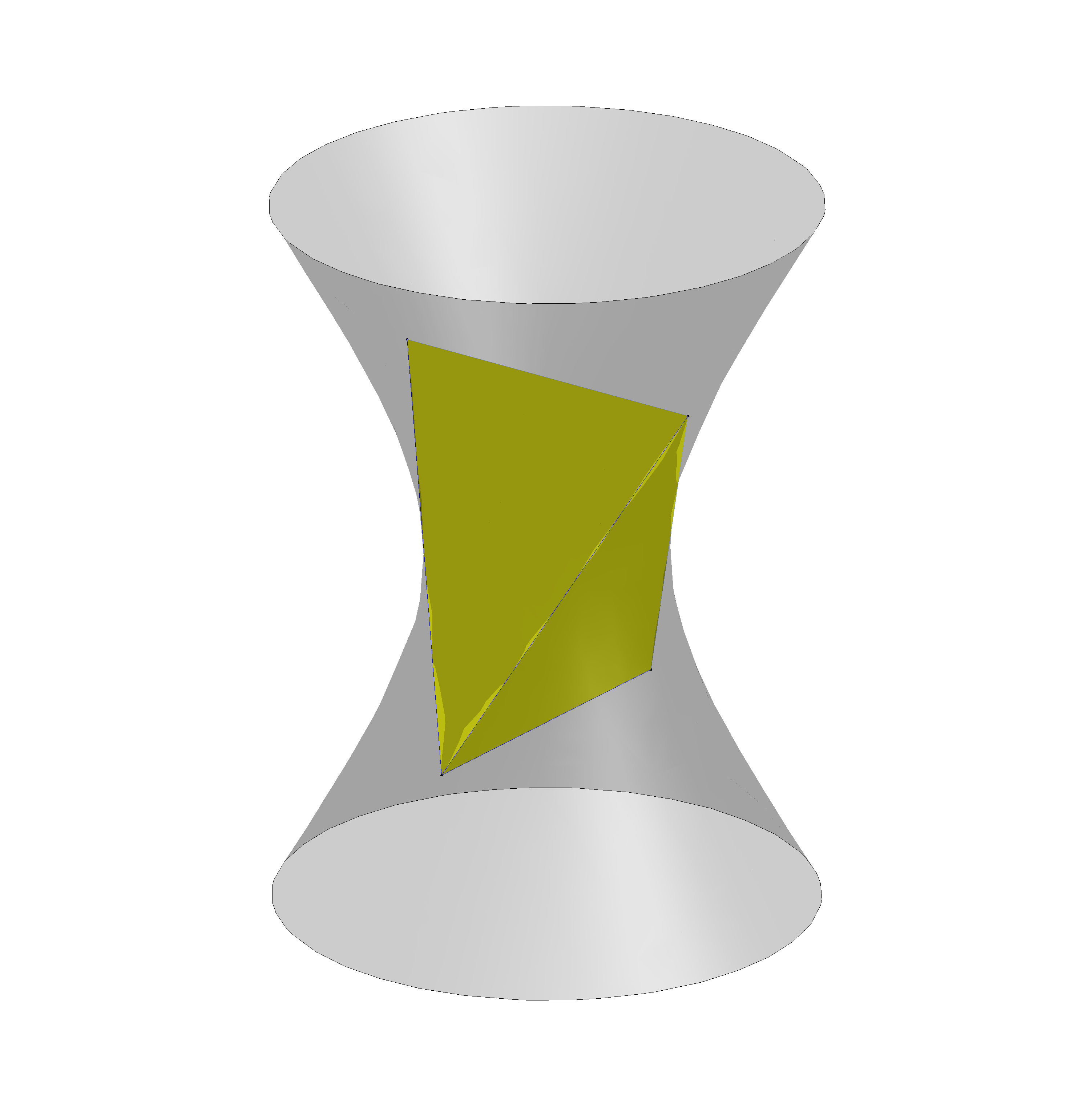}
\vspace{-1cm}
\caption{The lightlike tetrahedron $\mathscr T$: two of its edges are spacelike lines of $\AdSP{2,1}$, dual to one another (on the top and bottom), and the other four are lightlike segments contained in $\Quno$.}\label{fig:duallines}
\end{figure}

\begin{lemma} \label{lemma:twostep curves}
Let $x,y$ be different points in $\RP^1$. Then there exist exactly two proper achronal meridians in $\partial\AdSP{2,1}$ containing the points $(x,x),(y,x),(y,y),(x,y)$.  
\end{lemma}
\begin{proof}
Since $(x,x)$ and $(y,x)$ are in the same leaf $\lambda_x$ of left ruling of $\partial\AdSP{2,1}\cong\RP^1\times\RP^1$, a proper achronal meridian must necessily contain one of the two segments connecting $(x,x)$ and $(y,x)$ in $\lambda_x$, thus giving two possible choices. Once this choice is made, the same argument applies for the leaf $\mu_y$ of the right ruling  containing $(y,x)$ and $(y,y)$, but there is only one possible choice so as to give, concatenated with the previously chosen segment in $\lambda_x$, a locally achronal curve.

More precisely, if we choose an affine chart which contains the four points $(x,x)$, $(y,x)$, $(y,y)$ and $(x,y)$ and assume the segment chosen in the first step from $(x,x)$ to $(y,x)$ is future-directed in this affine chart, then the segment connecting $(y,x)$ to $(y,y)$ must necessarily be past directed. One then iterates this argument and obtains precisely two proper achronal meridians: if we assume for simplicity that $x=0$ and $y=\infty$ in $\RP^1\cong\R\cup\{\infty\}$, the first is the concatenation of $[0,\infty]\times\{0\}$, $\{\infty\}\times[0,\infty]$, $[\infty,0]\times\{\infty\}$ and $\{0\}\times[\infty,0]$; the second the concatenation of $[\infty,0]\times\{0\}$, $\{\infty\}\times[\infty,0]$, $[0,\infty]\times\{\infty\}$ and $\{0\}\times[0,\infty]$. See Figure \ref{fig:twostep}.
\end{proof}

 \begin{figure}[htb]
\includegraphics[height=5cm]{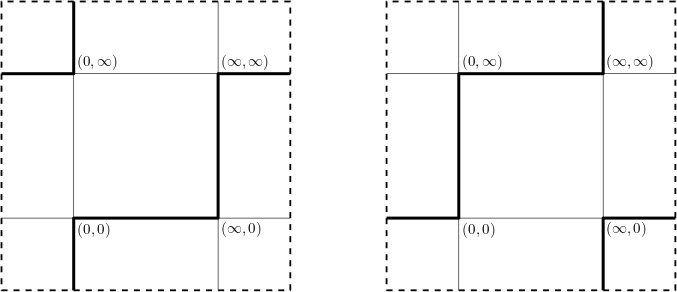}

\caption{A schematic picture of the two curves of Lemma \ref{lemma:twostep curves} in the torus $\RP^1\times\RP^1$, represented as a square with sides identified by translations.}\label{fig:twostep}
\end{figure}


Let us call $\Lambda_1$ and $\Lambda_2$ the two proper achronal meridians described in Lemma \ref{lemma:twostep curves}. Their lifts on the universal cover $\AdSU{2,1}$ are easily described. For this purpose, let us fix $x,y\in\RP^1$ and let us choose a lift $\widetilde{\mathcal L}$ to $\AdSU{2,1}$ of the spacelike geodesic in $\AdSP{2,1}$ connecting $p_1=(x,y)$ and $p_2=(y,x)$. Say $\widetilde p_1=(\xi_1, \mathsf t_1)$ and $\widetilde p_2=(\xi_2, \mathsf t_2)$ are the endpoints of $\widetilde{\mathcal L}$ in the boundary $\partial\D\times\R$ of the Poincar\'e model of $\AdSU{2,1}$. (Up to isometries, we could in fact assume that $\xi_1$ and $\xi_2$ are antipodal on $\Sp^1$ and $\mathsf t_1=\mathsf t_2=0$.)

Then $\widetilde \Lambda_1$ and $\widetilde \Lambda_2$ can be expressed as the graphs of $\mathsf f^{\Lambda_i}:\partial\D\to\R$ defined by:
\begin{equation}\label{eq twostep lift1}
    \mathsf f^{\Lambda_1}(\xi)=\min\{d_{\Sp^2}(\xi,\xi_1)+\mathsf t_1,d_{\Sp^2}(\xi,\xi_2)+\mathsf t_2\}~,
\end{equation}
and
\begin{equation}\label{eq twostep lift2}
    \mathsf f^{\Lambda_2}(\xi)=\max\{\mathsf t_1-d_{\Sp^2}(\xi,\xi_1),\mathsf t_2-d_{\Sp^2}(\xi,\xi_2)\}~.
\end{equation}
Indeed, for $\mathsf f^{\Lambda_1}$, since $(\xi_1, \mathsf t_1)$ and $(\xi_2, \mathsf t_2)$ are connected by a spacelike line, $\mathsf f^{\Lambda_1}(\xi_1)=\mathsf t_1$ and $\mathsf f^{\Lambda_1}(\xi_2)=\mathsf t_2$; moreover there are two points $\widetilde q_1=(\eta_1,\mathsf s_1)$ and $\widetilde q_2=(\eta_2,\mathsf s_2)$ at which the expressions $d_{\Sp^2}(\xi,\xi_1)+\mathsf t_1$ and $d_{\Sp^2}(\xi,\xi_2)+\mathsf t_2$ are equal, which are the endpoints of one lift of the dual line $\mathcal L'$. Hence the graph of $\mathsf f^{\Lambda_1}$ consists of four lightlike segments, two future-directed and two past-directed.
By the way, observe that $\mathsf f^{\Lambda_1}$ could be written by the equivalent expression:
\begin{equation}\label{eq twostep lift3}
    \mathsf f^{\Lambda_1}(\xi)=\max\{\mathsf s_1-d_{\Sp^2}(\xi,\eta_1),\mathsf s_2-d_{\Sp^2}(\xi,\eta_2)\}~.
\end{equation}

This analysis turns out to be extremely useful for the description of the invisible domain and the convex hull of $\Lambda_1$ and $\Lambda_2$. These are pictured in Figure \ref{fig:lightetra} below.

\begin{prop}\label{prop Lambda0}
Let $x,y$ be distinct points in $\RP^1$ and let $\Lambda_0$ be a proper achronal meridian in $\partial\AdSP{2,1}$ containing the points $(x,x),(y,x),(y,y),(x,y)$. Then $\overline{\Omega(\Lambda_0)}=C(\Lambda_0)$ is a tetrahedron bounded by four lightlike planes. 
\end{prop}
\begin{proof}
Let us first consider the picture in the universal cover $\AdSU{2,1}$, and consider the lift $\widetilde \Lambda_1$ defined as the graph of $\mathsf f^{\Lambda_1}$ as in Equation \eqref{eq twostep lift1}. As a simple consequence of the triangular inequality for the hemispherical metric, one sees that the functions $\mathsf f^{X}_-$ and $\mathsf f^{X}_+$ we introduced in Section \ref{subsec:invisible} and Lemma \ref{lem:ext} (where now $X=\widetilde\Lambda_1$) are given themselves by the expressions of Equations \eqref{eq twostep lift1} and \eqref{eq twostep lift3} respectively, except that the point $\xi$ is now allowed to vary in $\D\cup\partial\D$. 

Using the description of lightlike planes we gave in Section \ref{sec:geodesics}, see also Figure \ref{fig:causality}, the surfaces $S_\pm(\widetilde\Lambda_1)$ (which we recall are the graph of $\mathsf f^{X}_\pm$) consist of two lightlike half-planes meeting in a spacelike geodesic: the geodesic with endpoints $\widetilde q_1$ and $\widetilde q_2$ for $S_+(\widetilde\Lambda_1)$; the geodesic with endpoints $\widetilde p_1$ and $\widetilde p_2$ for $S_-(\widetilde\Lambda_1)$. Projecting down to $\AdSP{2,1}$, the same description holds for $S_\pm(\Lambda_0)$. Hence $\Omega(\Lambda_0)$ is the interior of a tetrahedron with lightlike faces. Its closure, which is the tetrahedron itself, clearly coincides with the convex hull of $\Lambda_0$ in an affine chart, which is also the convex hull of $\mathcal L\cup\mathcal L'$.
\end{proof}

\begin{figure}[htb]
\centering
\begin{minipage}[c]{.5\textwidth}
\centering
\includegraphics[height=8cm]{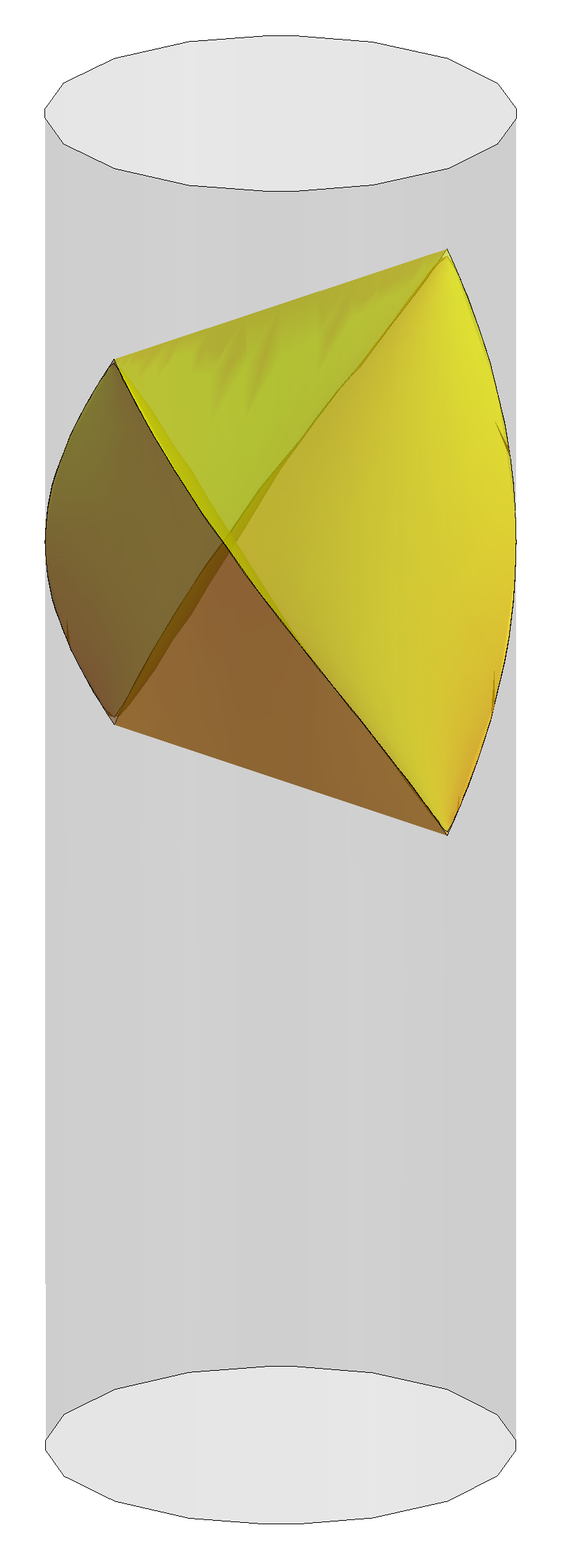}
\end{minipage}%
\begin{minipage}[c]{.5\textwidth}
\centering
\includegraphics[height=8cm]{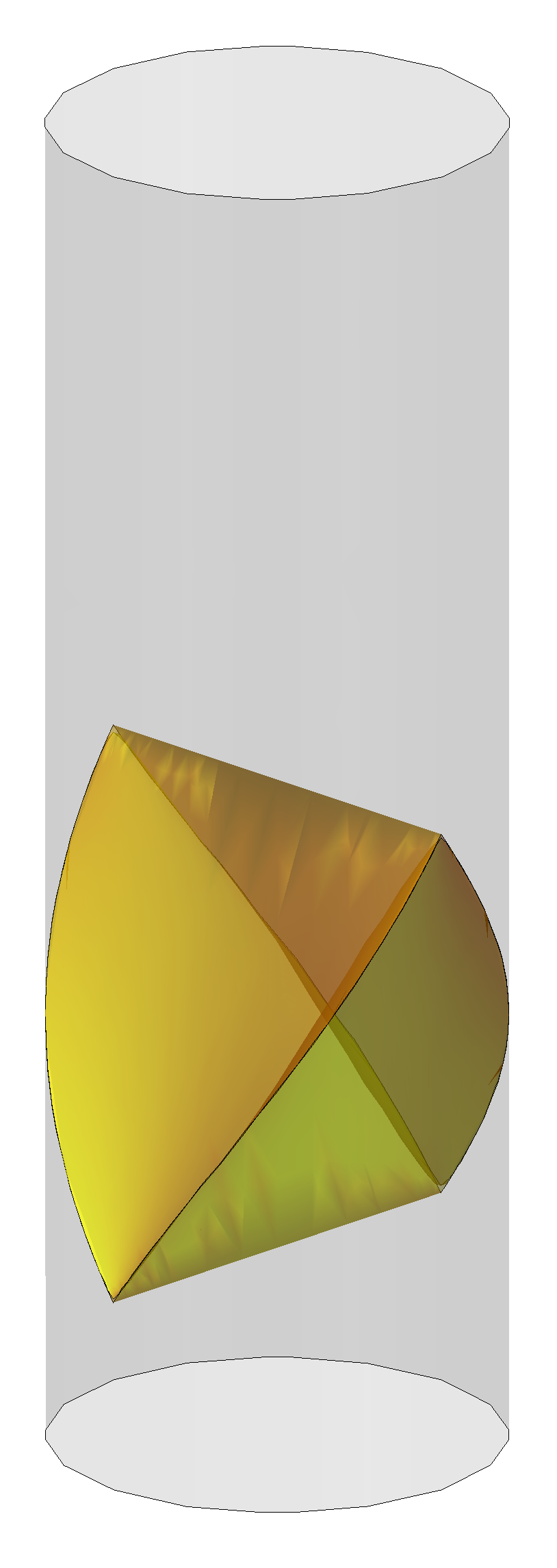}
\end{minipage}
\caption{The invisible domains of the two achronal meridians $\widetilde \Lambda_1$ and $\widetilde \Lambda_2$ composed of four lightlike segments in $\AdSU{2,1}$. The 1-skeleton of the two tetrahedra contains four lightlike segments together with two dual spacelike lines. The left and right tetrahedra actually differ by rotating on $\D$ and translating vertically.    \label{fig:lightetra}}
\end{figure}

\begin{remark} \label{rmk: two choices tetrahedra}
The region $\Omega(\Lambda_0)$ is, up to isometries, insensitive to the choice of $\Lambda_0$ as in Lemma \ref{lemma:twostep curves}. Namely, there is an orientation-preserving, time-preserving isometry of $\AdSP{2,1}$ which maps one proper achronal meridian as in Lemma \ref{lemma:twostep curves} to the other. The isometry is achieved simply by mapping the spacelike line $\mathcal L$ to its dual $\mathcal L'$, and therefore $\mathcal L'$ is mapped to $\mathcal L$.

In the universal cover $\AdSU{2,1}$, this isometry is easily expressed if we normalize $\widetilde{\mathcal L}$ so that its endpoints in $\partial\D\times\R$ are of the form $\widetilde p_1=(\xi_1, \mathsf t_1)$ and $\widetilde p_2=(\xi_2, \mathsf t_2)$ with $\mathsf t_1=\mathsf t_2$ and $\xi_1,\xi_2$ antipodal points in the sphere. Then the isometry we are looking for is induced by the isometry of $\AdSU{2,1}$ which acts as a rotation of angle $\pi/2$ on $\D$ and a vertical translation of $\pi/2$ on $\R$. See again Figure \ref{fig:lightetra}.
\end{remark}

In what follows, we will refer to the region $\Omega(\Lambda_0)$, which is uniquely determined up to isometries, as the \emph{lightlike tetrahedron} $\mathscr T$. To give a concrete description of the MGH spacetimes of genus one, the following geometric description of the tetrahedron, from an intrinsic point of view, will be useful.

\begin{lemma} \label{lemma:metric tetra}
The lightlike tetrahedron $\mathscr T$ is isometric to $\R^2\times (0,\pi/2)$ endowed with the Lorentzian metric 
\begin{equation}\label{eq:metric tetra}
\cos^2(z)dx^2+\sin^2(z)dy^2-dz^2~.
\end{equation}
\end{lemma}
\begin{proof}
The easiest way to perform this computation is in the quadric model $\AdS{2,1}$. Let us consider two lifts $\widehat{\mathcal L}$ and $\widehat{\mathcal L}'$ of the spacelike dual geodesics $\mathcal L$ and $\mathcal L'$ of $\AdSP{2,1}$. It follows from the discussion of the duality in Section \ref{sec:duality} that points in $\mathcal L'$ are the midpoints of the closed timelike geodesics leaving from $\mathcal L$ orthogonally. Hence in the double cover we have a timelike geodesic of length $\pi/2$ connecting every point of $\widehat{\mathcal L}$ to every point of $\widehat{\mathcal L}'$. Clearly these geodesics, projected to $\AdSP{2,1}$, foliate the interior of the convex hull of $\Lambda_0$, namely the lightlike tetrahedron $\mathscr T$.

Let $\gamma:\R\to \widehat{\mathcal L}$ and $\eta:\R\to \widehat{\mathcal L}$ be arclength parameterizations of the chosen spacelike geodesics in $\AdS{2,1}$. By virtue of the above description, and using the expression \eqref{eq:time geo quadric} for the geodesics in the quadric model, we have the following diffeomorphism $\Phi$ between $\R^2\times(0,\pi/2)$ and a lift of $\mathscr T$ in $\AdS{2,1}$:
$$\Phi(x,y,z)=\cos(z)\gamma(x)+\sin(z)\eta(y)~.$$
A direct computation, using that $\gamma(x)$ and $\eta(y)$ are orthogonal in $\R^{2,2}$ for every $x,y$, shows that the pull-back of the ambient metric $\langle,\cdot,\cdot\rangle_{2,2}$ of $\AdS{2,1}$ equals the metric \eqref{eq:metric tetra}.
\end{proof}

It is worth remarking that the surfaces given by $z=c$ under the diffeomorphism $\Phi$ are intrinsically flat and complete, hence properly embedded by Lemma \ref{lemma complete implies proper}. They are Cauchy surfaces for $\mathscr T$ by Proposition \ref{pr:spacelikeimm}. See Figure \ref{fig:foliationCMCflat}.

\begin{figure}[htb]
\centering
\begin{minipage}[c]{.5\textwidth}
\centering
\includegraphics[height=8cm]{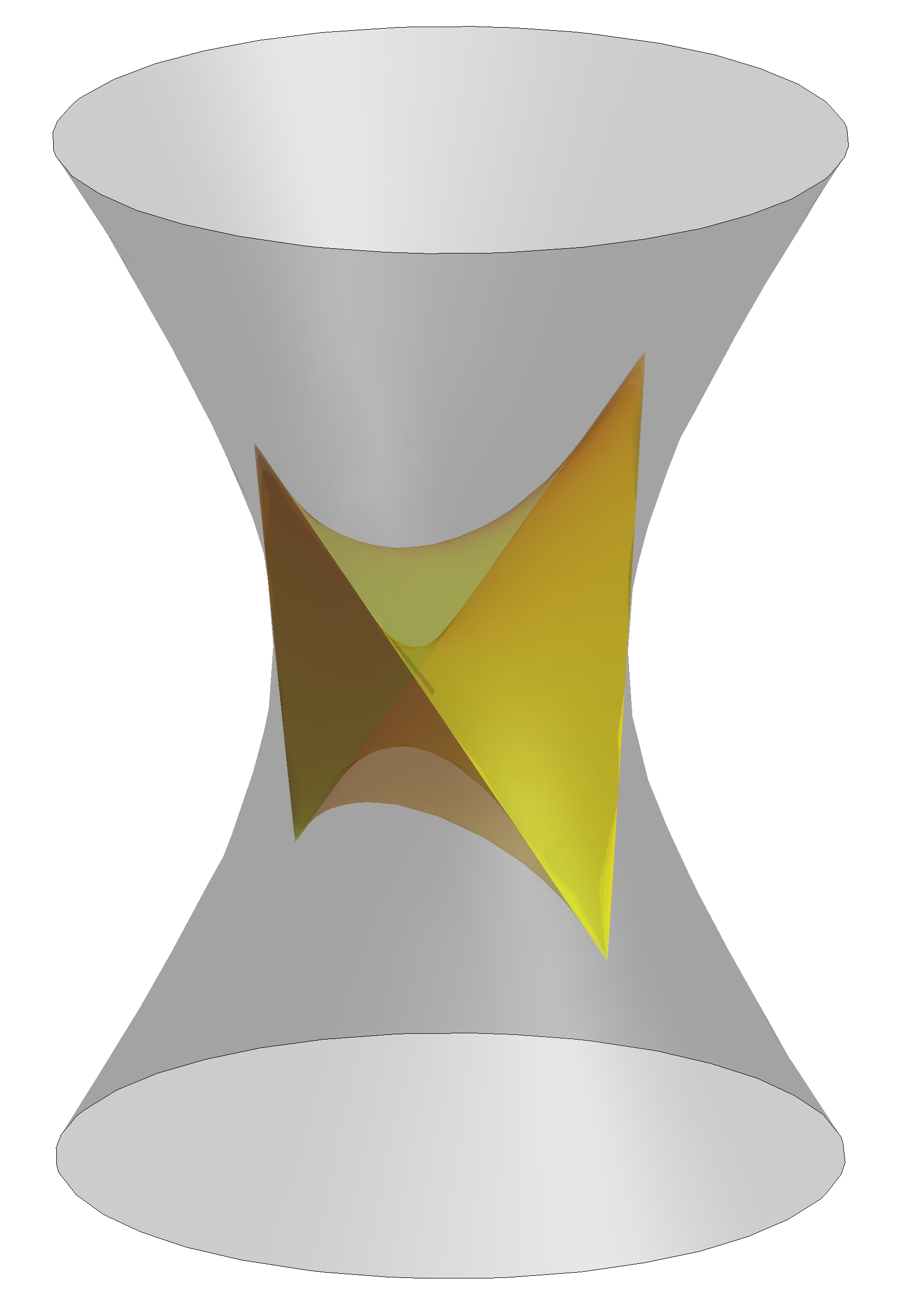}
\end{minipage}%
\begin{minipage}[c]{.5\textwidth}
\centering
\includegraphics[height=8cm]{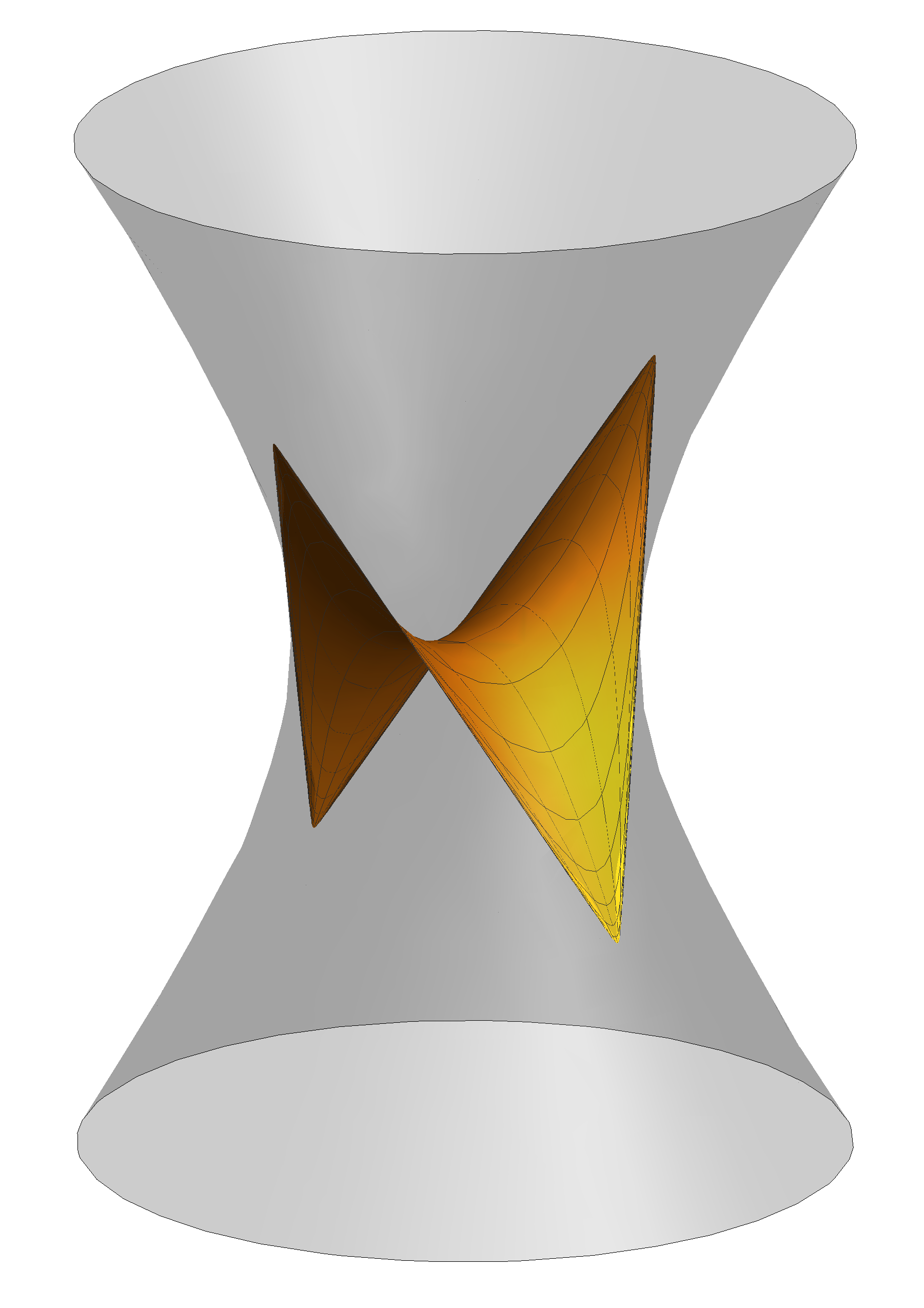}
\end{minipage}
\caption{The foliation of the lightlike tetrahedron $\mathscr T=\Omega(\Lambda_0)$ by flat CMC surfaces with constant values of $z$, in the coordinate system $\Phi$. On the right the maximal surface corresponding to $z=\pi/4$ is highlighted.    \label{fig:foliationCMCflat}}
\end{figure}

To conclude the construction of the examples, it only remains to study the stabilizer of the lines $\mathcal L$ and $\mathcal L'$. In light of the naturality of the construction of the dual line, the stabilizer of  $\mathcal L$ actually coincides with the stabilizer of $\mathcal L'$. In the $\PSL(2,\R)$-model, recall that we defined $\mathcal L$ as the one-parameter subgroup of $\PSL(2,\R)$ of hyperbolic transformations which fix a geodesic $\ell$ in $\Hyp^2$. The dual line consists of elliptic order-two isometries with fixed point on $\ell$. 

Let us denote by $\alpha_d$ the hyperbolic isometry which translates along $\ell$ of signed distance $d$. One then easily checks that the stabilizer of $\mathcal L$ which preserves an orientation of $\mathcal L$ is:
\begin{equation}\label{eq:stab geodesic}
\mathrm{Stab}_+(\mathcal L)=\{(\alpha_l,\alpha_m)\,|\,l,m\in\R\}\subset\PSL(2,\R)\times\PSL(2,\R)~,
\end{equation}
which is therefore isomorphic to $\R^2$. In fact, recalling the isometric  action of $\PSL(2,\R)\times\PSL(2,\R)$ on $\PSL(2,\R)$ from Equation \eqref{eq: action left right mult}, and the isometric identification of the dual plane $P_\En$ with $\Hyp^2$ (Lemma \ref{rmk dual plane traceless2}),
the isometries of the form $(\alpha_d,\alpha_d)$ fix  $\mathcal L$ pointwise and act on $\mathcal L'$ as a translation of length $d$. Conversely, the isometries of the form $(\alpha_d,\alpha_{-d})$ fix $\mathcal L'$ setwise and act on $\mathcal L$ as a translation of length $d$. 

The orientation-preserving, time-preserving stabilizer consists of the normal subgroup $\mathrm{Stab}_+(\mathcal L)$ and on another single coset, which consists of the rotations of angle $\pi$ along each of the timelike geodesics leaving $\mathcal L$ orthogonally and connecting $\mathcal L$ to the dual geodesic $\mathcal L'$. In conclusion, we have the following:

\begin{lemma}\label{lemma action on tetra}
The orientation-preserving, time-preserving stabilizer of  $\mathscr T$ is isomorphic to the semidirect product $\R^2\rtimes\Z/2\Z$. The normal subgroup $\R^2$ acts, in the coordinates given by Lemma \ref{lemma:metric tetra}, as 
$$(l,m)\cdot (x,y,z)=\left(x+\frac{l-m}{2},y+\frac{l+m}{2},z\right)~,$$
while a generator of the $\Z/2\Z$-factor acts as $(x,y,z)\mapsto (-x,-y,z)$. 
\end{lemma}

The full stabilizer of  $\mathscr T$ contains also orientation-reversing and time-reversing isometries, which can be easily figured out. Maximal globally hyperbolic spacetimes of genus 1 are then obtained as quotients of $\mathscr T$ by an action of $\Z^2$. 

\begin{prop}\label{pr:ex-genus1}
Given two linearly independent vectors $(l,m)$ and $(l',m')$, the group $\mathbb Z^2$ generated by $\alpha=(\alpha_l,\alpha_m)$ and $\alpha'=(\alpha_{l'}, \alpha_{m'})$ acts freely and properly discontinuously on $\mathscr T$ and the quotient is a 
MGH spacetime of genus $1$.
\end{prop}
\begin{proof}
The vectors $((l-m)/2,(l+m)/2)$ and $((l'-m')/2,(l'+m'/2)$ are linearly independent if and only if  $(l,m)$ and $(l',m')$ are linearly independent. It is then clear from Lemma \ref{lemma action on tetra}, using the coordinates of Lemma \ref{lemma:metric tetra}, that the action on $\mathscr T$ is free and properly discontinuous. Since any surface $\{z=c\}$ is a Cauchy surface in $\mathscr T$, they project to Cauchy surfaces in the quotient, which is therefore globally hyperbolic, and maximal by Proposition \ref{pr:MGHADS}. 
\end{proof}

\subsection{Genus $r=1$: classification}
In this section we will prove that any MGH spacetime of genus $r=1$ is isometric to one of those constructed in Proposition \ref{pr:ex-genus1}. The key step in the argument is the following proposition.



\begin{prop}\label{pr:genus1-hol}
Let $M$ be a globally hyperbolic spacetime of genus $r=1$ and let $\rho=(\rho_l,\rho_r):\pi_1(T^2)\to\PSL(2,\R)\times\PSL(2,\R)$ be the holonomy representation.
Then $\rho$ is discrete and faithful. Moreover $\rho_l$ and $\rho_r$ are elementary representations  with two fixed points in $\RP^1$.
\end{prop}
The last property in the statement  is equivalent to the fact that $\rho_l(\gamma)$ and $\rho_r(\gamma)$ are hyperbolic transformations for any $\gamma\in\pi_1(T^2)$.
\begin{proof}
By Proposition \ref{pr:MGHADS} the developing map $\dev:\widetilde M\to\AdSP{2,1}$ is injective, which implies that the holonomy representation is faithful.
Moreover $\dev(\widetilde M)$ is a domain in $\AdSP{2,1}$ on which $\rho(\pi_1(T^2))$ acts properly. It follows that $\rho(\pi_1(T^2))$  is a discrete subgroup of $\PSL(2,\R)\times\PSL(2,\R)$.
The fact that $\rho_l$ and $\rho_r$ are elementary representations is a simple consequence of the fact that $\pi_1(T^2)$ is abelian.
In order to prove that $\rho_l$ and $\rho_r$ fix two points on $\RP^1$ we will show that no other possibility can hold.

First assume that both $\rho_l$ and $\rho_r$ have  a fixed point in $\Hyp^2$.
Then $\rho$ is conjugate to a representation in $\mathrm{PSO}(2)\times \mathrm{PSO}(2)$. But there is no faithful and discrete representation of $\pi_1(T^2)$ into a compact group.

To exclude the other possibilities we will use that, by Proposition \ref{pr:MGHADS}, there is a proper achronal meridian $\Lambda$ in $\partial\AdSP{2,1}=\RP^1\times\RP^1$ invariant under the representation $\rho$.

Assume first that $\hol_l$ fixes a point in $\Hyp^2$, and $\hol_r$ fixes (at least) a point $y_0\in\RP^1$.
For homological reasons the curve $\Lambda$ must intersect the leaf $\lambda_{y_0}$ at a point, say $p_0=(x_0, y_0)$.
Let $\gamma$ be a non-trivial element of $\pi_1(T^2)$, and set $p_1:=\hol(\gamma)(x_0, y_0)=(\hol_l(\gamma)x_0, y_0)$. So $\Lambda$ meets $\lambda_{y_0}$ also at $p_1$. By Lemma \ref{lem:llike}, $\Lambda$ contains 
a lightlike segment $I$ in $\lambda_{y_0}$ with end-points $p_0$ and $p_1$. Since $\bigcup_{n} \hol(\gamma)^n(I)=\lambda_{y_0}$ we deduce that $\Lambda$ contains the entire leaf $\lambda_{y_0}$ but this is a contradiction with $\Lambda$ being a proper achronal meridian.

Let us now consider the case that $\hol_l(\gamma)$ and $\hol_r(\gamma)$ are  parabolic transformations for all $c\in\pi_1(T)$. 
Up to conjugation we may assume that the fixed points of $\rho_l$ and $\rho_r$ are both $\infty\in\RP^1$, hence $\hol$ takes values into the subgroup $G_{\infty}$ of $\PSL(2,\R)\times\PSL(2,\R)$ fixing $p_\infty=(\infty, \infty)$. 
Notice that $G_\infty$ acts by translations on the domain 
$$U_0=\RP^1\times\RP^1\setminus(\lambda_{\infty}\cup\mu_{\infty})=(\RP^1\setminus\{\infty\})\times(\RP^1\setminus\{\infty\})=\R^2~,$$ and such an action provides
an isomorphism $G_\infty\cong\R^2$.
Since the holonomy is  discrete and faithful, $\hol(\pi_1(T^2))$ is identified to a lattice of $G_\infty$.
This implies that for every $p=(x_0, y_0)\in U_0$, the orbit of $p$ is the set of vertices of a  tessellation of $\R^2$ by parallelograms. 
In particular such an orbit must contain points of $\fut_{U_0}(p)=\{(x,y)\,|\, x-x_0>0,\ y-y_0<0\}$, which shows that in $U_0$ there is no achronal orbit for the action of $\pi_1(T^2)$.
It follows that $\Lambda$ cannot meet $U_0$, so it is contained in $\lambda_\infty\cup\mu_\infty$.
On the other hand, arguing as above we see that if $\Lambda$ intersects the leaf $\lambda_{\infty}$ (resp. $\mu_\infty)$ at a point different from $p_{\infty}$, then it must contain the whole leaf $\lambda_{\infty}$ (resp. $\mu_\infty$),
and this gives a contradiction. 

Finally consider the case where for all $\gamma\in\pi_1(T^2)$ we have $\hol_l(\gamma)$  parabolic, and $\hol_r(\gamma)$ hyperbolic. We can assume that $\infty$ is the fixed point of $\hol_l$, and $0,\infty$ are the fixed points of $\hol_r$.
We consider the partition of $\RP^1\times\RP^1$ into $\hol$-invariant subsets $\mu_{\infty}$, $\lambda_0$, $\lambda_\infty$,  $U_+=\R\times\R_{+}$, and $U_-=\R\times\R_{-}$.
We will prove that no $\pi_1(T^2)$-orbit of $U_\pm$ is achronal, showing that $\Lambda\subset\mu_\infty\cup\lambda_0\cup\lambda_\infty$.
Let $G$ be the subgroup of $\PSL(2,\R)\times\PSL(2,\R)$ made of elements whose left factor is a parabolic transformation with fixed point at $\infty$ and whose right factor hyperbolic transformation with fixed points $0,\infty$.
Let us consider the diffeomorphism 
$$\Phi:\R^2\to U_+\qquad \Phi(x,y)=(x, e^y)~,$$ which  conjugates the action of $G_\infty$ on $\R^2$ and of $G$ on $U_+$. In particular $\Phi^{-1}\circ\hol(\pi_1(T^2))\circ\Phi$ is a lattice in $G_\infty$.
Thus as before  no $\Phi^{-1}\circ\hol(\pi_1(T^2))\circ\Phi$-orbit in $\R^2$ is achronal. But $\Phi$ is conformal with respect to the Lorentzian metric $dxdy$ on $\R^2$ and the conformal Lorentzian structure of   $\partial\AdSP{2,1}$ restricted to $U_+$. We deduce that no $\hol(\pi_1(T^2))$-orbit is achronal in $U_+$.
A similar proof works for $U_-$. 

This shows that $\Lambda$ is contained in $\mu_\infty\cup\lambda_0\cup\lambda_\infty$. Hence $\Lambda\cap\lambda_0$ is either one point or an arc. In the latter case the end-points of the arc should lie in $\rho_\infty$, but the intersection
of $\lambda_0$ and $\mu_\infty$ is only at the  point $(\infty,0)$, which contradicts that $\Lambda$ is a proper achronal meridian. So  $\Lambda$ intersects $\lambda_0$ only at $(\infty,0)$. Similarly $\Lambda$ intersects $\lambda_\infty$ only at $(\infty,\infty)$. 
This implies that $\Lambda\subset \mu_\infty$ which is a contradiction.
\end{proof}

Now, given a pair of elementary representations $\rho_l,\rho_r:\Z^2\to\PSL(2,\R)$ which map every non-trivial element to a hyperbolic transformation, assume for simplicity that the fixed points of $\rho_l$ and $\rho_r$ coincide, and let us call them $x$ and $y$.  Recall from Lemma \ref{lemma:twostep curves} that there are two proper achronal meridians containing the four points 
$(x,x),(y,x),(y,y),(x,y)$ in $\partial\AdSP{2,1}$. Each of them is clearly invariant under the $\Z^2$-action induced by $\rho$. The next step consists in showing that these are the only invariant proper achronal meridians.

\begin{prop}\label{prop invariant meridians torus}
Let $\hol:\pi_1(T^2)\to\PSL(2,\R)\times\PSL(2,\R)$ be a representation such that $\hol_l$ and $\hol_r$ are elementary representations with two fixed points in $\RP^1$. Then there are exactly two proper achronal meridians in $\partial\AdSP{2,1}$ which are invariant under the action of  $\pi_1(T^2)$ induced by $\rho$.
\end{prop}
\begin{proof}
Up to conjugation  we may assume that both $\hol_l$ and $\hol_r$ fix $0$ and $\infty$.
It will be sufficient to show that any $\hol$-invariant proper achronal meridian $\Lambda$ necessarily contains the four points $(0,0),(0,\infty),(\infty,0),(\infty,\infty)$. Indeed by Lemma \ref{lemma:twostep curves} this will imply that $\Lambda$ is either $\Lambda_1$ or $\Lambda_2$.

We claim that $\Lambda$  must be contained in the union of leaves
\[
      X=\lambda_0\cup\lambda_\infty\cup\mu_0\cup\mu_\infty\,.
\]
First let us show how to conclude assuming the claim. Notice that the leaves $\lambda_j$ and $\mu_i$ meet at points $p_{i,j}=(i,j)$ for $i,j=0,\infty$.
If $\Lambda$ is an achronal meridian contained in $X$, then it must be a concatenation of arcs on the leaves $\lambda_0,\lambda_\infty,\mu_0,\mu_\infty$ with end-points
in $\{p_{i,j}\,|\,i,j=0,\infty\}$.  Notice that
\begin{itemize}
\item If $\Lambda$ contains an arc on $\lambda_j$ (resp. $\mu_i$) then it contains both $p_{0,j}$, and $p_{\infty,j}$ (resp. $p_{i,0}$ and $p_{i, \infty}$).
\item If  $p_{ij}$ is contained in $\Lambda$, then $\Lambda$ contains an arc on both $\lambda_i$ and $\mu_j$ (otherwise $\Lambda$ should contain a leaf).
\end{itemize}
In particular we easily deduce that $\Lambda$ must contain all points $p_{i,j}$ and we conclude.

In order to prove the claim we will check that no point in $\partial\AdSP{2,1}\setminus X$ has an achronal orbit.
Notice that $\partial\AdSP{2,1}\setminus X$ has the following four connected components: $$U_{+,+}=\R_+\times\R_+,\,U_{+,-}=\R_+\times\R_-,\,U_{-,+}=\R_-\times\R_+,\,U_{-,-}=\R_-\times\R_-~.$$
Each of these components is preserved by $\hol$. Let us focus on $U_{+,+}$.  Using the notation of Proposition \ref{pr:genus1-hol} consider the diffeomorphism
\[
 \Phi_{+,+}:\R^2\to U_{+,+}\qquad \Phi_{+,+}(x,y)=(e^x, e^y)\,,
\]
which is conformal, similarly to the last part of the  proof of Proposition \ref{pr:genus1-hol}. 
Let $\widehat G=\mathrm{Stab}_+(\mathcal L)$ be the stabilizer of the geodesic $\mathcal L=L_{\ell,\ell}$ preserving an orientation, as in \eqref{eq:stab geodesic}, where $\ell$ is the oriented geodesic of $\Hyp^2$ with endpoints $0$ and $\infty$. Namely $\widehat G$ is the subgroup of $\PSL(2,\R)\times\PSL(2,\R)$ of pairs of hyperbolic transformations with fixed points at $0,\infty$. Then $\Phi_{+,+}$ conjugates $G_\infty$ and $\widehat G$.
As in Proposition \ref{pr:genus1-hol}, we deduce that $\Phi_{+,+}^{-1}\circ \hol(\pi_1(M))\circ\Phi_{+,+}$ is a lattice in $G_\infty=\R^2$ and therefore the action cannot have achronal orbits in $\R^2$.
Since $\Phi_{+,+}$ is conformal, then the action of $\hol(\pi_1(M))$ cannot have achronal orbits in $U_{+,+}$. The proof for the other connected components $U_{\pm,\pm}$ is completely analogous.
\end{proof}

A consequence of Proposition \ref{prop invariant meridians torus} is the following. Recall from Section \ref{sec:torus examples} that $\mathscr T$ denotes a lightlike tetrahedron whose boundary in $\partial\AdSP{2,1}$ is a proper achronal meridian consisting of the concatenation of four lightlike segments. In Lemma \ref{lemma action on tetra} we showed that $\widehat G=\mathrm{Stab}_+(\mathcal L)$, which is the orientation-preserving, time-preserving stabilizer of $\mathscr T$, is isomorphic to the semi-direct product $\R^2\rtimes\Z/2\Z$.

\begin{cor}\label{cor genus 1 are quotients}
Any MGH spacetime of genus one is isometric to a quotient of $\mathscr T$ by a subgroup of $\widehat G$ acting freely and properly discontinuously on $\mathscr T$.
\end{cor}
\begin{proof}
By Proposition \ref{pr:MGHADS}, any MGH spacetime $M$ of genus one is isometric to the quotient of the invisible domain of a proper achronal meridian invariant under the action of $\rho(\pi_1(T^2))$, where $\rho:\pi_1(T^2)\to\PSL(2,\R)\times\PSL(2,\R)$ is the holonomy representation. By Proposition \ref{pr:genus1-hol}, $\rho$ maps every non-trivial element to a pair of hyperbolic transformations, and by Proposition \ref{prop invariant meridians torus} there are exactly two proper achronal meridians invariant under such a $\rho$, namely those described in Lemma \ref{lemma:twostep curves}. However, by Remark \ref{rmk: two choices tetrahedra}, there is an orientation-preserving, time-preserving isometry of $\AdSP{2,1}$ sending one invariant proper achronal meridian to the other. Hence, up to composing with an isometry, we see that $M$ is isometric to a quotient of $\mathscr T$, which is the invisible domain of the proper achronal meridian $\Lambda_0$ as in Proposition \ref{prop Lambda0}.
\end{proof}

Let us conclude this section by a discussion on the classification of MGH spacetimes of genus one. For this purpose, we introduce the \emph{deformation space}
$$\MADS(T^2)=\{g\text{ MGH AdS metric on }T^2\times\R\}/\mathrm{Diff}_0(T^2\times\R)~,$$
where the group of diffeomorphisms isotopic to the identity acts by pull-back. It is a well-known fact from the theory of $(G,X)$-structures that the holonomy map, which is well-defined with image in the space of representations of the fundamental group into $G$ up to conjugacy (in this case $G=\PSL(2,\R)\times \PSL(2,\R)$), descends to the quotient $\MADS(T^2)$.

Now, Corollary \ref{cor genus 1 are quotients} tells us that MGH spacetimes of genus 1 are determined by the holonomy representations of $\Z^2$ which take value in $\widehat G$ and act freely and properly discontinuously on $\mathscr T$. 

Two MGH spacetimes $\mathscr T/\rho_1(\Z^2)$ and $\mathscr T/\rho_2(\Z^2)$ represent the same point in $\MADS(T^2)$ if and only if $\rho_1$ and $\rho_2$ are conjugate in $\isom(\AdSP{2,1})$, but in fact in this case they are necessarily conjugate in $\widehat G$.
Hence the deformation space $\MADS(T^2)$ is identified to the space of $\Z^2$-representations in $\widehat G$ acting freely and properly discontinuously on $\mathscr T$ up to conjugacy in $\widehat G$.

By the proof of Proposition \ref{pr:ex-genus1} we see that $\rho(\pi_1(T^2))$ acts freely and properly discontinuously on $\mathscr T$ if and only if, under the isomorphism between  $\widehat G$ and the semi-direct product $\R^2\rtimes\Z/2\Z$, its acts freely and properly discontinuously on $\R^2$. Under this isomorphism, conjugacy by elements in the normal subgroup $\R^2$ do not change $\rho$, while conjugacy by the generator of $\Z/2\Z$ acts by minus the identity. In conclusion, we have the following classification result:

\begin{theorem}
The deformation space $\MADS(T^2)$ is homeomorphic to the space of discrete and faithful representations of $\pi_1(T^2)$ into $\R^2$ up to sign. 
\end{theorem}

As a final comment, the space of discrete and faithful representations of $\pi_1(T^2)$ into $\R^2$ coincides with the space of translation structures on the torus. Since they are considered up to sign change, $\MADS(T^2)$ is identified to the deformation space of \emph{semi-translation structures} on the torus.

\subsection{Genus $r\geq 2$: examples}

Let us now consider  $\Sigma_r$ an oriented surface of genus $r\geq 2$.
 Let us recall the definition of Fuchsian representations.

\begin{defi} \label{defi pos Fuch}
A representation $\rho:\pi_1(S)\to\PSL(2,\R)$ is \emph{positive Fuchsian} if there is a $\rho$-equivariant orientation-preserving homeomorphism $\delta:\widetilde \Sigma_r\to\Hyp^2$.
\end{defi}
The definition is invariant under conjugation in $\PSL(2,\R)\cong\Isom_0(\Hyp^2)$, but not under conjugation in $\Isom(\Hyp^2)$. By a celebrated result by Goldman \cite{MR2630832}, a representation $\rho$ is positive Fuchsian if and only if  the associated flat
$\RP^1$ bundle $E_\rho$, constructed as the quotient of $\widetilde \Sigma_r\times\R P^1$ by the diagonal action of $\pi_1(S)$ given by the obvious action by deck transformation on the first factor, and by $\rho$ on the second factor, has Euler class $2-2r$.  This is also equivalent to the existence of an orientation-preserving fiber bundle isomorphism between $E_\rho$ and the unit tangent bundle of $\Sigma_r$.

The following classical fact in Teichm\"uller theory, see for instance \cite{MR1730906}, is essential for the construction of MGH spacetimes of genus $r\geq 2$. 
\begin{lemma} \label{lemma ext circle homeo}
Given two positive Fuchsian representations $\rho_l,\rho_r:\pi_1(\Sigma_r)\to\PSL(2,\R)$, any $(\rho_l, \rho_r)$-equivariant orientation-preserving homeomorphism of $\Hyp^2$, which exist as a consequence of Definition \ref{defi pos Fuch}, extends continuously to an orientation-preserving homeomorphism of  $\Hyp^2\cup\RP^1$. Moreover, its extension $\varphi:\RP^1\to \RP^1$ is the unique 
$(\rho_l, \rho_r)$-equivariant orientation preserving homeomorphism of $\RP^1$. 
 \end{lemma}
 
By  $(\rho_l, \rho_r)$-equivariance of $\varphi$ we mean the condition that for every $\gamma\in\pi_1(S)$:
\begin{equation}\label{eq:equivariance}
\varphi\circ\rho_l(\gamma)=\rho_r(\gamma)\circ\varphi~.
\end{equation}
Now let $\rho_l, \rho_r:\pi_1(\Sigma_r)\to\PSL(2,\R)$ be two positive Fuchsian representations. We will consider the representation 
$$\rho=(\rho_l, \rho_r):\pi_1(S)\to\isom_0(\AdSP{2,1})\cong\PSL(2,\R)\times\PSL(2,\R)\cong\isom(\AdSP{2,1})~.$$

\begin{defi}
Given a pair of positive Fuchsian representations $\rho_l, \rho_r:\pi_1(\Sigma_r)\to\PSL(2,\R)$, we define $\Lambda(\rho)$ to be the graph in $\RP^1\times\RP^1$ of the unique  $(\rho_l, \rho_r)$-equivariant orientation-preserving homeomorphism of $\RP^1$, and $\Omega_\rho:=\Omega(\Lambda(\rho))$ its invisible domain in $\AdSP{2,1}$.
\end{defi}
Using the above construction, we can build examples of MGH spacetimes having holonomy any $\rho=(\rho_l, \rho_r)$ of this form.

\begin{prop}\label{pr:xxx}
The domain $\Omega_\rho$ is invariant under the isometric action of $\pi_1(\Sigma_r)$ on $\AdSP{2,1}$ induced by $\rho$. Moreover $\pi_1(\Sigma_r)$ acts freely and properly discontinuously on $\Omega_\rho$ and the quotient
 is a MGH spacetime of genus $r$ and holonomy $\rho$.
\end{prop}
\begin{proof}
By the definition of $\varphi$ and the action of $\PSL(2,\R)\times\PSL(2,\R)$ on it is clear that $\Lambda(\rho)$ is invariant by the action of $(\rho_l(\gamma),\rho_r(\gamma))$, for every $\gamma\in\pi_1(\Sigma_r)$. Recalling from Corollary \ref{cor:inv-homeo} that $\Omega_\rho$ is the set of elements $x\in\PSL(2,\R)$ such that  $x\circ\varphi$ have no fixed point on $\RP^1$, the invariance of $\Omega_\rho$ also follows immediately: indeed 
$$(\rho_l(\gamma)\circ x \circ \rho_r(\gamma)^{-1})\circ \varphi=\rho_l(\gamma)\circ (x\circ \varphi)\circ \rho_l(\gamma)^{-1}$$
acts freely on $\RP^1$ if $x\circ \varphi$ does.

Let us show that for a compact set $K$ in $\Omega_\rho$, $\rho(\gamma)\cdot K$ stays in a compact region of $\Omega_\rho$ only for finitely many $\gamma\in\pi_1(\Sigma_r)$. This will also show that the action is free, since $\pi_1(\Sigma_r)$ has no torsion. For this purpose, take a sequence $x_n\in K$ and a sequence $\gamma_n\in \pi_1(\Sigma_r)$ not definitively constant. We claim that up to a subsequence, $(\rho(\gamma_n)\cdot x_n)$ converges to some $(\xi_+,\varphi(\xi_+))$ in $\Lambda(\rho)$. We will apply the criterion of convergence to $\partial\AdSP{2,1}$ given in Lemma \ref{lemma convergence at infinity}. 
 
Since Fuchsian representations act cocompactly on $\Hyp^2$, the sequence $\rho_l(\gamma_n)$ has no converging subsequences in $\PSL(2,\R)$. By a well-known dynamical property of $\PSL(2,\R)$ (see for instance \cite{MR2217992}), up to taking a subsequence, there exist $\xi_-,\xi_+$ on $\RP^1$ such that $\rho_l(\gamma_n)^{\pm 1}(\xi)\to\xi_{\pm}$ for all $\xi\neq\xi_{\mp}$ and that the convergence is 
uniform on compact sets of $(\Hyp^2\cup\RP^1)\setminus\{\xi_\mp\}$. By the equivariance \eqref{eq:equivariance}, the same holds for $\rho_r(\gamma_n)$ where now $\xi_\pm$ are replaced by $\varphi(\xi_\pm)$. 

To apply the criterion of Lemma \ref{lemma convergence at infinity}, pick any $p\in\Hyp^2$, and recall that $\rho(\gamma_n)\cdot x_n=\rho_l(\gamma_n)\circ x_n \circ\rho_r(\gamma_n)^{-1}$. By the dynamical property above, for any $\delta>0$ one can find $n_0$ such that $\rho_r(\gamma_n)^{-1}(p)$ is in the $\delta$-neighborhood of $\varphi(\xi_-)$ (for the Euclidean metric on the closed disc), say $U_\delta$.  Since $x_n$ lies in a compact region of $\Omega_\rho$, we can assume that it converges to $x_\infty\in \Omega_\rho$, hence $x_\infty\circ \varphi$ has no fixed point, and in particular $x_\infty\circ \varphi(\xi_-)\neq \xi_-$. 

Up to taking $\delta$ sufficiently small and $n_0$ large, $x_n(U_\delta)$ lies in a neighborhood $V_\epsilon$ of $x_\infty\circ\varphi(\xi_-)$ such that the closure of $V_\epsilon$ is disjoint from $\xi_-$. By construction $x_n\circ \rho_r(\gamma_n)^{-1}(p)\in V_\epsilon$ and by the uniform convergence on compact sets of the complement of $\xi_-$,  $\rho_l(\gamma_n)\circ x_n\circ \rho_r(\gamma_n)^{-1}(p)$ converges to $\xi_+$ for $n$ large. The very same argument then shows that $(\rho(\gamma_n)\cdot x_n)^{-1}(p)=\rho_r(\gamma_n)\circ x_n^{-1} \circ\rho_l(\gamma_n)^{-1}(p)$ converges to $\varphi(\xi_+)$. This concludes the claim.




Finally, the past and future boundary components $\partial_\pm C(\Lambda(\rho))$ are contained in $\Omega_\rho$, since $\Lambda(\rho)$ is the graph of an orientation-preserving homeomorphism (see Remark \ref{rmk:light triangles convex invisible}). Hence they are $\rho$-invariant properly embedded Cauchy surfaces in $\Omega_\rho$ and  project to Cauchy surfaces of the quotient by the action of $\rho(\pi_1(\Sigma_r))$, which are homeomorphic to $\Sigma_r$. This shows that the quotient is a globally hyperbolic spacetime of genus $r$, which is maximal by Proposition \ref{pr:MGHADS}.
\end{proof}

\begin{figure}[htb]
\includegraphics[height=6cm]{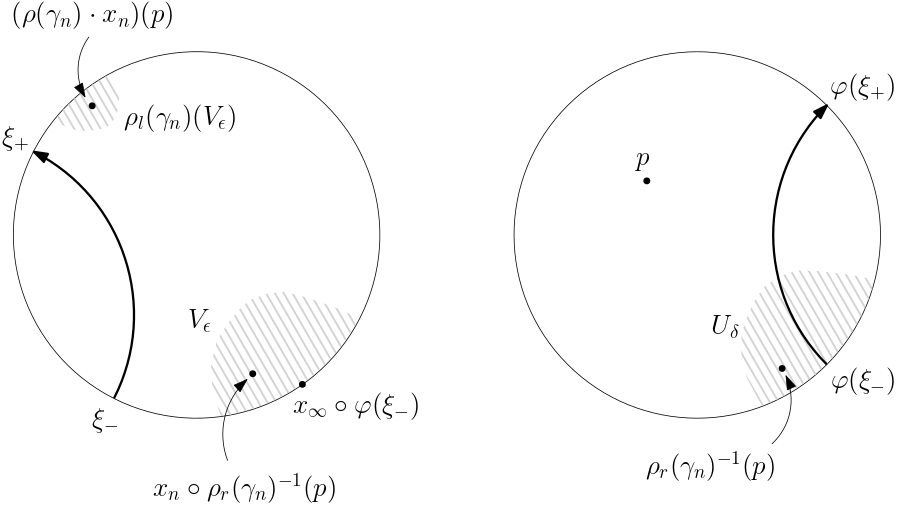}

\caption{The proof of Proposition \ref{pr:xxx}, and in particular the fact that the image of $p\in\Hyp^2$ under $\rho(\gamma_n)\cdot x_n=\rho_l(\gamma_n)\circ x_n \circ\rho_r(\gamma_n)^{-1}$ converges to $\xi_+$ as $n\to +\infty$. A completely analogous argument shows that the image of $p$ under the inverse converges to $\varphi(\xi_+)$.}\label{fig:dynamics}
\end{figure}

\subsection{Genus $r\geq 2$: classification}

In this section we will conclude the classification result, by showing essentially that the examples of Proposition \ref{pr:xxx} are \emph{all} the MGH spacetimes of genus $r$.

\begin{lemma}
Let $\rho=(\rho_l, \rho_r)$ be a pair of positive Fuchsian representations, and $\varphi:\RP^1\to\RP^1$ be the unique $(\rho_l, \rho_r)$-equivariant orientation-preserving homeomorphism of $\RP^1$.
Then $\Lambda(\rho)$ is the unique proper achronal meridian in $\partial\AdSP{2,1}$ invariant under the action of $\pi_1(\Sigma_r)$ induced by $\rho$.
\end{lemma}
\begin{proof}
Let $\Lambda$ be a proper achronal meridian invariant under the action of $\pi_1(\Sigma_r)$. We claim that the intersection $\Lambda\cap\Lambda_{\varphi}$ is not empty.
Once the claim will be showed, the proof is concluded in the following way.
If $(\xi_0,\varphi(\xi_0))\in\Lambda$, then 
$$(\rho_l(\gamma)\cdot \xi_0, \varphi(\rho_l(\gamma)\cdot\xi_0))=(\rho_l(\gamma)\cdot\xi_0, \rho_r(\gamma)\cdot\varphi(\xi_0))\in\Lambda~.$$
However the $\rho_l(\pi_1(\Sigma_r))$-orbit of $\xi_0$
is dense in $\RP^1$, hence we deduce that  $\Lambda$ contains $\Lambda_\varphi$.
But both $\Lambda_\varphi$ and $\Lambda$ are homeomorphic to $S^1$, which necessarily implies $\Lambda_\varphi=\Lambda$.

Let us prove the claim.
Let $\gamma$ be a non-trivial element in $\pi_1(\Sigma_r)$. It is known that $\rho_l(\gamma)$ and $\rho_r(\gamma)$ are necessarily hyperbolic elements in $\PSL(2,\R)$, hence we denote by $\xi^+_l(\gamma)$, and $\xi^+_r(\gamma)$ their attractive fixed points respectively. 
Notice that $\xi^+_r(\gamma)=\varphi(\xi^+_l(\gamma))$, hence
$$(\xi^+_l(\gamma), \xi^+_r(\gamma))\in\Lambda_\varphi~.$$
By homological reasons the curve $\Lambda$ must meet the leaf of the left ruling of $\partial\AdSP{2,1}$:
 $$\lambda_{\xi^+_r(\gamma)}=\{(\eta, \xi^+_r(\gamma))\,|\,\eta\in\RP^1\}~.$$
That is, there exists $\eta_0\in\RP^1$ such that $(\eta_0, \xi^+_r(\gamma))$ lies in $\Lambda$. But then $(\rho_l(\gamma)^k\eta_0, \xi^+_r(\gamma))$ lies in $\Lambda$ for $k>0$.
If $\eta_0\neq\xi^-_l(\gamma)$ we can pass to the limit on $k$ and deduce that $(\xi^+_l(\gamma), \xi^+_r(\gamma))$ lies in $\Lambda$.

So far, the choice of $\gamma$ was arbitrary. To conclude, assume now by contradiction that for every  $\gamma\in\pi_1(\Sigma_r)$ the point $(\xi^-_l(\gamma),\xi^+_r(\gamma))$ lies in $\Lambda$. Take $\alpha,\beta\in\pi_1(S)$ so that the axes of $\rho_l(\alpha)$ and $\rho_l(\beta)$ do not intersect.
We may assume that the cyclic order of end-points of those axes is 
\begin{equation}\label{eq cyclic order 1}
\xi_l^+(\alpha)<\xi_l^+(\beta)<\xi_l^-(\beta)<\xi_l^-(\alpha)~.
\end{equation} 
Since $\xi_r^\pm(\alpha)=\varphi(\xi_l^\pm(\alpha))$ and $\xi_r^\pm(\beta)=\varphi(\xi_l^\pm(\beta))$,
we also have that 
\begin{equation}\label{eq cyclic order 2}\xi_r^+(\alpha)<\xi_r^+(\beta)<\xi_r^-(\beta)<\xi_r^-(\alpha)~.\end{equation}
On the other hand by assumption (applied to $\alpha,\beta$ and their inverses) the curve $\Lambda$ contains $(\xi_l^+(\alpha), \xi_r^-(\alpha))$, $(\xi_l^+(\beta), \xi_r^-(\beta))$, $(\xi_l^-(\beta), \xi_r^+(\beta))$, $(\xi_l^-(\alpha), \xi_r^+(\alpha))$. By achronality of $\Lambda$, the cyclic order of the second components must be the same as that of the first components (although not necessarily strict), hence from \eqref{eq cyclic order 1} we obtain
$$\xi_r^-(\alpha)\leq \xi_r^-(\beta)\leq \xi_r^+(\beta)\leq \xi_r^+(\alpha)~,$$
which contradicts \eqref{eq cyclic order 2}.
\end{proof}

Given a pair $\rho=(\rho_l, \rho_r)$ of positive Fuchsian representations of $\pi_1(\Sigma_r)$, let us denote by $M_\rho$ the MGH spacetime $\Omega_\rho/\rho(\pi_1(\Sigma_r))$ of Proposition \ref{pr:xxx}.

\begin{cor}\label{cor:yyy}
For any  pair $\rho=(\rho_l, \rho_r)$ of positive Fuchsian representations of $\pi_1(\Sigma_r)$, $M_\rho$ is the unique MGH spacetime with holonomy $\rho$.
\end{cor}

The last step for the classification result is that the left and right holonomies are necessarily positive Fuchsian.

\begin{prop}\label{pr:maximal}
Let $M$ be an oriented, time-oriented, globally hyperbolic spacetime of genus $r\geq 2$ and
let us endow a Cauchy surface $\Sigma$ with the orientation induced by  the future normal vector.
Then the left and right components of the holonomy $\rho=(\rho_l,\rho_r):\pi_1(\Sigma)\to\PSL(2,\R)\times\PSL(2,\R)$  are positive Fuchsian representations.
\end{prop}
In the statement, we refer to the holonomy $\rho$ with respect to an orientation-preserving developing map. Therefore $\rho$ is well-defined up to conjugacy in $\PSL(2,\R)\times\PSL(2,\R)$. 

\begin{proof}
We will prove that the $\RP^1$-flat bundles with holonomy $\rho_l$ and $\rho_r$ are isomorphic to the unit tangent bundle of $\Sigma$. For the sake of definiteness, let us focus on $\rho_l$.
We will construct an  isomorphism
\[
   \Phi_l:T^1\widetilde \Sigma\to \widetilde \Sigma\times\RP^1
\]
equivariant with respect to the action on $T^1\widetilde \Sigma$ by the actions by deck transformation, and the diagonal action given by $\rho_l$ on $\widetilde \Sigma\times\RP^1$.
 
The map $\Phi_l$ is defined in the following way. For an element $(x,v)\in T^1\widetilde \Sigma$, let 
$$\xi(x, v)=(\xi^l(x,v), \xi^r(x,v))\in\RP^1\times\RP^1$$ be the end-point of the spacelike geodesic ray $\exp_{x}(tv)$ in $\AdSP{2,1}$, for positive $t$.
Then we define $\Phi_l(x,v)=(x, \xi^l(x,v))$. 
This map is clearly continuous, proper, equivariant and fiber preserving. 

In order to prove that it is bijective it is sufficient to notice that for any $x\in\widetilde \Sigma$ the map $\xi_x:T^1_x(\widetilde\Sigma)\to\RP^1\times\RP^1$ is an embedding with image the boundary of the totally geodesic plane tangent to $\widetilde\Sigma$ at $x$.
This boundary is the graph of an orientation-preserving map of $\RP^1$, so the projection $v\to\xi^l(x,v)$ is bijective. Moreover, by our choice of the orientation on $\Sigma$, the orientation on $T_x^1\widetilde \Sigma$ corresponds to the orientation induced on
$\xi_x(T_x^1\widetilde \Sigma)$ as graph of an orientation-preserving homeomorphism. The proof for $\rho_r$ is completely analogous.
\end{proof}

We conclude by stating the classification result. Let us denote the \emph{deformation space} of MGH spacetimes of genus $r$ by:
$$\MADS(\Sigma_r)=\{g\text{ MGH AdS metric on }\Sigma_r\times\R\}/\mathrm{Diff}_0(\Sigma_r\times\R)~,$$
where the group of diffeomorphisms isotopic to the identity acts by pull-back. Again the holonomy map takes value in the space of representations of $\pi_1(\Sigma_r)$ into $\PSL(2,\R)\times\PSL(2,\R)$ and is well-defined on the quotient $\MADS(\Sigma_r)$.

By Proposition \ref{pr:maximal}, the left and right components of the holonomy of elements of $\MADS(\Sigma_r)$ are positive Fuchsian representations. The space of these representations up to conjugacy is identified with the Teichm\"uller space of $\Sigma_r$ by the aforementioned work of Goldman \cite{MR2630832}:
$$\Teich(\Sigma_r)\cong\{\rho:\pi_1(\Sigma_r)\to\PSL(2,\R)\text{ positive Fuchsian representations}\}/\PSL(2,\R)~.$$
Therefore the holonomy map can be considered as a map 
from $\MADS(\Sigma_r)$ with values in $\Teich(\Sigma_r)\times\Teich(\Sigma_r)$.
We can summarize Proposition \ref{pr:xxx} and Corollary \ref{cor:yyy} with the following theorem of Mess.

\begin{theorem} \label{thm:classification rgeq2}
The holonomy map $$\hol:\MADS(\Sigma_r)\to\Teich(\Sigma_r)\times\Teich(\Sigma_r)$$ is a homeomorphism.
\end{theorem}


\section{Gauss map of spacelike surfaces} \label{ch:gauss}

In this section we will introduce the \emph{Gauss map} associated to a spacelike surface in Anti-de Sitter space, study its properties, and deduce some results which will further highlight the deep relation of Anti-de Sitter geometry with Teichm\"uller theory and hyperbolic surfaces.

\subsection{Spacelike surfaces and immersion data}\label{sec:spacelike immersions}

Let us start by recalling some generalities of (immersed) spacelike surfaces in Anti-de Sitter geometry. For the moment, we shall assume that all our immersed surfaces are of class $C^1$.

Let us therefore assume that $\sigma:\Sigma\to\AdSP{2,1}$ is a $C^1$ immersion, and recall that $\sigma$ is \emph{spacelike} if the pull-back $\sigma^*g_{\AdSP{}}$ of the ambient Lorentzian metric $g_{\AdSP{}}$ is positive definite for every point of $\Sigma$. In this case the Riemannian metric $I:=\sigma^*g_{\AdSP{}}$ is called \emph{first fundamental form} of $\sigma$.

The tangent bundle $TS$ is then naturally identified to a subbundle of the pull-back bundle $\sigma^*(TM)$ by means of the differential $d\sigma$.
The normal bundle $N_\sigma$ of $\sigma$ is defined as the $g_{\AdSP{}}$-orthogonal complement of $TS$ in $\sigma^*(TM)$, and the restriction of $g_{\AdSP{}}$ to $N_\sigma$ is negative definite. 
Using the decomposition
\[
         \sigma^*(TM)=TS\oplus N_\sigma\,,
\]
the pull-back of the ambient Levi-Civita connection $\nabla$, restricted to sections tangent to $S$ splits as the sum of the Levi Civita connection $\nabla^I$ of the first fundamental form $I$, and a symmetric $2$-form with value in $N_\sigma$.
As a consequence of time-orientability of $\AdSP{2,1}$, the normal bundle admits a natural trivialization, which is the same as a choice of a continuous unit normal vector field for $\sigma$. We will denote by $\nu:S\to N_\sigma$ the future-directed unit normal vector field, and consider the decomposition for all vector field $X,Y$ tangent to $S$:
\[
     \nabla_VW=\nabla^{\I}_VW+\II(V,W)\nu\,,
\]
where the symmetric $(2,0)$-tensor $\II$ is called \emph{second fundamental form}.
It will be convenient to consider the $\I$-symmetric $(1,1)$-tensor $B\in (TS)^*\otimes TS$ defined by $\II(V,W)=\I(B(V), W)$, which is called \emph{shape operator} of $\sigma$. 
Similarly to the Riemannian  case, it turns out that $\sigma_*(B(v))=\nabla_v\nu$.

The first and the second fundamental form of an immersion $\sigma$ satisfy constraint equations, known as the \emph{Gauss-Codazzi equations}.
More precisely the Gauss equation consists of the identity
\begin{equation}\label{eq:gauss}
    K_{\I}=-1-\det_{\I}\II
\end{equation}
where $K_{\I}$ is the curvature of $\I$ and $\det_{\I}\II:=\det B$ by definition.
On the other hand the Codazzi equation states that $\nabla^{\I}\II$ is a totally symmetric $3$-form. In other words we have
\begin{equation}\label{eq:codazzi}
   (\nabla^{\I}_V\II)(W,U)=(\nabla^{\I}_W\II)(V,U)
\end{equation}
which sometimes is also written in the equivalent form
   $d^{\nabla^{I}}\!\!B=0$
where 
\begin{equation}\label{eq:ext derivative}d^{\nabla^{I}}\!\!B(V,W)=(\nabla^{\I}_V B)(W)-(\nabla^{\I}_WB)(V)=\nabla^{\I}_V(B(W))-\nabla^{\I}_W(B(V))-B([V,W])~.
\end{equation}

The following classical result states that the Gauss-Codazzi equations are the only constraints to be satisfied by the first and second fundamental forms.

\begin{theorem}[Fundamental theorem of immersed surfaces]\label{thm:fund theorem}
Let $S$ be a simply connected surface, let $\I$ be a Riemannian metric on $S$ and $\II$ be a symmetric $(2,0)$-tensor on $S$. If $I$ and $\II$ satisfy the Gauss-Codazzi equations \eqref{eq:gauss} and \eqref{eq:codazzi}, then there exists a spacelike immersion $\sigma:S\to\AdSP{2,1}$ having $\I$ and $\II$ as first and second fundamental form. Moreover if $\sigma$ and $\sigma'$ are two such immersions, then there exists a time-preserving isometry $f$ such that $\sigma'=f\circ \sigma$.
\end{theorem}

\subsection{Germs of spacelike immersions in AdS manifolds}

Let us now consider the case of an oriented surface $\Sigma$, not necessarily simply connected. Given a spacelike immersion $\sigma:\Sigma\to (M,g)$, where $(M,g)$ is an oriented Anti-de Sitter manifold, we can associate to $\sigma$ the pair $(\I,\II)$ of first and second fundamental form exactly as in the previous section, where $\II$ is computed with respect to the future unit normal vector $\nu$ of $\sigma$. Moreover, in this section we shall always assume that the orientation of $\Sigma$ and $\nu$ are compatible with the orientation of $M$.

The pair $(\I,\II)$ only depends on the \emph{germ} of $\sigma$, which we introduce in the following definition:

\begin{defi}
A \emph{germ} of a spacelike immersion of $\Sigma$ into an Anti-de Sitter three-manifold is an equivalence class of spacelike immersions $\sigma:\Sigma\to (M,g)$, where the time-oriented Lorentzian manifold $(M,g)$ has constant sectional curvature -1, by the following relation: $\sigma:\Sigma\to(M,g)$ and $\sigma':\Sigma\to(M',g')$ are equivalent if there exist open subsets $U$ in $M$ and $U'$ in $M'$ and an orientation-preserving, time-preserving isometry $f:(U,g)\to(U',g')$ such that $\sigma'=f\circ \sigma$. 
\end{defi}

Observe that in the definition, $\sigma'=f\circ \sigma$ implies that $U$ is an open neighbourhood of the image of $\sigma$, and similarly for $U'$.
It is a simple exercise to check that the above definition gives an equivalence relation. 

Now, given a pair $(\I,\II)$ on a surface $\Sigma$, one can perform the following construction. 
If $\pi:\widetilde\Sigma\to\Sigma$ is a universal cover, the pair $(\pi^*\I,\pi^*\II)$ clearly satisfy the Gauss-Codazzi equations on $\widetilde \Sigma$, hence by the existence part of Theorem \ref{thm:fund theorem} there exists a spacelike immersion $\widetilde\sigma:\widetilde\Sigma\to\AdSP{2,1}$ having immersion data $(\pi^*\I,\pi^*\II)$. The uniqueness part of 
Theorem \ref{thm:fund theorem} then has two consequences:
\begin{itemize}[leftmargin=0.5cm]
\item Any two such immersions differ by post-composition by a global isometry of $\AdSP{2,1}$.
\item Given any such $\widetilde \sigma$, there exists a map $\rho:\pi_1(\Sigma)\to\isom_0(\AdSP{2,1})$ such that, for every $\gamma\in \pi_1(\Sigma)$, $f\circ \gamma=\rho(\gamma)\circ f$.
\end{itemize}
It is easily checked that $\rho$ is in fact a group representation. Moreover changing $\widetilde \sigma$ by post-composition with an isometry $f$ has the effect of conjugating $\rho$ by $f$. The immersion $\sigma$ can then be extended to an immersion of $U$, an open neighbourhood of $\Sigma\times\{0\}$ in $\Sigma\times\R$, into $\AdSP{2,1}$, by mapping $(x,t)$ to the point $\gamma(t)$ on the timelike geodesic $\gamma$ such that $\gamma(0)=\sigma(p)$ and $\gamma'(0)$ is the future normal vector of $\sigma$ at $x$. We collect here the expression of the Anti-de Sitter metric in such a tubular neighborhood of $\sigma$, which is in fact a local computation and will be useful for future reference:

\begin{lemma} \label{lemma metric tub neigh}
Given a spacelike immersion $\sigma:\Sigma\to\AdSP{2,1}$, the pull-back of the ambient metric by means of the map $(p,t)\mapsto \exp_{\sigma(x)}(t\nu(x))$ has the expression:
\begin{equation}\label{eq:metric tub neigh}
-dt^2+\cos^2(t)\I+2\cos (t)\sin (t)\II+\sin^2 (t)\III~,
\end{equation}
where $\I$, $\II$ and $\III$ are the first, second and third fundamental form of $\sigma$.
\end{lemma}

Recall that the third fundamental form is defined as $\III(\cdot,\cdot)=\I(B(\cdot),B(\cdot))$ where $B$ is the shape operator. Conversely observe that, by a simple computation, the immersion data of $x\mapsto(x,0)$ in \eqref{eq:metric tub neigh} are $(\I,\II)$.
\begin{proof}
We may use the quadric model, introduced in Section \ref{subsec:quadric}. By equation \eqref{eq:time geo quadric}, we have $\exp_{\sigma(x)}(t\nu(x))=\cos(t)\sigma(x)+\sin(t)\nu(x)$. The differential in $t$ gives the vector $-\sin(t)\sigma(x)+\cos(t)\nu(x)$, while the differential in  the spatial direction $(V,0)$ gives the vector  $\cos(t)d\sigma_x(v)+\sin(t)d_x\nu(v)$. The two vectors are orthogonal. Recalling that $I(\cdot,\cdot)=\langle d\sigma(\cdot),d\sigma(\cdot)\rangle$ and that the differential of $\sigma$ identifies $B(v)$ and $\nabla_v\nu$, namely the tangential component of $d\nu(v)$, the expression of the pull-back metric follows immediately. 
\end{proof}

Therefore, given a pair $(\I,\II)$, the expression \eqref{eq:metric tub neigh} provides a Lorentzian metric of constant curvature -1 on an open set $U$ in $\Sigma\times\R$ containing the slice $\Sigma\times\{0\}$, and thus a germ of immersion of $\Sigma$ into an Anti-de Sitter three-manifold with immersion data $(\I,\II)$.
The conclusion of the above discussion is summarized in the following statement:
\begin{prop} \label{prop summarize items}
Given a surface $\Sigma$, there are natural identifications between the following spaces:
\begin{enumerate}
\item \label{eq:item 2}The space of pairs $(\I,\II)$ on $\Sigma$ which are solutions of the Gauss-Codazzi equations. 
\item  \label{eq:item 1}The space of germs of spacelike immersions of $\Sigma$ into Anti-de Sitter manifolds.
\item \label{eq:item 3} The space of spacelike immersions of $\widetilde\Sigma$ into $\AdSP{2,1}$, equivariant with respect to a representation $\rho:\pi_1\Sigma\to \isom_0(\AdSP{2,1})$, up to the action of $\isom_0(\AdSP{2,1})$ by post-composition.
\end{enumerate}
The identifications are equivariant with respect to the actions of $\mathrm{Diff}(\Sigma)$, by pull-back in item \eqref{eq:item 2} and by pre-composition in items  \eqref{eq:item 1} and \eqref{eq:item 3}.
\end{prop}

Let us now consider the case where $\Sigma$ is a closed surface. By the arguments of the previous section, the equivariant immersion $\widetilde\sigma$ in item \eqref{eq:item 3} of Proposition \ref{prop summarize items} is necessarily an embedding, which can be extended to an embedding of $\widetilde\Sigma\times\R$ onto a domain of dependence in $\AdSP{2,1}$. The representation $\rho:\pi_1(\Sigma)\to\PSL(2,\R)\times\PSL(2,\R)$ coincides with the holonomy of a maximal globally hyperbolic Anti-de Sitter manifold $(M,g)$ (after identifying $\pi_1(\Sigma)$ with $\pi_1(M)$ using the embedding of $\Sigma$ into $M\cong \Sigma\times\R$), and therefore $\rho$ consists of a pair of positive Fuchsian representations by Proposition \ref{pr:maximal}.

Quite remarkably, the embedding data $(\I,\II)$ permit to recover explicitly the pair of elements in the space $\Teich(S)\times\Teich(S)$ which parameterizes maximal globally hyperbolic Anti-de Sitter manifolds with compact Cauchy surfaces --- recall Theorem \ref{thm:classification rgeq2}. Such an explicit formula is the content of Proposition \ref{prop:formulae projections} in the next section.

\subsection{Gauss map and projections} \label{sec:gauss and projections}

We are now ready to define the Gauss map for spacelike surfaces in $\AdSP{2,1}$, see \cite{MR3888623}. Recall from Proposition \ref{prop space of time geo} that the space of timelike geodesics of $\AdSP{2,1}$ is naturally identified with $\Hyp^2\times\Hyp^2$, where the identification maps a geodesic of the form $$L_{p,q}=\{X\in\PSL(2,\R)\,|\,X\cdot q=p\}$$
to the pair $(p,q)\in\Hyp^2\times\Hyp^2$. We still suppose that our spacelike immersions are $C^1$ here, and will discuss certain cases of weaker regularity in the next section.

\begin{defi}
Let $\sigma:S\to\AdSP{2,1}$ a spacelike immersion. The \emph{Gauss map} $G_\sigma:S\to\Hyp^2\times\Hyp^2$ is defined as $G_\sigma(x)=(p,q)$ such that $L_{p,q}$ is the timelike geodesic orthogonal to $\im(d_x\sigma)$ at $\sigma(x)$. 
\end{defi}

As a consequence of the equivariance property given in Proposition \ref{prop space of time geo}, the Gauss map $G_\sigma$ is natural with respect to the action of the isometry group, meaning that 
$$G_{f\circ\sigma}=f\cdot G_\sigma$$
for every $f\in \Isom_0(\AdSP{2,1})=\PSL(2,\R)\times \PSL(2,\R)$.

\begin{example} \label{ex:gauss map dual plane}
Recall that in Lemma \ref{rmk dual plane traceless2} we gave an isometric embedding of $\Hyp^2$
in $\AdSP{2,1}$ with image the plane $P_\En$ dual to the identity. This isometric embedding is defined by sending $p\in\Hyp^2$ to the unique order-two element in $\PSL(2,\R)$ fixing $p$, which by definition lies on the geodesic $L_{p,p}$. Moreover the geodesic $L_{p,p}$ is orthogonal to $P_\En$. Hence the Gauss map associated to this isometric embedding of $\Hyp^2$ is simply $p\mapsto(p,p)$.
\end{example}

By construction, the Gauss map of a spacelike immersion $\sigma$ is invariant by reparametrization, in the sense that $G_{\sigma\circ \phi}=G_\sigma\circ\phi$ for a diffeomorphism $\phi:S'\to S$. Hence it makes sense to talk about the Gauss map of a spacelike surface in $\AdSP{2,1}$. For example, for the plane $P_\En$ dual to the identity as in Example \ref{ex:gauss map dual plane},  the Gauss map of $P_\En$ is sends order-two element of $\PSL(2,\R)$ to the pair $(p,p)$ where $p$ is its fixed point. 

\begin{lemma}\label{lemma: G vs Fix}
Given a spacelike immersion $\sigma:S\to\AdSP{2,1}$ with future unit normal vector field $\nu$, if $\sigma(p)=\En$, then 
\begin{equation}\label{eq:G vs Fix}
G_\sigma(p)=G_{P_\En}\left(\exp\left(\frac{\pi}{2}\nu(p)\right)\right)~.
\end{equation}
\end{lemma}
\begin{proof}
The proof follows from Example \ref{ex:gauss map dual plane} and the observation that the geodesic leaving from $\En$ with velocity $\nu(p)$ meets orthogonally $P_\En$ at $\exp((\pi/2)\nu(p))$.
\end{proof}

Let us now introduce the map 
$$\mathrm{Fix}:T^{1,+}_\En\AdSP{2,1}\to \Hyp^2$$
where $T^{1,+}_\En\AdSP{2,1}$ denotes the hyperboloid of future unit timelike vectors in $T_{\En}\AdSP{2,1}$, such that $\mathrm{Fix}(\nu)$ is the fixed point of the one-parameter elliptic group $\{\exp(t\nu)\,|\,t\in\R\}$. This map is equivariant for the action of $\PSL(2,\R)$, which acts on the hyperboloid $T^{1,+}_\En\AdSP{2,1}$ by the adjoint representation and on $\Hyp^2$ by the obvious action. Since both $T^{1,+}_\En\AdSP{2,1}$ and $\Hyp^2$ have constant curvature $-1$, it follows immediately from the equivariance that $\mathrm{Fix}$ is an isometry.

In terms of the map $\mathrm{Fix}$, Equation \eqref{eq:G vs Fix} reads
\begin{equation}\label{eq:G vs Fix2}
G_\sigma(p)=\left(\mathrm{Fix}(\nu(p)),\mathrm{Fix}(\nu(p))\right)~,
\end{equation}
provided $\sigma(p)=\En$.
Using Lemma \ref{lemma: G vs Fix} and the naturality, we get the following description of the Gauss map.
\begin{lemma}\label{lm:GG}
Given a spacelike immersion $\sigma:S\to\AdSP{2,1}$ with future unit normal vector field $\nu$, 
$$G_\sigma(p)=\left(\mathrm{Fix}((R_{\sigma(p)^{-1}})_*(\nu(p))),\mathrm{Fix}((L_{\sigma(p)^{-1}})_*(\nu(p)))\right)~.$$
\end{lemma}
\begin{proof}
Let us first observe that, if $\sigma(p)=\En$, then the equality holds true by  Equation \eqref{eq:G vs Fix2}. In the general case, the immersion $\sigma'=(\En,\sigma(p))\circ\sigma$ has the property that $\sigma'(p)=\En$, and the future normal vector at $\sigma'(p)$ equals $\nu'(p)=(R_{\sigma(p)^{-1}})_*(\nu(p))$. 
Therefore 
$$G_{\sigma'}(p)=\left(\mathrm{Fix}((R_{\sigma(p)^{-1}})_*(\nu(p))),\mathrm{Fix}((R_{\sigma(p)^{-1}})_*(\nu(p)))\right)~.$$
By the naturality of the Gauss map, 
\begin{align*}G_{\sigma}(p)&=(\En,\sigma(p)^{-1})\cdot G_{\sigma'}(p) \\
&=\left(\mathrm{Fix}((R_{\sigma(p)^{-1}})_*(\nu(p))),\sigma(p)^{-1}\circ\mathrm{Fix}((R_{\sigma(p)^{-1}})_*(\nu(p)))\right) \\
&=\left(\mathrm{Fix}((R_{\sigma(p)^{-1}})_*(\nu(p))),\mathrm{Fix}((L_{\sigma(p)^{-1}})_*(\nu(p)))\right)~,
\end{align*}
where in the last line we used the fact that $\mathrm{Fix}$ is equivariant with respect to the adjoint action on the hyperboloid $T^{1,+}_\En\AdSP{2,1}$.
\end{proof}

The components of the Gauss map are called \emph{left} and \emph{right projections}, and will be denoted by $\Pi_l,\Pi_r:S\to \Hyp^2$. 

\begin{remark} \label{rmk gauss map and parallel tran}
Under the identification given by $\mathrm{Fix}$, the left and right projections can be interpreted in the following way. Given $p\in S$, $\Pi_l(p)$ is the parallel transport in $\En$ of the future unit vector $\nu(p)$ at $\sigma(p)$ with respect to the right-invariant connection $D^r$ we introduced in Section \ref{subsec levi civita}. The right projection   $\Pi_r(p)$ is instead obtained by parallel transport with respect to the left-invariant connection.
\end{remark}

\begin{remark} \label{rmk gauss map alla mess}
Another interpretation of the Gauss map, which was originally given in the work of Mess, is the following. Given $p\in S$ one can find a unique left isometry $f_l(p)$, and a unique right isometry $f_r(p)$, sending the tangent plane $P$ to the image of $\sigma$ at  $\sigma(p)$ to $P_\En$. Indeed the isometries $f_l(p)$ and $f_r(p)$ are simply obtained by left and right multiplication by the inverse of dual point of the tangent plane $P$, namely $\varsigma(p)=\exp_{\sigma(p)}((\pi/2)\nu(p))$. Using the identification of the dual plane $P_\En$ with $\Hyp^2$ provided by Lemma \ref{rmk dual plane traceless2}, $\Pi_l(p)$ and $\Pi_r(p)$ are the image of $\sigma(p)$ under the right and left isometries respectively:
$$\Pi_l(p)=f_r(p)\circ \sigma(p)=(\En,\varsigma(p))\cdot\sigma(p)\qquad\text{and}\qquad \Pi_r(p)=f_l(p)\circ \sigma(p)=(\varsigma(p)^{-1},\En)\cdot\sigma(p)~.$$
\end{remark}

We are now ready to prove the formulae which express the pull-back of the hyperbolic metrics by the left and right projections. When applying these formulae to the embedding data of a surface in an MGH Cauchy compact Anti-de Sitter spacetime $(M,g)$, we obtain a pair of hyperbolic metrics whose isotopy classes give are the parameters of $(M,g)$ in $\Teich(S)\times\Teich(S)$. (See also Proposition \ref{prop summarize items} and the following paragraph.)

\begin{prop}\label{prop:formulae projections}
Let $\sigma:S\to\AdSP{2,1}$ be a spacelike immersion, let $\Pi_l,\Pi_r:S\to\Hyp^2$ be the left and right projections and let $g_{\Hyp^2}$ be the hyperbolic metric. Then
$$\Pi_{l}^*g_{\Hyp^2}=I((\id-\JJ B)\cdot, (\id-\JJ B)\cdot)\qquad\text{and}\qquad
\Pi_{r}^*g_{\Hyp^2}=I((\id+\JJ B)\cdot, (\id+\JJ B)\cdot)~,$$
where $\I$ is the first fundamental form of $\sigma$, $\JJ$ its associated almost-complex structure, and $B$ the shape operator.
\end{prop}

These formulae appeared in \cite[Lemma 3.16]{KS07}, and are proved also in \cite[Section 6.2]{MR3888623}. Here we provide a self-contained proof.

\begin{proof}
Let us check the formula for the pull-back of $\Pi_r$. By Lemma \ref{lm:GG},
 $$\Pi_r(x)=\mathrm{Fix}((L_{\sigma(x)^{-1}})_*(\nu(x)))~.$$ Since $\mathrm{Fix}:T^{1,+}_\En\AdSP{2,1}\to\Hyp^2$ is an isometry, $\Pi_{r}^*g_{\Hyp^2}$ equals the pull-back of the Anti-de Sitter metric through the map
$\widehat\Pi_r:S\to T_\En\AdSP{2,1}$ defined by $\widehat\Pi_r(x)=(L_{\sigma(x)^{-1}})_*(\nu(x))$.

Let us fix a orthonormal basis of left-invariant vector fields $E_1,\ldots, E_n$ on $T\AdSP{2,1}$.
Then we can express the unit normal vector as $\nu(x)=\sum_{i}\nu_i(x)E_i(\sigma(x))$, for some functions $\nu_i: S\to\R$.
By Remark \ref{rmk gauss map and parallel tran}, 
$$
\widehat\Pi_r(x)=\sum_{i}\nu_i(x)E_i(\En)~.
$$
By differentiating we obtain 
\begin{equation}\label{eq:pippo1}
d\widehat\Pi_r(v)=\sum_{i}d\nu_i(v) E_i(\En)~.
\end{equation}
On the other hand, since left-invariant vector fields are parallel for the left-invariant connection $D^l$, we have 
\begin{equation}\label{eq:pippo2}
D^l_v\nu=\sum_{i}d\nu_i(v) E_i(\sigma(x))~.
\end{equation}
The identities \eqref{eq:pippo1} and \eqref{eq:pippo2} together show that $\Pi_r^* g_{\Hyp^2}(v,w)$, which equals the Anti-de Sitter metric $g_{\AdSP{2,1}}$ at the identity evaluated on the tangent vectors $d\widehat\Pi_r(v)$ and $d\widehat\Pi_r(w)$, equals $g_{\AdSP{2,1}}(D^l_v\nu, D^l_w\nu)$.

Using Equation \eqref{eq:levi-civita general} and Lemma \ref{lemma cross product lie bracket},
$$D^l_v\nu=\nabla_v\nu+ v\boxtimes\nu=B(v)-\nu\boxtimes v=(B-\JJ)v~.$$
We conclude that
\[
\Pi_r^* g_{\Hyp^2}(v,w)=I((B-\JJ)v,(B-\JJ)w)=I((\id+\JJ B)v, (\id+\JJ B)w)
\]
as claimed. The proof for the left projection is exactly the same, using right-invariant vector fields and the right-invariant connection, and one gets a difference in sign when applying Equation \eqref{eq:levi-civita general}.
\end{proof}

\subsection{Consequences and comments}\label{sec:consequences formula projections}
We collect here several consequences and remarks around Proposition \ref{prop:formulae projections}.
\begin{itemize}[leftmargin=0.5cm]
\item A first simple remark is that if $\sigma$ is a totally geodesic immersion, which means that $B$ vanishes identically, then the projections are local isometries. Even without using Proposition \ref{prop:formulae projections}, we have already observed this fact in Example \ref{ex:gauss map dual plane} for the totally geodesic plane $P_\En$, and it is therefore true by the naturality property for every totally geodesic immersion.
\item Proposition \ref{prop:formulae projections} shows that the differential of the left and right projections essentially has the expression 
$$d\sigma_x\circ(B\pm\JJ)~,$$
up to post-composing with an isometry sending the image of $d\sigma_x$ to a fixed copy of $\Hyp^2$. Since $B$ is $\I$-symmetric, $\JJ\circ B$ is traceless, and therefore 
\begin{equation}\label{eq:deteminant differential projections}
\det(B\pm\JJ)=1+\det B=-K_\I~.
\end{equation}
This shows that $\Pi_l$ is a local diffeomorphism at a point $x$ if and only if $\Pi_r$ is, which is the case if and only if the intrinsic curvature of $\I$ at $x$ is different from $0$.
\item Since the trace of $B\pm\JJ$ equals $2$, the differentials of $\Pi_l$ and $\Pi_r$ have either rank 2 or rank 1. (In fact, by \eqref{eq:deteminant differential projections}, when the differential of $\Pi_l$ has rank 1, the same holds for the differential of $\Pi_r$.) Hence the differential of the Gauss map $G:S\to\Hyp^2\times\Hyp^2$ is always non-singular. Moreover, we have the following dichotomy: for every point $x$, either the image of $G$ is locally a graph of a map between (open subsets of) $\Hyp^2$, or it is tangent at $G(x)$ to a maximal flat of $\Hyp^2\times\Hyp^2$, that is, to the product of two geodesics.  
\item If an immersed surface has the property that the curvature of the first fundamental form never vanishes, and if moreover $\Pi_l$ and $\Pi_r$ are globally injective, then the image of $G$ is the graph of a  diffeomorphism $F_\sigma$ between two subsets of $\Hyp^2$,  called the \emph{associated map}. From Equation \eqref{eq:deteminant differential projections}, the Jacobians of $\Pi_l$ and $\Pi_r$ are equal, hence the associated map is area-preserving.
When $\Pi_l$ and $\Pi_r$ are only locally injective, but not globally, we still obtain an area-preserving local diffeomorphism $F_\sigma$ which is now defined between two hyperbolic surfaces, not globally isometric to subsets of $\Hyp^2$.
\item More generally, as a consequence of the previous points, the image of $G$ is always a \emph{Lagrangian submanifold} in $\Hyp^2\times\Hyp^2$ with respect to the symplectic form 
\begin{equation} \label{eq symplectic form}
\Omega=\pi_l^*\omega_{\Hyp^2}-\pi_r^*\omega_{\Hyp^2}~,
\end{equation}
where $\omega_{\Hyp^2}$ is the hyperbolic area form. This result has been proved in several works with different methods: see \cite{MR3888623,MR3937298,MR3745436}. Moreover the Lagrangian condition is \emph{locally} the only obstruction to inverting this construction, that is, to realizing an immersed surfaces in $\Hyp^2\times\Hyp^2$ locally as the image of the Gauss map of a spacelike immersion in $\AdSP{2,1}$.
\item Finally, given a spacelike immersion $\sigma$,  the normal evolution of $\sigma$ is defined as 
$$\sigma_t(x)=\exp_{\sigma(x)}(t\nu(x))~,$$
where $\nu$ is the future unit normal vector field. In general $\sigma_t$ may fail to be an immersion for $|t|$ large. (We will come back to this point in Section \ref{subsec:foliations}, in particular Remark \ref{remark normal flow singular}). When it is an immersion, the computation of the metric in Lemma \ref{lemma metric tub neigh} shows that the image of $\sigma_t$ at $x$ is orthogonal to the geodesic $\gamma(t)=\exp_{\sigma(x)}(t\nu(x))$. In other words, the Gauss map of $\sigma_t$ is equal to the Gauss map of $\sigma$. Hence with respect to the previous point, given a spacelike immersion, there is actually a one-parameter family of immersions, which differ from one another by the normal evolution, which have the same  Lagrangian submanifold of $\Hyp^2\times\Hyp^2$ as Gauss map image.
This phenomenon can be explained in a more transparent way in terms of the unit tangent bundle, see Section \ref{subsec unit tangent} below.
\end{itemize}

\begin{remark} \label{remark codazzi and curvature}
Given a Riemannian metric $\I$, suppose $A$ a $(1,1)$-tensor which is $\I$-symmetric and $\I$-Codazzi, namely satisfying 
$d^{\nabla^I}\!\!\!A=0$ (recall Equation \eqref{eq:ext derivative} for the definition of the exterior derivative). Then the curvature of the metric $g=\I(A\cdot,A\cdot)$ is expressed by the formula
$$K_g=\frac{K_\I}{\det(A)}~,$$
see \cite{labourieCP1}. The tensors $\id\pm\JJ\circ B$ appearing in Proposition \ref{prop:formulae projections} are $\I$-Codazzi, since $\id$, $\JJ$ and $B$ are all $\I$-Codazzi. Then using the Gauss equation and Equation \eqref{eq:deteminant differential projections}, one verifies directly that, when non-degenerate, the pull-back metrics of Proposition \ref{prop:formulae projections} are hyperbolic.
\end{remark}

\subsection{Future unit tangent bundle perspective} \label{subsec unit tangent}
The Gauss map of an embedded surface can be described concisely in terms of the \emph{future timelike unit tangent bundle} $T^{1,+}\AdSP{2,1}$, namely the bundle whose fiber over $x\in\AdSP{2,1}$ is the subset of $T_x \AdSP{2,1}$ consisting of future-directed timelike vectors of square norm $-1$ (which is therefore a copy of $\Hyp^2$). The total space of $T^{1,+}\AdSP{2,1}$ has also a structure of principal $\mathbb S^1$-bundle over $\Hyp^2\times\Hyp^2$: the projection $\pi:T^{1,+}\AdSP{2,1}\to \Hyp^2\times\Hyp^2$ maps $(x,v)$ to the equivalence class (up to reparametrization) of the timelike geodesic  $\gamma$ with $\gamma(0)=x$ and $\gamma'(0)=v$, and the $\mathbb S^1$-action is by the geodesic flow, namely the action of $e^{it}$ maps  $(x,v)$ to $(\gamma(2t),\gamma'(2t))$ (recalling that timelike geodesics have length $\pi$).

One can then define the Gauss map for any spacelike immersed surface $\sigma:S\to\AdSP{2,1}$: this is simply obtained by first lifting $\sigma$ to a map $\widetilde\sigma:S\to T^{1,+}\AdSP{2,1}$, where $\widetilde\sigma=(\sigma,\nu)$, and then projecting $\widetilde\sigma$ to $\Hyp^2\times\Hyp^2$ by composition with $\pi$. Hence one clearly recovers the fact that the Gauss map is invariant by the normal evolution, since normal evolution corresponds to the action of the geodesic flow on $T^{1,+}\AdSP{2,1}$. 

We will not pursue this point of view very far here, and we refer to \cite{MR3937298} for the interested reader and to \cite{kobanomizu} for the necessary background. However, it is worth mentioning that there is a natural fiber bundle connection on the principal bundle $\pi:T^{1,+}\AdSP{2,1}\to \Hyp^2\times\Hyp^2$, which has interesting consequences. To define the  connection, recall that the 
Levi-Civita connection of $\AdSP{2,1}$ induces a decomposition of $T_{(x,v)}T\AdSP{2,1}=H\oplus V$, where $V$ is the tangent space to the fiber, hence naturally isomorphic to $T_x\AdSP{2,1}$, while $H$ consists of vectors tangent to the lifts to $T\AdSP{2,1}$ of geodesics of $\AdSP{2,1}$, and therefore the differential of the projection $T\AdSP{2,1}\to\AdSP{2,1}$ identifies $H$ with $T_x\AdSP{2,1}$. Then 
 the pseudo-Riemannian Sasaki metric $g_S$ is defined on $T\AdSP{2,1}$ by declaring that $H$ and $V$ are orthogonal, and that $g_S$ restricted to $H$ and $V$ coincides with the metric of $\AdSP{2,1}$ under the above isomorphisms. 
 It turns out that the Sasaki metric is invariant both under the action of the isometry group $\isom(\AdSP{2,1})$ and by the geodesic flow.
 Then the connection, which in this case is simply a real-valued 1-form on $T^{1,+}\AdSP{2,1}$, is defined as
 $$\omega(\cdot)=g_S(\chi,\cdot)~,$$
 where $\chi$ is the infinitesimal generator of the geodesic flow: at a point $(x,v)$, the component of $\chi$ along $V$ vanishes, while its component along $H$ is $v$. One can then prove (see \cite[Proposition 3.9]{MR3937298}) that the curvature of the  connection $\omega$, which is defined as $d\omega$ (in general there is an additional term $\omega\wedge\omega$, which vanished automatically here), coincides up to a factor with the pull-back $\pi^*\Omega$, where $\Omega$ is the symplectic form of $\Hyp^2\times\Hyp^2$ defined in \eqref{eq symplectic form}. This permits to recover once more the fact that the Lagrangian condition is the only local obstruction to realize a submanifold of $\Hyp^2\times\Hyp^2$ as the Gauss map image of a spacelike surface in $\AdSP{2,1}$, a fact we mentioned already in Section \ref{sec:consequences formula projections}. In \cite{MR3937298} this technology is applied to determine a global obstruction to the reversibility of the Gauss map construction for MGH Cauchy compact manifold, namely in presence of the action of a pair of Fuchsian representations $\rho=(\rho_l,\rho_r):\pi_1\Sigma\to\PSL(2,\R)\times\PSL(2,\R)$, in terms of Hamiltonian orbits.

\subsection{Non-smooth surfaces}\label{sec:nonsmooth}
The construction of the Gauss map can be extended in the non-smooth setting, for instance for convex spacelike surfaces $S$ in $\AdSP{2,1}$, which means that every support plane of $S$ is spacelike. Then one defines the set-valued Gauss map as the map sending each point $x$ of $S$ to the set of future unit vectors in $T^{1,+}_x\AdSP{2,1}$ orthogonal to support planes of $S$ at $x$. Hence the image of a point $x\in S$ is a convex subset of $T_x\AdSP{2,1}$, and it is reduced to a single point if and only if $S$ is differentiable at $x$. The image of $G$ in $T^{1,+}_x\AdSP{2,1}$ is a  $C^{1,1}$ surface.

In this pioneering work, Mess highlighted the relation between pleated surfaces and earthquake maps. Recall that, given an achronal meridian $\Lambda\subset\Quno$, the upper and lower boundary components $\partial_\pm C(\Lambda)$ of the convex hull of $\Lambda$ are a convex and a concave pleated surface, see Proposition \ref{prop bdy convex hull achronal} and Remark \ref{rmk:lightlike triangles in boundary components}. 

In general the pleated surfaces $\partial_\pm C(\Lambda)$ may contain lightlike triangles, which happens exactly in correspondence of a sawtooth, see Remark \ref{rmk:light triangles convex invisible} and also Remark \ref{rmk:Ksurfaces sawtooth} below. 
In this case, the Gauss map is of course not defined on these lightlike triangles. 
The fundamental claim is then the following, see also \cite{MR3888623} for more details:

\begin{lemma}\label{lemma earthquake}
Given an achronal meridian $\Lambda\subset\Quno$, the image of the Gauss map of $\partial_+ C(\Lambda)$ and $\partial_- C(\Lambda)$, which are defined only on the spacelike parts, are the graph of (left and right respectively) earthquake maps between straight convex sets of $\Hyp^2$. 
\end{lemma}

More precisely, what happens is that the left and right projections from $\partial_+ C(\Lambda)$ to $\Hyp^2$ are (right and left respectively) earthquake maps with image a straight convex set, and the earthquake measured laminations coincide with the bending measured laminations. Hence the composition $\Pi_r\circ\Pi_l^{-1}$ gives a left earthquake map, which is in fact defined in the complement of the simplicial leaves of the lamination, and its earthquake measured lamination is identified to the bending measured lamination of $\partial_+ C(\Lambda)$. The same holds for $\partial_- C(\Lambda)$, by reversing the roles of left and right. 

We will not give a full proof of Lemma \ref{lemma earthquake}, but in the next section we will explain the case of a surface pleated along a single geodesic, which is the essential step. The full lemma can then be proved by an approximation argument as in \cite{MR2499272}. For more details on the part of the statement about straight convex sets, see \cite{MR2764871}. 

Now, when the curve $\Lambda$ is the graph of an orientation-preserving homeomorphism, one obtains as a result earthquake maps of $\Hyp^2$. When moreover $\varphi$ is the homeomorphism which conjugates the left and right representations $\rho_l,\rho_r:\pi_1\Sigma\to\PSL(2,\R)$ of the holonomy of a MGH Cauchy compact manifold, the naturality of the construction implies that the earthquake map descends to an earthquake map from the left to the right hyperbolic surfaces, namely $\Hyp^2/\rho_l(\pi_1\Sigma)$ and $\Hyp^2/\rho_r(\pi_1\Sigma)$. (By Lemma \ref{lemma extension} which will be discussed below, one actually sees directly that the earthquake maps of $\Hyp^2$ extends continuously to $\varphi$ on the boundary at infinity.)

Let us denote, for a measured geodesic lamination $\mu$ on $\Sigma$, the left and right earthquake maps by
$$E^l_\mu:\Teich(\Sigma)\to\Teich(\Sigma)\qquad\text{and}\qquad E^r_\mu:\Teich(\Sigma)\to\Teich(\Sigma)~,$$
seen as maps of the Teichm\"uller space of $\Sigma$ to itself. As a consequence of the previous discussion, and the example to be explained in the next section, Mess recovered Thurston's Earthquake Theorem:

\begin{theorem}[Earthquake Theorem]
Given any two hyperbolic metrics $h,h'$ on $\Sigma$, there exists a unique pair of measured geodesic laminations $(\mu_l,\mu_r)$ on $\Sigma$ such that $[h']=E^l_{\mu_l}([h])=E^r_{\mu_r}([h])$.
\end{theorem}

We mention here that in \cite{MR3231480} the Gauss map is considered for convex polyhedral surfaces $\Sigma$ in MGH Cauchy compact manifolds $M$. These convex polyhedral surfaces  are therefore contained in the complement of the convex core, and their Gauss map will be again set-valued. The bending locus of $\Sigma$, which replaces the bending lamination, induces two geodesic graphs on the left and right hyperbolic surfaces with different combinatorics, called \emph{flippable tilings}. Roughly speaking, this is because the image of the vertices of $\Sigma$ are mapped to hyperbolic polygons under the left and right projections, and the associated map ``flips'' these polygons with respect to the adjacent components of the complement of the geodesic graph. As a result of this construction, in \cite{MR3231480} the authors prove the existence of (many) left and right ``flip'' maps between any two closed hyperbolic surfaces of the same genus.

\subsection{The fundamental example} Finally, in order to understand how earthquake maps are associated to pleated surfaces, let us now consider the fundamental example. Let $S$ be a piecewise totally geodesic surface consisting of the union of two half-planes in $\AdSP{2,1}$ meeting along a spacelike geodesic, see Figure \ref{fig:pleated}. 

 \begin{figure}[htb]
\includegraphics[height=8.5cm]{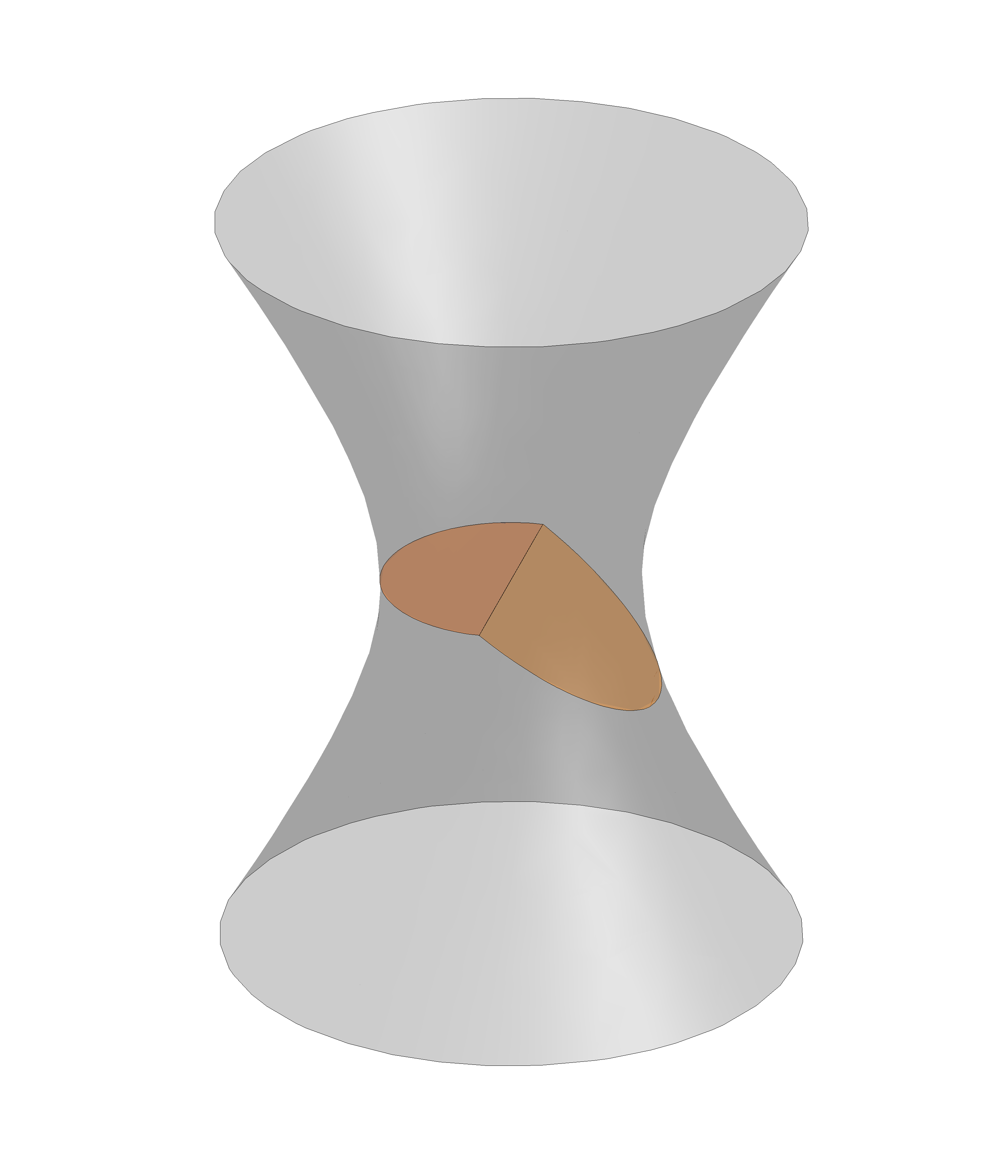}
\vspace{-1cm}
\caption{A pleated surface with bending locus a single geodesic, in an affine chart for $\AdSP{2,1}$.}\label{fig:pleated}
\end{figure}

Our aim is to understand the left and right projection for this surface $S$. Observe that these are well-defined in the complement of the spacelike geodesic which constitutes the bending locus of $S$.  As already observed above (see the first point in Section \ref{sec:consequences formula projections}), the projections $\Pi_l$ and $\Pi_r$ are isometric on each (totally geodesic) connected component of the complement of such bending geodesic in $S$. Let us call these two components $S_1$ and $S_2$.

We may assume that $S_1$ is contained in the plane $P_\En$, composed of order-two elliptic elements in $\PSL(2,\R)$. Therefore the bending locus is a spacelike geodesic contained in $P_\En$, namely the set of  order-two  elliptic elements having fixed point in a geodesic $\ell$ of $\Hyp^2$. From the notation of
 Section \ref{sec:geodesics PSL2R}, it has the form
$$L_{\ell,\ell'}=\{X\in\PSL(2,\R)\,|\,X\cdot \ell'=\ell\text{ as oriented geodesics}\}~,$$
where $\ell'$ is the same geodesic but endowed with the opposite orientation.

To understand the behaviour of the projections, the key point is to understand the stabilizer of the spacelike geodesic $L_{\ell,\ell'}$. This is a group isomorphic to $\R^2$, consisting of pairs $(A,B)\in\PSL(2,\R)\times\PSL(2,\R)$ where both $A$ and $B$ are hyperbolic transformations preserving $\ell$. The stabilizer of $L_{\ell,\ell'}$ fixes setwise also the dual geodesic, namely $L_{\ell,\ell}$ (Proposition \ref{prop:dual geodesics}). 

In fact, by the definition of dual geodesic (Definition \ref{def:dual geodesics}), the dual point of the spacelike plane $S_2$ lies in the dual geodesic, and is therefore a hyperbolic transformation $\sigma_0$ with axis $\ell$. Now, from the discussion of Section \ref{sec:gauss and projections} (see in particular Remark \ref{rmk gauss map alla mess}), the left projection $\Pi_l:S_1\sqcup S_2\to\Hyp^2$ is obviously the identity on $S_1$ (where we identify as usual $P_\En$ with a copy of $\Hyp^2$), while on $S_2$ it is given by multiplication on the right by $\sigma_0^{-1}$. Similarly, the right projection is the identity when restricted to $S_1$, and left multiplication by $\sigma_0^{-1}$ when restricted to $S_2$. 

In conclusion, the composition $\Pi_r\circ \Pi_l^{-1}$ acts on $P_\En$ as the identity on one connected component of the complement of $L_{\ell,\ell'}$, and conjugates by $\sigma_0$ on the other connected component --- which simply means acting by the hyperbolic transformation $\sigma_0$ under the identification between $P_\En$ and $\Hyp^2$ (Lemma \ref{rmk dual plane traceless2}). This is exactly an earthquake map with associated earthquake lamination the geodesic $\ell$. Since the angle between the spacelike planes containing $S_1$ and $S_2$ equals the distance in the dual geodesic $L_{\ell,\ell}$
 between the corresponding dual points, we also conclude that the bending measure equals the measure associated with the earthquake map. In short, the bending and earthquake measured laminations are identified.

{\large {\part{Further results}\label{part3}}}

In this part we will explain various results which have been obtained after the work of Mess. Of course we do not aim at an exhaustive treatment here; as mentioned already in the introduction, our choice is to underline mostly the relations between Anti-de Sitter geometry and Teichm\"uller theory. 

\vspace{0.3cm}
\section{More on MGH Cauchy compact AdS manifolds} \label{section:closed}
In this section we first consider MGH Cauchy compact manifolds, which have been studied in Chapter \ref{sec:GH AdS mfds}, with the purpose of describing more deeply their structure, their deformation space, and the applications to Teichm\"uller theory of closed hyperbolic surfaces.

\subsection{Foliations} \label{subsec:foliations}

A smoothly embedded spacelike surface $S$ in $\AdSP{2,1}$ has \emph{constant mean curvature} if its mean curvature, namely the trace of the shape operator $B$, is constant. We will mostly denote by $H$ the constant value of the mean curvature, whose sign actually depends on the choice of a normal unit vector (we will implicitly consider the \emph{future} unit normal vector here). A particular case are \emph{maximal surfaces}, for which $H=\tr(B)=0$. We will implicitly assume that surfaces are spacelike.

\begin{theorem}[{\cite{MR2328923}}] \label{thm CMC foliation compact}
Every maximal globally hyperbolic Anti-de Sitter manifold with  compact Cauchy surface is uniquely foliated by closed CMC surfaces, where the mean curvature $H$ varies in $(-\infty,+\infty)$.
\end{theorem}

In fact, for every $H$ the CMC surface $\Sigma_H$ is unique, as an application of the maximum principle. Moreover, the \emph{CMC function} $\tau:M\to\R$ which associates to $p$ the unique $H$ such that the CMC surface $\Sigma_{H}$ contains $p$ is a \emph{time function}, namely it is strictly increasing along future-directed causal curves.

The embedded surface $S$ has \emph{constant Gaussian curvature} if the determinant $\det B$ is constant. In this case  the value of the constant is well-defined, and we will consider here the case of positive Gaussian curvature, which will be denoted by $K\in (0,+\infty)$. Hence a CGC surface is either locally convex or locally concave, where the distinction between convex and concave is relative to the time orientation. 

\begin{theorem}[{\cite{MR2895066}}]\label{thm CGC foliation compact}
Let $M$ be a maximal globally hyperbolic Anti-de Sitter manifold with  compact Cauchy surface. Then each connected component of $M\setminus C(M)$ is uniquely foliated by closed CGC surfaces, where the Gaussian curvature $K$ varies in $(0,+\infty)$.
\end{theorem}

Again, each surface with constant Gaussian curvature $K$ is unique in its connected component. On each connected component the function which associates to every point the corresponding value of $K$ is again a time function (up to changing the sign if necessary). 

There is a remarkable relation between Theorem \ref{thm CMC foliation compact} and Theorem \ref{thm CGC foliation compact}, which is given by the normal evolution, and which will appear again in the generalizations of these foliation results to the setting of universal Teichm\"uller space. Let us introduce this relation here.

Recall that, given a spacelike immersion $\sigma:S\to\AdSP{2,1}$ with future unit normal vector field $\nu$, the normal evolution of $\sigma$ is defined as 
$$\sigma_t(x)=\exp_{\sigma(x)}(t\nu(x))~.$$
(See also the final item in Section \ref{sec:consequences formula projections}.) The computation of Lemma \ref{lemma metric tub neigh} shows that the pull-back of the ambient metric by $\sigma_t$ has the form:
\begin{equation}\label{eq:equidistant metric}
\sigma_t^*g_{\AdSP{}}=\I((\cos(t)\id+\sin(t)B)\cdot,(\cos(t)\id+\sin(t)B)\cdot)~,
\end{equation}
where as usual $\I$ is the first fundamental form of $\sigma$ and $B$ its shape operator. 

\begin{remark} \label{remark normal flow singular}
The pull-back $\sigma_t^*g_{\AdSP{}}$ might in general be degenerate, corresponding to the fact that the differential of $\sigma_t$ might be singular for some $t$. Under the identification between the tangent space of the image of $\sigma$ and $\sigma_t$ at $x$, $d\sigma$ and $d\sigma_t$ differ by pre-composition with $\cos(t)\id+\sin(t)B$, whose eigenvalues are $\cos(t)+\sin(t)\lambda_1$ and $\cos(t)+\sin(t)\lambda_2$, $\lambda_i$ being the principal curvatures.

Under certain conditions relating $B$ and $t$, however, one can make sure that the map $\sigma_t$ is an immersion. For instance, by compactness there exists $\epsilon>0$ (which depends on the norm of $B$) such that $\sigma_t$ is an immersion for $t\in(-\epsilon,\epsilon)$. A more significant condition is the following: if $\sigma$ is a convex immersion (meaning that $B$ is positive definite with respect to the future unit normal vector), then $\sigma_t$ is an immersion for positive times $t$, and of course the same holds for a concave immersion and negative times.
\end{remark}

The relation between CMC surfaces and CGC surfaces is then contained in the following statement, which also appears in \cite{MR3978545}. 

\begin{prop} \label{prop normal flow CMC CGC}
Let $\sigma:S\to\AdSP{2,1}$ be an immersion of constant Gaussian curvature $K>0$. Then the normal evolution $\sigma_{t_K}$ on the convex side of $\sigma$, for time $t_K=\arctan(K^{1/2})$ is an immersion of constant mean curvature $H=K^{-1/2}(K-1)$.
\end{prop}
\begin{proof}
By Remark \ref{remark normal flow singular} the normal evolution $\sigma_t$ is an immersion. Let us denote by $\I_t$ its first fundamental form, and by $B_t$ its shape operator. To check the relation between mean and Gaussian curvature, it is smarter to apply Equation \eqref{eq:equidistant metric} to $\sigma_t$ for negative times and the expression of the curvature given in Remark \ref{remark codazzi and curvature}. Then one obtains
$$K_\I=\frac{K_{\I_t}}{\det(\cos(t)\id-\sin(t)B_t)}=-\frac{1+\det B_t}{\cos^2(t)+\sin^2(t)\det B_t-\cos(t)\sin(t)\tr B_t}~.$$
Hence one can check that $K_\I$ (which equals $-1-\det B$) is constant if and only if $\tr B_t=2/\tan(2t)$, in which case an explicit computation shows that $\tan(t)=1/\sqrt K$. The result follows.
\end{proof}

We remark here that \emph{a priori}, the construction of Proposition \ref{prop normal flow CMC CGC} cannot be reversed, since the normal evolution $\sigma_t$ obtained from a  constant mean curvature immersion $\sigma$ might be singular at some points. From the proof of Proposition \ref{prop normal flow CMC CGC}, one sees that in fact this occurs if and only if the intrinsic curvature of the CMC immersion $\sigma$ vanishes, which is equivalent to the determinant of the shape operator being equal to $-1$. 

But \emph{a posteriori} the correspondence is indeed bijective when applied to the closed surfaces of constant mean and Gaussian curvature of Theorems \ref{thm CMC foliation compact} and \ref{thm CGC foliation compact}, as a consequence of the uniqueness statements. Indeed if a closed surface of constant mean curvature $H$ in a MGH spacetime $(M,g)$ were not obtained by the ``reversed'' normal evolution construction with respect to Proposition \ref{prop normal flow CMC CGC} (either in the future or in the past), then applying Proposition \ref{prop normal flow CMC CGC}  to an actual surface of the expected constant Gaussian curvature (given by Theorem \ref{thm CGC foliation compact}) one would find a new surface of constant mean curvature in $(M,g)$, thus contradicting the uniqueness of Theorem \ref{thm CMC foliation compact}. This means that each surface $\Sigma_H$ of constant mean curvature $H$ has two equidistant surfaces of constant Gaussian curvature $K_+$ and $K_-$ (which only depend on $H$), one convex in the past of $\Sigma_H$, the other concave in its future.

\subsection{Minimal Lagrangian maps and landslides}\label{subsec min lag landslides}

Using the results of the previous section, we can recover the existence of special maps between closed hyperbolic surfaces, as maps associated to surfaces with constant mean or Gaussian curvature.

\begin{defi}\label{defi minimal lag}
Given two hyperbolic metrics $h$ and $h'$ on a surface $S$, a smooth map $f:(S,h)\to (S,h')$ is \emph{minimal Lagrangian} if its graph is a minimal Lagrangian submanifold of $S\times S$  with  respect to the Riemannian product metric $h\oplus h'$ and the symplectic form $\pi_l^*\omega_{h}-\pi_r^*\omega_{h'}$.
\end{defi}

Given a maximal surface $\Sigma_0$ in a maximal globally hyperbolic spacetime $(M,g)$, with compact Cauchy surface $\Sigma$, we claim that the associated map $f_0$ is a minimal Lagrangian map from $(\Sigma,h)$ to $(\Sigma,h')$, where $h$ and $h'$ are the quotient metrics induced in $\Hyp^2/\rho_l(\pi_1\Sigma)$ and $\Hyp^2/\rho_l(\pi_1\Sigma)$. We have already discussed the Lagrangian condition, which amounts to $f_0$ being area-preserving and is always verified by the Gauss map image (Section \ref{sec:consequences formula projections}). The fact that the graph of $f_0$ is minimal in $(\Sigma\times\Sigma,h\oplus h')$ is a consequence of Proposition \ref{prop:formulae projections}.

In fact, we shall show that the Gauss map is conformal and harmonic, which implies that its image is a minimal surface. 
By Proposition \ref{prop:formulae projections} the pull-back of the product Riemannian metric has the expression $2(\I+\III)$.
When the trace of $B$ vanishes identically, by the Cayley-Hamilton theorem $B^2+(\det B)\id=0$, which implies $\III=-(\det B)\I$ showing conformality. Also, observe that the projections are local diffeomorphisms since by the previous section $\Sigma_0$ is obtained as the equidistant surface from a convex surface (of intrinsic curvature $-2$), the projections are always local diffeomorphisms on convex surfaces (Section \ref{sec:consequences formula projections}), and its Gauss map coincides with that of $\Sigma_0$ (Section \ref{sec:consequences formula projections}, last item). By a topological argument, the projections are then global diffeomorphisms from $\Sigma_0$ to $(\Sigma,h)$ and $(\Sigma,h')$.

The harmonicity of the Gauss map is equivalent to the harmonicity of each projection. Since the notion of harmonic map between Riemannian surfaces only depends on the conformal structure on the source, it suffices to show that 
$$\Pi_l:(\Sigma_0,\I)\to (\Sigma_0,\I((\id-\JJ\circ B)\cdot,(\id-\JJ\circ B)\cdot)$$ is harmonic, and the same for $\Pi_r$. To see this,  we can rewrite the target metric as $(\I+\III)-2\I((\JJ\circ B)\cdot,\cdot)$. We have used that $B$ is traceless and thus $\JJ\circ B$ is $\I$-symmetric. Together with the Codazzi property of $\JJ\circ B$, 
 this also implies that  $2\I((\JJ\circ B)\cdot,\cdot)$ is the real part of a holomorphic quadratic differential, in light of the following well-known fact, see \cite{MR40042,taubes,KS07}.
 
\begin{lemma}\label{lemma hopf}
Given a Riemannian metric $g$ on a surface and a $(1,1)$-tensor $A$, $A$ is traceless if and only if $g(A\cdot,\cdot)$ is the real part of a quadratic differential for the conformal structure of $g$. Moreover the quadratic differential is holomorphic if and only if $A$ is $g$-Codazzi.
\end{lemma}
 
Therefore $\Pi_l$ is harmonic. The same proof clearly holds for $\Pi_r$. This construction can actually be reversed, in the sense that every minimal Lagrangian map can be realized as the  map associated with a maximal surface. This permits to reprove the following theorem of existence and uniqueness of minimal Lagrangian diffeomorphisms in a given isotopy class:

\begin{theorem}[{\cite{labourieCP1,MR1201611}}]
Given a closed surface $\Sigma$ and two hyperbolic metrics $h$ and $h'$ on $\Sigma$, there exists a unique minimal Lagrangian diffeomorphism $f_0:(\Sigma,h)\to (\Sigma,h')$ isotopic to the identity.
\end{theorem}

Let us briefly turn our attention to \emph{landslides}, a natural generalization of minimal Lagrangian maps introduced in \cite{MR3035326}, which turn out to be precisely the maps associated to constant mean curvature and constant Gaussian curvature surfaces.

Given two hyperbolic metrics $h$ and $h'$ on a surface $S$, and $\theta\in(0,\pi)$ a $\theta$-landslide $f_\theta:(S,h)\to(S,h')$ is a smooth map which satisfies one of the equivalent conditions (see \cite[Section 4.3]{MR3789829} for more details and for the equivalence):
\begin{enumerate}[leftmargin=*]
\item There exists a smooth $(1,1)$-tensor $A$ such that (if $\JJ_h$ is the almost-complex structure of $h$): 
$$f_\theta^*h'=h(((\cos\theta)\id+(\sin\theta) \JJ_h\circ A)\cdot,((\cos\theta)\id+(\sin\theta) \JJ_h\circ A)\cdot)$$ which is positive-definite, $h$-symmetric, $h$-Codazzi and has unit determinant.
\item There exist harmonic maps $f:(S,X)\to (S,h)$ and $f':(S,X)\to (S,h')$, where $X$ is a conformal structure on $S$, such that $f_\theta=f'\circ f^{-1}$ whose Hopf differentials satisfy
$$\mathrm{Hopf}(f)=e^{2i\theta}\mathrm{Hopf}(f')~.$$
Moreover, in the non-compact case, one has to further impose that $f$ and $f'$ have the same holomorphic energy density. 
\end{enumerate}

When $\theta=\pi/2$ we recover minimal Lagrangian maps, as the above two conditions are in fact equivalent to Definition \ref{defi minimal lag}.
It then turns out that $\theta$-landslides are precisely the maps associated to surfaces of constant mean curvature $H=2/\tan\theta$, and therefore also to the two equidistant surfaces of constant Gaussian curvature $\tan^2(\theta/2)$ and $1/\tan^2(\theta/2)$. Hence the following result is a consequence of Theorem \ref{thm CMC foliation compact} (or Theorem \ref{thm CGC foliation compact}):

\begin{theorem}
Given a closed surface $\Sigma$ and two hyperbolic metrics $h$ and $h'$ on $\Sigma$, and $\theta\in(0,\pi)$, there exists a unique diffeomorphism $f_\theta:(\Sigma,h)\to (\Sigma,h')$ isotopic to the identity which is a $\theta$-landslide.
\end{theorem}

It is worth remarking that, when $\theta$ approaches $0$, then one of the two associated surfaces of constant Gaussian curvature (namely the one having Gaussian curvature $\tan^2(\theta/2)$) approaches a boundary component of the convex core of the ambient manifold $(M,g)$, while the other escapes at infinity in the other end of $(M,g)$. When $\theta$ approaches $\pi$ instead, the roles are switched. Hence the landslide maps $f_\theta$ converge to the left and right earthquake maps between $(\Sigma,h)$ and $(\Sigma,h')$ as $\theta$ diverges in its interval of definition $(0,\pi)$. 
Morally, $\theta$-landslides are a natural one-parameter family of smooth extensions which interpolate between left earthquake, minimal Lagrangian maps, and right earthquakes.

As a final remark for this section, recall from Section \ref{sec:consequences formula projections} that the area-preserving condition for maps from $(\Sigma,h)$ to $(\Sigma,h')$, or more generally the Lagrangian condition for submanifolds of $\Sigma\times\Sigma$ (endowed with the symplectic form induced in the quotient by \eqref{eq symplectic form}), are the only \emph{local} obstructions to reversing the Gauss map construction. 

Roughly speaking, this means that any Lagrangian immersion of a simply connected surface in $\Hyp^2\times\Hyp^2$ can be realized as the Gauss map image of a spacelike immersion in $\AdSP{2,1}$.
However, if the Lagrangian immersion in $\Hyp^2\times\Hyp^2$ is equivariant with respect to a pair of Fuchsian representations $\rho=(\rho_l,\rho_r)$, the corresponding immersion in $\AdSP{2,1}$ is not necessarily $\rho$-equivariant. 


There is indeed an additional obstruction to reversing the Gauss map construction \emph{globally}. As mentioned already in Section \ref{subsec unit tangent}, this obstruction has been studied in 
 \cite{MR3937298} and \cite{MR3745436} 
 in terms of an orbit of the group of Hamiltonian symplectomorphisms. This obstruction translates to a cohomological vanishing condition by means of the \emph{flux homomorphism}, a tool from symplectic geometry.

\subsection{Cotangent bundle of Teichm\"uller space} \label{subsec cotangent}

The existence and uniqueness of maximal and constant mean curvature surfaces can be remarkably applied to provide new parameterizations of the deformation space of MGH Cauchy compact Anti-de Sitter three-manifolds. The fundamental observation is that, from Lemma \ref{lemma hopf}, given a maximal surface $\Sigma$ of constant mean curvature $H$, the second fundamental form of $\Sigma$ uniquely determines a holomorphic quadratic differential $\alpha$ for the conformal structure associated to $\I$. Hence the conformal class of $g$, together with $\alpha$, determines an element of the cotangent bundle $T^*\Teich(\Sigma)$.

By virtue of the following theorem, the construction can be perfectly reversed. This approach is known in the physics literature as \emph{ADM reduction} based on the article \cite{zbMATH03150642}, see also \cite{moncrief} and \cite[Chapter 2]{carlipbook}.

\begin{theorem}[{\cite[Lemma 3.10]{KS07}}]\label{thm KS}
Given a Riemannian metric $g$ on a closed surface $\Sigma$ of genus at least $2$ and a holomorphic quadratic differential $\alpha$ for $g$, there exists a unique germ of maximal spacelike embedding in a MGH Anti-de Sitter manifold having first fundamental form conformal to $g$ and second fundamental form the real part of $\alpha$.
\end{theorem}

The proof roughly consists in applying PDE methods to find a function $f$ such that the Riemannian metric $g'=e^{2f}g$, together with the real part of $\alpha$, solves the Gauss equation. The Codazzi equation is still automatically satisfied as a consequence of Lemma \ref{lemma hopf}, and therefore the pair $(g',\mathrm{Re}(\alpha))$ determines the embedding data of a maximal surface using Theorem \ref{thm:fund theorem}. 

\begin{remark}
In \cite[Lemma 3.10]{KS07} the proof is actually given in the more general  case of surfaces of constant mean curvature $H\in(-1,1)$, where now the traceless part of the second fundamental form, namely $\II_0=\II-(H/2)\I$, is the real part of a holomorphic quadratic differential.
\end{remark}

Theorem \ref{thm KS}, together with Theorem \ref{thm CMC foliation compact} and the parameterization of Mess, then provides a homeomorphism $F:T^*\Teich(S)\to\Teich(S)\times\Teich(S)$. It was asked in \cite[Question 8.1]{open} whether this map is a symplectomorphism with respect to the natural symplectic structures $\Omega_{\mathrm{COT}}$ on the cotangent bundle, and 
$\pi_1^*\Omega_{\mathrm{WP}}-\pi_2^*\Omega_{\mathrm{WP}}$ 
on $\Teich(S)\to\Teich(S)$, where $\Omega_{\mathrm{WP}}$ is the Weil-Petersson symplectic form on $\Teich(S)$.
The answer is affirmative up to a multiplicative factor.

\begin{theorem}[{\cite[Theorem 1.11]{MR3807587}}] \label{thm scarinci schlenker}
The map $F:T^*\Teich(S)\to \Teich(S)\times\Teich(S)$ satisfies
$$F^*(\pi_1^*\Omega_{\mathrm{WP}}-\pi_2^*\Omega_{\mathrm{WP}})=-2\Omega_{\mathrm{COT}}~.$$
\end{theorem}

\subsection{Volume}

In this section we will briefly mention the results of \cite{MR3714718} on the volume of MGH Cauchy compact Anti-de Sitter manifolds, which is an interesting invariant on the deformation space $\Teich(S)\times\Teich(S)$. Here we are mostly interested in the coarse behaviour of the volume function. 

A first foundational fact is that the volume of a MGH Cauchy compact manifold is finite, and is coarsely comparable to the volume of the convex core, up to a constant which only depends on the topology. 

\begin{prop}[{\cite[Proposition G]{MR3714718}}]
Given a MGH Anti-de Sitter manifold $M$ with compact Cauchy surface homeomorphic to $\Sigma$, the volume of $M\setminus C(M)$ is at most $\pi^2|\chi(\Sigma)|/2$, with equality if and only if $M$ is Fuchsian.
\end{prop}

Special surfaces  in $M$ can then be used to obtain bounds on the volume in terms of certain quantities defined on the deformation space $\Teich(S)\times\Teich(S)$, related to energies of $L^1$-type, by means of their associated maps (for instance, earthquake maps from pleated surfaces, and minimal Lagrangian maps from maximal surfaces). 

More concretely, given a $C^1$-map $f:(\Sigma,h)\to (\Sigma,h')$, where $h$ and $h'$ are hyperbolic metrics on $\Sigma$, the  $L^1$-\emph{energy} of $f$ is defined
 as:
 
$$E_1(f)=\int_\Sigma ||df||\mathrm{dA}_h~,$$
where the norm $||df||$ of the differential is computed with respect to the metrics $h$ and $h'$, and 
 $\mathrm{dA}_h$ denotes the area form of $h$. Unlike the more studied $L^2$-energy, which is the integral of $||df||^2$, the $L^1$-energy is not conformally invariant on the source, but fully depends on the Riemannian metrics both on the source and on the target.

\begin{remark}
The $L^1$-energy can be defined under weaker regularity assumptions on $f$, and in fact it coincides with the notion of total variation for BV maps. For earthquake maps, the total variation is essentially the length of the earthquake lamination, up to a constant which only depends on the genus. The latter is defined for simple closed curves as the product of the length of the $h$-geodesic realization of the simple closed curve and its weight, and is then extended to general measured laminations by a continuity argument.
\end{remark}

Another important energy of $L^1$-type is the \emph{holomorphic energy}, which is defined as
$$E_\partial(f)=\int_\Sigma ||\partial f||\mathrm{dA}_h~,$$
and was studied in \cite{MR1371212}. In particular, it was shown that minimal Lagrangian diffeomorphisms are minima of $E_\partial(f)$ on the space of diffeomorphisms from $(\Sigma,h)$ to $(\Sigma,h')$ isotopic to the identity.

Let us now summarize the results of \cite{MR3714718}, although we omit some of the details here. We say that two quantities $f,g$, defined here on $\Teich(\Sigma)\times\Teich(\Sigma)$, are \emph{coarsely equivalent} if there exist positive constants $M_1,M_2,A_1,A_1$ ($M$ for multiplicative and $A$ for additive) such that
$$M_1 f-A_1\leq g\leq M_2 f+A_2~.$$

\begin{theorem}[{\cite{MR3714718}}]
The following quantities are coarsely equivalent over $\Teich(\Sigma)\times\Teich(\Sigma)$, with explicit multiplicative constants (universal) and additive constants (which instead depend only on the genus of $\Sigma$):
\begin{itemize}
\item the volume of the MGH Cauchy compact manifold with left and right metrics $h$ and $h'$;
\item the volume of its convex core;
\item the infimum of the $L^1$-energy over $C^1$ maps from $(\Sigma,h)$ to $(\Sigma,h')$ homotopic to the identity;
\item the length of the left earthquake lamination from $(\Sigma,h)$ to $(\Sigma,h')$;
\item the length of the right earthquake lamination from $(\Sigma,h)$ to $(\Sigma,h')$;
\item the holomorphic energy density of the minimal Lagrangian diffeomorphism $f_0:(\Sigma,h)\to(\Sigma,h')$ isotopic to the identity.
\end{itemize}
\end{theorem}

It is also worth mentioning that the multiplicative constants from above and below for the length of the left and right earthquake maps all agree, hence one obtains as a corollary that, given two points in Teichm\"uller space, the length of the left and right earthquake laminations differ by at most a constant which only depends on the topology (explicitly, the constant is $2\pi^2|\chi(\Sigma)|$).

Bounds on the volume are obtained also in terms of the exponential of the Thurston distance between the corresponding points in $\Teich(\Sigma)$ (either 
of the two asymmetric distances), from above, and in terms of the exponential of the Weil-Petersson distance, from below, always up to additive and multiplicative constants, both depending on the topology in this case. These results answer to some extent Question 4.1 of \cite{open}.
 
\subsection{Realization of metrics and laminations} \label{subsec realization}

A consequence of the pleated surface construction of Section \ref{sec:nonsmooth} is that the geometry of a MGH  Anti-de Sitter manifold with  compact Cauchy surface homeomorphic to $\Sigma$ is determined by the pair of a hyperbolic metric $h$ on $\Sigma$ and a measured geodesic lamination. In fact, lifting the measured geodesic lamination $\mu$ on the universal cover $\Hyp^2$ of $(\Sigma,h)$, one can realize a pleated surface having bending lamination $\mu$, which will be equivariant for some  representation $\rho:\pi_1\Sigma\to\PSL(2,\R)\times\PSL(2,\R)$. Such pleated surface maps in the quotient to a  boundary component of the convex core of a MGH Anti-de Sitter manifold.

Theorem \ref{thm KS} is, to some extent, a smooth analogue, where pleated surfaces are replaced by maximal surfaces, and the data of a holomorphic quadratic differential is a measure of the curvature of the surface. It is then a natural question to ask,  if the geometry of the MGH Cauchy compact manifold is uniquely determined by other pairs, for instance the two bending laminations on the boundary components of the convex core, or the hyperbolic metrics induced, or even by the induced metrics on a pair of smooth surfaces. 
These questions can therefore provide new parameterizations of the deformation space, and are of course motivated also by their counterparts for quasi-Fuchsian hyperbolic manifolds. We briefly collect here the state-of-the-art on these questions. Most of these questions were asked in \cite[Section 3]{open}.

\begin{theorem}[{\cite[Theorem 1.4]{MR2913100}}] \label{thm bon schl prescribe lam}
Given a compact surface $\Sigma$ and two measured laminations $\mu_-,\mu_+$ which fill $\Sigma$, there exists a MGH Anti-de Sitter manifold homeomorphic to $\Sigma\times\R$ such that the bending measured laminations of the lower and upper boundary components of the convex core are isotopic to $\mu_-$ and $\mu_+$ respectively.
\end{theorem}

The statement can be equivalently reformulated as the fact that, given any two measured laminations which fill $\Sigma$, the composition of the corresponding left earthquakes, seen as maps from $\Teich(\Sigma)$ to itself, has a fixed point. See  {\cite[Theorem 1.1]{MR2913100}}. Observe that, similarly to the quasi-Fuchsian case, the hypothesis that $\mu_-$ and $\mu_+$ fill $\Sigma$ is a necessary condition. Uniqueness of the MGH manifold (up to isotopy of course) is still open. 

The following result of Diallo concerns the prescriptions of the induced hyperbolic metric on the boundary of the convex core.

\begin{theorem}[{\cite{diallo}}] \label{thm diallo}
Given a compact surface $\Sigma$ and two hyperbolic metric  $h_-,h_+$, there exists a MGH Anti-de Sitter manifold homeomorphic to $\Sigma\times\R$ such that induced metrics on the lower and upper boundary components of the convex core are isotopic to $h_-$ and $h_+$ respectively.
\end{theorem}

The following result is a smooth analogue, for convex/concave surfaces lying outside the convex core.

\begin{theorem}[{\cite{tambIMRN}}]\label{thm:tambIMRN}
Given a compact surface $\Sigma$ and two Riemannian metrics  $g_-,g_+$ of curvature $<-1$, there exists a MGH Anti-de Sitter manifold homeomorphic to $\Sigma\times\R$ containing a convex surface $\Sigma_-$ and concave surface $\Sigma_+$ whose first fundamental forms are isotopic to $g_-$ and $g_+$ respectively.
\end{theorem}

Again, uniqueness is not known in general. When the curvature of $g_-$ and $g_+$ is equal to $-2$, uniqueness is proved in \cite[Theorem 3.21]{KS07} by proving that $g_-$ and $g_+$ can be uniquely realized as the metrics on the equidistant surfaces from the maximal surface (in the sense of Proposition \ref{prop normal flow CMC CGC} for $H=0$, $K=-1$). This gives a new parameterization of the deformation space by $\Teich(S)\times\Teich(S)$. 

When the two metrics $g_-$ and $g_+$ coincide, then Theorem \ref{thm:tambIMRN} had been already obtained in \cite{MR1752780}, by showing that there exists a Fuchsian realization. This has been recently generalized by Labeni in \cite{labeni}, showing that one can realize any locally CAT(-1) distance on $\Sigma$ as the induced distance on a convex surface in a Fuchsian MGH Anti-de Sitter spacetime. The result of Labeni generalizes also \cite{MR2794916}, which concerns the realizability of a hyperbolic metric with cone singularities. It is natural to expect that any two locally CAT(-1) distances can be realized, probably uniquely, in a (non-Fuchsian, in general) MGH Anti-de Sitter manifold, but this is still an open question.

\section{Non-closed surfaces} \label{section:non closed}

In this last chapter we will survey other results where the topology of spacelike surfaces is not that of a closed surface. We will first discuss a number of \emph{universal} constructions, meaning that they generalize the situation one sees in the universal covering of a MGH Cauchy compact Anti-de Sitter manifold, which was explained in Chapter \ref{sec:GH AdS mfds}. This will have applications for the theory of universal Teichm\"uller space. Then we briefly discuss the state-of-the-art for manifolds with conical singularities of timelike type, called \emph{particles} and corresponding to Cauchy surfaces with cone points, and the so-called \emph{multi-black holes} which instead correspond to surfaces with boundary.

\subsection{Foliations with asymptotic boundary}\label{sec:foliations universal}

Recall from Chapter \ref{sec:GH AdS mfds} and in particular Proposition \ref{pr:MGHADS} that, given a MGH Cauchy compact manifold $M$, any Cauchy surface in $M$ lifts to a spacelike embedded surface in $\AdSP{2,1}$ having asymptotic boundary a proper achronal meridian $\Lambda$ which is the graph of the unique homeomorphism conjugating the left and right representations of $\pi_1\Sigma$ in $\PSL(2,\R)$. 

Some of the constructions we discussed above can be generalized to the setting of any proper achronal meridian. Recall that $\Omega(\Lambda)$ denotes the invisible domain of $\Lambda$, see Section \ref{subsec:invisible}. 

\begin{theorem}\label{thm CMC universal}
Let $\Lambda\subset\Quno$ be the graph of an orientation-preserving homeomorphism. Then there exists a unique foliation of the invisible domain $\Omega(\Lambda)$ by properly embedded spacelike surfaces of constant mean curvature $H$, as $H\in(-\infty,+\infty)$.
\end{theorem}

The CMC function associated to the foliation is a  time function, similarly to the compact case (see Theorem \ref{thm CMC foliation compact} and the following discussion). 

In the literature Theorem \ref{thm CMC universal} does not appear  as it is stated here. The existence for maximal surfaces (that is $H=0$) is proved in \cite[Theorem 1.6]{bon_schl_inv} (where the statement is indeed given in any dimension). The result for any value of $H$, including uniqueness and the foliation statement, appears in \cite[Theorem 3.1]{MR3885182} again under the assumption that $\Lambda$ is the graph of a quasi-symmetric homeomorphism.

Moreover, as in the compact case, for every $H$ the surface of constant mean curvature $H$ and asymptotic boundary $\Lambda$ is unique.
This is proved in \cite[Theorem 1.10]{bon_schl_inv} for $\Lambda$ the graph of a quasi-symmetric homeomorphism (Definition \ref{defi quasisym} below) and under the additional assumption of bounded second fundamental form. Using the foliation result, uniqueness for the CMC surfaces is showed in \cite[Theorem 5.2]{MR3885182} again assuming that $\Lambda$ the graph of a quasi-symmetric homeomorphism.

The proofs, however, can be extended without further difficulty to the general case of any orientation-preserving homeomorphism. We believe that, by a refinement of the arguments, the statement can also be proved for $\Lambda$ any proper achronal meridian.

When $\Lambda$ is the graph of an orientation-preserving homeomorphism, the existence part of Theorem \ref{thm CMC universal} can be actually be obtained as a straightforward consequence of the following result for constant Gaussian curvature, by applying the normal evolution described in Section \ref{subsec:foliations}.

\begin{theorem}[{\cite[Theorem 1.3]{MR3789829}}]\label{thm CGC universal}
Let $\Lambda$ be any proper achronal meridian in $\Quno$ which is not a two-step curve. Then there exists a foliation of each connected component of the complement of the convex hull $C(\Lambda)$ in the invisible domain $\Omega(\Lambda)$ by spacelike surfaces of constant Gaussian curvature $K$, as $K\in(0,+\infty)$.
\end{theorem}

Here we say that a proper achronal meridian $\Lambda\subset\Quno$ is a two-step curve if it is the union of four lightlike segments, two horizontal and two vertical in an alternate fashion, under the natural identification of $\Quno$ with $\RP^1\times\RP^1$. Up to isometry, the configuration of a two-step curve is the one drawn in Figure \ref{fig:duallines}. 

In other words, every domain of dependence in $\AdSP{2,1}$ which is not the invisible domain of  a two-step curve admits a foliation of the complement of the convex hull by surfaces of constant Gaussian curvature. The Gaussian curvature function is a time function, up to changing the sign as for the closed case (Theorem \ref{thm CMC foliation compact}). Moreover in \cite[Theorem 1.4]{MR3789829} uniqueness of the surfaces of constant Gaussian curvature $K$ in each connected component is proved under the assumption that $\Lambda$ is the graph of a quasisymmetric homeomorphism.

 Theorems \ref{thm CMC universal} and \ref{thm CGC universal} provide affirmative answers to Questions 8.3 and 8.4 of \cite{open} respectively.

\begin{remark} \label{rmk:Ksurfaces sawtooth}
Unlike Theorem \ref{thm CMC universal}, the surfaces $S^\pm_K$ of constant Gaussian curvature $K\in(0,+\infty)$ are not always properly embedded in $\AdSP{2,1}$. Their boundary, however, is explicitely described: it coincides with $\Lambda$ in the complement of all {sawteeth} of $\Lambda$, where by a \emph{sawtooth} we mean two consecutive lightlike segments in $\Quno$, one horizontal and one vertical, which are maximal (meaning that they cannot be extended to longer lightlike segments). In correspondence of each sawtooth, the boundary of $S^\pm_K$ has an interior spacelike geodesic, namely the geodesic connecting the two endpoints of the sawtooth. See Figure \ref{fig:sawtooth}, and compare also with Remark \ref{rmk:light triangles convex invisible}. In conclusion, the surfaces $S^\pm_K$ are actually properly embedded in $\Omega(\Lambda)$.
When $\Lambda$ is the graph of an orientation-preserving homeomorphism, it has no sawteeth, hence the boundary of $S^\pm_K$ is precisely $\Lambda$, hence in this case $S^\pm_K$ is properly embedded also in $\AdSP{2,1}$.
\end{remark}

\begin{figure}[htb]
\includegraphics[height=8.5cm]{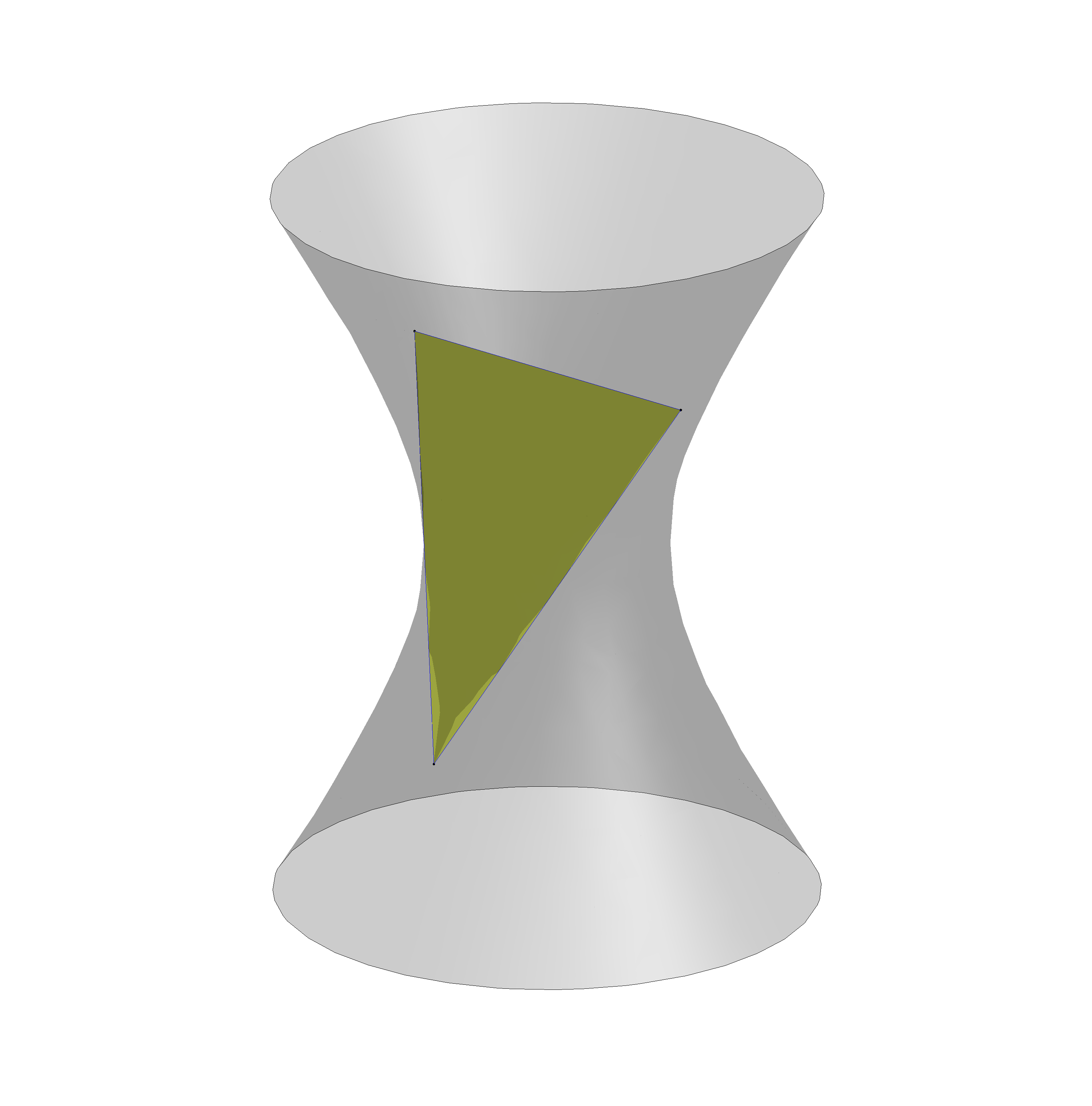}
\vspace{-1cm}
\caption{A lightlike triangle, whose boundary consists of a sawtooth, contained in $\Quno$, and a spacelike geodesic of $\AdSP{2,1}$.}\label{fig:sawtooth}
\end{figure}

\begin{remark}
The invisible domain of a two-step curve coincides with its convex hull, hence clearly there can be no existence for surfaces of constant Gaussian curvature in $(0,+\infty)$ in this case, since such a surface would be either convex or concave and therefore live in the convex hull complement, which is empty. 
On the other hand, when $\Lambda$ is a two-step curve, there does exists a foliation by surfaces of constant mean curvature with asymptotic boundary $\Lambda$, which is given by the surfaces with $z=c$ in the parameterization of $\Omega(\Lambda)$ of Lemma \ref{lemma:metric tetra} and pictured in Figure \ref{fig:foliationCMCflat}.
 \end{remark}

\subsection{Extensions and universal Teichm\"uller space}

The results of Section \ref{sec:foliations universal} have applications to the extensions of circle homeomorphisms, by means of the associated map which has already been discussed in the closed case, together with the relevant definitions, in Section \ref{subsec min lag landslides}.

\subsubsection*{Extensions of circle homeomorphisms} The essential lemma to obtain extensions of circle homeomorphisms is the following. Recall from Section \ref{sec:bdy in PSL2R} that we denoted by $\pi_l,\pi_r$ the projections from $\Quno$ to $\partial\Hyp^2$ which come from the left and right rulings, or in other words, the projections to the first and second factor in the identification of $\Quno$ with $\RP^1\times\RP^1$. 

\begin{lemma}[{\cite[Lemma 3.18, Remark 3.19]{bon_schl_inv}}]\label{lemma extension}
Let $\Lambda\subset\Quno$ the graph of an orientation-preserving homeomorphisms and $S$ be a spacelike convex (or concave) surface in $\AdSP{2,1}$ with boundary at infinity $\Lambda$. Then the left and right projections $\Pi_l,\Pi_r:S\to\Hyp^2$ extend to the restrictions of $\pi_l,\pi_r:\RP^1\times\RP^1\to\RP^1$ on $\Lambda$.
\end{lemma}

We do not repeat the computations of Section \ref{subsec min lag landslides} here, but essentially for every $\theta\in(0,\pi)$ there are three special surfaces with asymptotic boundary $\Lambda$, equidistant from one another, one with constant mean curvature $H=2/\tan\theta$, one convex with constant Gaussian curvature $\tan^2(\theta/2)$ and one concave with constant Gaussian curvature $1/\tan^2(\theta/2)$, which all have  associated map a $\theta$-landslide. Hence Theorem \ref{thm CMC universal}, or Theorem \ref{thm CGC universal}, together with the extension lemma (Lemma \ref{lemma extension}), imply the following result:

\begin{cor} \label{cor extension circle homeo}
Given any orientation-preserving homeomorphism $\varphi:\mathbb S^1\to \mathbb S^1$ and any $\theta\in(0,\pi)$, there exists a $\theta$-landslide $\Phi_\theta:\Hyp^2\to\Hyp^2$ whose extension to $\mathbb S^1$ equals $\varphi$.
\end{cor}

Since surfaces of constant mean curvature or Gaussian curvature are smooth, the extension $\Phi_\theta$ will be smooth on $\Hyp^2$, and continuous up to the boundary. 

\subsubsection*{Quasiconformal mappings} We now briefly recall some notions of the theory of quasiconformal mappings. Since here we will be only interested in smooth maps, we give a simplified definition under the smoothness assumption. Here we denote by $\D$ the unit disc in $\C$, as a Riemann surface. A diffeomorphism of $\D$ is quasiconformal if its differential maps circles in $T_z\D\cong \C$ to ellipses of uniformly bounded eccentricity. More formally:

\begin{defi}\label{defi quasiconf}
A diffeomorphism $\Phi:\D\to\D$ is \emph{quasiconformal} if
$$K(\Phi):=\sup_{z\in\D}\left(\frac{\text{largest eigenvalue of }d\Phi_z^Td\Phi_z}{\text{smallest eigenvalue of }d\Phi_z^Td\Phi_z}\right)^{1/2}<+\infty~.$$
In this case, the quantity $K(\Phi)$ is called \emph{maximal dilatation} of $\Phi$.
\end{defi}

Quasiconformal mappings of $\D$ are precisely those which extend to  homeomorphisms of the circle which are quasisymmetric, namely, they distort the cross-ratio of symmetric quadruples in a uniformly controlled way.

\begin{defi}\label{defi quasisym}
An orientation-preserving homeomorphism $\varphi:\partial\D\to \partial\D$ is \emph{quasisymmetric} if
$$||\varphi||:=\sup_{\mathrm{cr}(Q)=-1}|\log|\mathrm{cr}(\varphi(Q))||<+\infty~.$$
In this case, the quantity $||\varphi||$ is called \emph{cross-ratio norm} of $\varphi$.
\end{defi}

Some explanation is required concerning Definition \ref{defi quasisym}. In the definition of $||\varphi||$, the supremum is taken over all quadruples of points in $\partial\D$ which have cross-ratio equal to $-1$. We use here a definition of cross-ratio such that a quadruple $(p,q,r,s)$ of points of $\partial\D$ has cross-ratio $-1$ if and only if it is symmetric, meaning that the two geodesics  for the Poincar\'e metric on $\D$ with endpoints $(p,r)$ and $(q,s)$ intersect orthogonally.

\begin{remark} \label{rmk universal teich}
Observe also that the cross-ratio norm is a well-defined invariant of a curve $\Lambda$ in the boundary of $\AdSP{2,1}$, since applying isometries of $\AdSP{2,1}$, which corresponds to pre- and post-composing $\varphi$ with elements of $\PSL(2,\R)$, do not change the cross-ratio norm. The cross-ratio norm is indeed well-defined on \emph{universal Teichm\"uller space} $\Teich(\D)$, which can be defined as the space of quasisymmetric homeomorphism up to post-composition with $\PSL(2,\R)$.
\end{remark}

\begin{theorem}[{\cite{MR86869}}]
Any quasiconformal diffeomorphism of $\D$ extends to a quasisymmetric homeomorphism of $\partial\D$. Conversely, every quasisymmetric homeomorphism of $\partial\D$ admits a quasiconformal extension to $\D$.
\end{theorem}

Back to Anti-de Sitter geometry, the following result is proved in  \cite[Theorem 1.4]{bon_schl_inv} for minimal Lagrangian extensions (i.e. $\theta=\pi/2$) using maximal surfaces, and in \cite[Corollary 1.5]{MR3789829} in full generality, using surfaces of constant Gaussian curvature. 

\begin{theorem} \label{thm quasiconformal extension landslide}
Given any quasisymmetric homeomorphism $\varphi:\partial\D\to\partial\D$ and any $\theta\in(0,\pi)$, there exists a unique quasiconformal diffeomorphism $\Phi_\theta:\D\to\D$ extending $\varphi$ which is a $\theta$-landslide. 
\end{theorem}

This result, which can be seen as a version of the Schoen Conjecture (\cite{MR1201611,zbMATH06731860,zbMATH06699459}) for minimal Lagrangian maps and $\theta$-landslides, has essentially two additional points with respect to Corollary \ref{cor extension circle homeo}. The first is the fact that one can find an extension which is \emph{quasiconformal}, and the second is the uniqueness \emph{under the quasiconformality assumption}. 

For the former, the proof given in \cite{MR3789829} essentially consists in showing that the surfaces with constant Gaussian curvature with asymptotic boundary the graph of $\varphi$ have bounded principal curvatures, a condition which translates in the quasiconformality of the associated map. We will discuss the ingredients of the proof given in \cite{bon_schl_inv} below together with other applications. 

For the uniqueness instead, roughly speaking the key point is that not only the construction of the associated map can be reversed, but moreover if one starts with a quasiconformal extension (minimal Lagrangian, or landslide more  generally), then this can be obtained as the associated map of a surface $S$ (which can be taken of constant mean curvature or of constant Gaussian curvature) in $\AdSP{2,1}$ whose first fundamental form is complete. Then $S$ has asymptotic boundary a curve $\Lambda$ in $\Quno$, which by Lemma \ref{lemma extension} coincides necessarily with the graph of $\varphi$. The uniqueness of the extension $\Phi_\theta$ then follows from the uniqueness of the surface, as mentioned in the comments after Theorem \ref{thm CMC universal} and Theorem \ref{thm CGC universal}.

\subsubsection*{Optimality of minimal Lagrangian extensions}

An essential ingredient in the proof of Theorem \ref{thm quasiconformal extension landslide} given in \cite{bon_schl_inv} for minimal Lagrangian extension is the \emph{width} of the convex hull of $\Lambda$, which is defined for every proper achronal meridian $\Lambda$ as the supremum of the length of timelike curves contained in $C(\Lambda)$.  It turns out that the width is always at most $\pi/2$. One then has the following equivalent conditions:
\begin{itemize}
\item the proper achronal meridian $\Lambda$ is the graph of a quasisymmetric homeomorphism;
\item the width of the convex hull $C(\Lambda)$ is strictly less than $\pi/2$;
\item the principal curvatures of the maximal surface with asymptotic boundary $\Lambda$ are in $(-1+\epsilon,1-\epsilon)$ for $\epsilon>0$;
\item the minimal Lagrangian associated map is quasiconformal.
\end{itemize}

These four points all played an essential r\^ole in the proof of the following theorem.

\begin{theorem}[{\cite[Corollary 2.D]{MR3945744}}] \label{thm seppi estimate}
There exists a universal constant $C>0$ such that, for every quasisymmetric homeomorphism $\varphi:\partial\D\to\partial\D$, the maximal dilatation of the quasiconformal minimal Lagrangian extension $\Phi$ satisfies:
$$\log K(\Phi)\leq C||\varphi||~.$$
\end{theorem}

Theorem \ref{thm seppi estimate} thus adresses a question asked in \cite[Section 4.3]{open} about the efficiency of minimal Lagrangian extensions in terms of Teichm\"uller distance. Let us outline the strategy of the proof. Essentially, the equivalence of the four equivalent points mentioned above need to be quantified. One can in fact obtain quantitative inequalities between the cross-ratio norm of $\varphi$ and the width of the convex hull, between the width and the principal curvatures of the maximal surface, and finally  between the principal curvatures and the maximal dilatation of the minimal Lagrangian extension. Putting together all the quantitative estimates, one obtains the inequality of Theorem \ref{thm seppi estimate}.

There are two extreme cases to be understood. The first is when $\varphi$ is the graph of a transformation in $\PSL(2,\R)$, which is equivalent to the maximal surface being totally geodesic, to the width of the convex hull being equal to $0$ and to the minimal Lagrangian extension being isometric. In this case $\|\varphi\|=0$ and $K(\Phi)=1$. The qualitatively opposite case occurs for a two-step curve $\Lambda$, namely the concatenation of four lightlike segments, which we described in detail in Section \ref{sec:torus examples}. In this case the maximal surface is given by $\{z=\pi/4\}$ in the coordinates introduced in Lemma \ref{lemma:metric tetra} on the lightlike tetrahedron $\mathscr T$. See also Figure \ref{fig:foliationCMCflat}. Here the maximal surface is intrinsically flat, hence by the Gauss equation its principal curvatures are $1$ and $-1$ at every point. It follows from the discussion of Section \ref{sec:consequences formula projections} that in this case there is no associated map at all between subsets of $\Hyp^2$. The width of the convex hull, which is the tetrahedron $\mathscr T$ itself, equals $\pi/2$. The proof of Theorem \ref{thm seppi estimate} roughly speaking consists in showing that, as $\|\varphi\|$ approaches $0$ or $+\infty$, the geometry of the corresponding maximal surface approaches that of the two extreme examples, the totally geodesic plane and the flat maximal surface with principal curvatures identically $1$ and $-1$.

\subsection{Related results}

We briefly mention here some related results. The work \cite{MR3073162} studies a ``universal'' version of the correspondence between MGH Anti-de Sitter manifolds and Teichm\"uller space, and parameterizes a suitably defined moduli space of these structures by the product of two copies of the universal Teichm\"uller space $\Teich(\D)$. (Recall Remark \ref{rmk universal teich} for the definition of $\Teich(\D)$.) Moreover, a universal version of the map $F$ of Theorem \ref{thm scarinci schlenker} is constructed, namely a map $T^*\Teich(\D)\to \Teich(\D)\times\Teich(\D)$, using the fact that the cotangent bundle $T^*\Teich(\D)$ is identified to the space of bounded holomorphic quadratic differentials. 

A qualitatively opposite situation is described in \cite{MR3964153}, where maximal surfaces whose first fundamental form is conformal to $\C$, and the second fundamental form is the real part of a polynomial quadratic differential on $\C$, are considered. It is shown that these maximal surfaces are characterized by having asymptotic boundary a curve in $\Quno$ composed of the concatenation of a finite number of lightlike sides.

A result of prescription of the induced metric on convex surfaces of constant Gaussian curvature, generalizing to the universal setting some of the results discussed in Section \ref{subsec realization}, is proved in \cite{bdms}. In a similar spirit, results about the realization of metrics on the boundary of ideal polyhedra in $\AdSP{2,1}$ are presented in \cite{dms}, and a first result on the prescription of bending laminations on the boundary of the convex hull in the universal setting will appear in \cite{merlin}.

\subsection{Cone singularities and manifolds with particles}
The last part of this paper will briefly survey results on Anti-de Sitter manifolds with spacelike surfaces of finite type, namely homeomorphic to the complement of a finite number of punctures in a closed surface. Depending on the geometry near the removed points, different geometric structures can arise. For instance MGH Anti-de Sitter manifolds with \emph{particles} can be defined, where a particle is a cone singularity of timelike type. It is required in the definition that the manifold contains a locally convex Cauchy surface orthogonal to the singular locus. Hence the first fundamental form of such a Cauchy surface has a cone point in correspondence of each intersection with a particle. Many of the results we mention here are the counterpart ``with particles'' of the results which have been described in Section \ref{section:closed} for the closed case. Hence we will omit most of the details here.

If one fixes the number of cone points (say $n$) and the cone angles $\theta_1,\ldots,\theta_n$ at each cone point, which is assumed to be smaller than $\pi$, it was proved in \cite[Theorem 1.4]{MR2507219} that the deformation space of MGH Anti-de Sitter manifolds 
with a Cauchy surface homeomorphic to $\Sigma$ and particles of cone angles $\theta_1,\ldots,\theta_n$ is homeomorphic to the product of two copies of the Teichm\"uller space of $\Sigma$ with $n$ cone points of angles $\theta_1,\ldots,\theta_n$. An earthquake theorem for hyperbolic surfaces with cone points of angle less than $\pi$ has then been proved, see \cite[Theorem 1.2]{MR2507219}. The prescription of measured geodesic laminations on the boundary of the convex core, as an analogue of Theorem \ref{thm bon schl prescribe lam}, has been established in \cite{MR2913100}.

The existence and uniqueness of a maximal surface orthogonal to the singular locus was proved in \cite[Theorem 1.4]{MR3494175}, thus obtaining as a consequence the existence and uniqueness of a minimal Lagrangian map between two hyperbolic surfaces with the same cone angles (less than $\pi$) in a given isotopy class (\cite[Theorem 1.3]{MR3494175}). Moreover, together with \cite[Theorem 5.11]{KS07} which is a version ``with particles'' of Theorem \ref{thm KS}, one obtains a parameterization of the deformation space of MGH AdS manifolds with particles by means of the cotangent bundle of the Teichm\"uller space of $\Sigma$ with cone angles $\theta_1,\ldots,\theta_n<\pi$, which is also identified to the bundle of holomorphic quadratic differentials with at most simple poles at the punctures.

The existence and uniqueness of the maximal surface orthogonal to the singular locus has then been improved in \cite[Theorem 1.1]{MR3978545} to the existence of a foliation by constant mean curvature surfaces. The proof actually relies on the results of \cite{chenschl}, namely the existence of a foliation by surfaces of constant Gaussian curvature of each connected component of the convex core complement (\cite[Theorem 1.1]{chenschl}). These results have of course applications for the existence of $\theta$-landslides between hyperbolic surfaces with cone angles, see \cite[Theorem 5.8]{chenschl}. Many of these results  had been conjectured in \cite[Section 6.2]{open}, see Questions 6.2-6.5. 

The general study of cone singularities besides the case of particles, including the possibility of intersections between singularities (``collisions''), and introducing the notion of global hyperbolicity in this setting together with examples, has been pursued in \cite{MR2842974} and \cite{MR3192047}.

\subsection{Boundary components and multi-black holes}

Teichm\"uller theory for hyperbolic surfaces with boundary components or cusps is instead intimately related to the geometry of Anti-de Sitter manifolds with multi-black holes. Here we only sketch the definition, and we refer to \cite{MR2443264} and \cite{MR2764871} for more details.

Let us now assume that $\rho_l,\rho_r:\pi_1\Sigma\to\PSL(2,\R)$ are the holonomy representations of  complete hyperbolic structures on $\Sigma$ with cusps and geodesic boundary components. Then there is a maximal domain $\Omega$ in $\AdSP{2,1}$ on which the action of $(\rho_l,\rho_r)$ is free and properly discontinuous, which is however not globally hyperbolic. The domain $\Omega$ can actually be described as the union of all globally hyperbolic domains on which the action is free and properly discontinuous, and each of these can be obtained by the following construction. 

The limit set of the action of $\rho_l(\pi_1\Sigma)$ on $\Hyp^2$ can be described as the complement of a family of open arcs in $\partial\Hyp^2$, where the endpoints of each removed arc are the endpoints of a lift to $\Hyp^2$ of a geodesic boundary component. The limit set for $\rho_r(\pi_1\Sigma)$ has an analogue description, and similarly to the closed case, one can find a circle homeomorphism $\varphi:\partial\Hyp^2\to\partial\Hyp^2$ which is equivariant with respect to the actions of $\rho_l$ and $\rho_r$. This is however not uniquely determined when geodesic boundary components are present. In fact, there are many possible choices of the equivariant map $\varphi$, and the freedom of such choice corresponds to the definition of $\varphi$ on each open arc in $\partial\Hyp^2$ in the complement of the limit set. 

The choice of some particular equivariant map $\varphi$ gives a globally hyperbolic domain, namely the invisible domain of the graph of $\varphi$, on which the action is free and properly discontinuous. The union of all such domains provides the maximal domain $\Omega$, whose quotient is a maximal Anti-de Sitter manifold with multi-black holes.

Using the pleated surface construction in this setting, in \cite{MR2764871} an earthquake theorem was proved for surfaces with boundary, namely given two hyperbolic structures on $\Sigma$ with geodesic boundary, there exist $2^k$ left (or right) earthquake maps, where $k$ is the number of boundary components of $\Sigma$. The $2^k$ choices correspond to the choice, for every boundary component of $\Sigma$, to the sense in which the earthquake lamination ``spirals'' around the boundary; in terms of Anti-de Sitter geometry, this is the choice of a future or past sawtooth in $\Quno$. An extension of this result to crowned hyperbolic surfaces was given in \cite{rosmondi}.

The PhD thesis \cite{rosmondi} contains a result of prescription of two filling  measured geodesic laminations on a hyperbolic surface with boundary, as the bending laminations on the boundary components of the convex core of an Anti-de Sitter manifold with multi-black holes. Finally, \cite{tambregular} contains a study of the maximal surfaces which appear in this case, and of the associated minimal Lagrangian diffeomorphisms, in terms of holomorphic quadratic differentials with poles of order at most 2, hence extending the parameterization we discussed in Section \ref{subsec cotangent} in the closed case.

\setcounter{secnumdepth}{3}
\setcounter{tocdepth}{2}

\bibliographystyle{alpha}
\bibliographystyle{ieeetr}
\bibliography{survey_biblio.bib}

\end{document}